\documentclass{amsart}
\usepackage{amscd,amsmath,amssymb,amsfonts,bbm}
\usepackage{amsthm}
\usepackage[utf8]{inputenc}
\usepackage[T1]{fontenc}
\usepackage{lmodern}
\usepackage[unicode,pdfborder={0 0 0},final]{hyperref}
\usepackage{mathtools}
\usepackage{enumerate}
\usepackage[all,cmtip]{xy}
\usepackage{footmisc}
{\setbox0\hbox{$ $}}\fontdimen16\textfont2=\fontdimen17\textfont2

\newtheorem{thm}{Theorem}[section]
\newtheorem*{thm*}{Theorem}
\newtheorem{prop}[thm]{Proposition}

\newtheorem{lem}[thm]{Lemma}
\newtheorem{cor}[thm]{Corollary}

\theoremstyle{definition}
\newtheorem{Def}[thm]{Definition}

\theoremstyle{remark}
\newtheorem{rem}[thm]{Remark}
\newtheorem{rems}[thm]{Remarks}

\numberwithin{equation}{section}

\newcommand{\isoto}{\myxrightarrow{\,\sim\,}}
\newcommand{\isofrom}{\myxleftarrow{\,\sim\,}}
\makeatletter
\def\myrightarrow{{\setbox\z@\hbox{$\rightarrow$}\dimen0\ht\z@\multiply\dimen0 6\divide\dimen0 10\ht\z@\dimen0\box\z@}}
\def\myrightarrowfill@{\arrowfill@\relbar\relbar\myrightarrow}
\newcommand{\myxrightarrow}[2][]{\ext@arrow 0359\myrightarrowfill@{#1}{#2}}
\def\myleftarrow{{\setbox\z@\hbox{$\leftarrow$}\dimen0\ht\z@\multiply\dimen0 6\divide\dimen0 10\ht\z@\dimen0\box\z@}}
\def\myleftarrowfill@{\arrowfill@\myleftarrow\relbar\relbar}
\newcommand{\myxleftarrow}[2][]{\ext@arrow 3095\myleftarrowfill@{#1}{#2}}
\makeatother

\def\bmu{\boldsymbol\mu}

\newcommand{\GG}{\mathrm{G}}
\newcommand{\K}{\mathrm{K}}
\newcommand{\SK}{\mathrm{SK}}

\newcommand{\bC}{\mathbf{C}}
\newcommand{\Q}{\mathbf{Q}}
\newcommand{\R}{\mathbf{R}}
\newcommand{\Ql}{\mathbf{Q}_{\ell}}
\newcommand{\Zl}{\mathbf{Z}_{\ell}}
\newcommand{\sK}{\mathcal{K}}

\newcommand{\sO}{\mathcal{O}}
\newcommand{\sL}{\mathcal{L}}
\newcommand{\sM}{\mathcal{M}}

\newcommand{\sN}{\mathcal{N}}
\newcommand{\sQ}{\mathcal{Q}}
\newcommand{\sS}{\mathcal{S}}
\newcommand{\sE}{\mathcal{E}}
\newcommand{\sF}{\mathcal{F}}

\newcommand{\wX}{\widetilde{X}}

\newcommand{\et}{\mathrm{\acute{e}t}}
\newcommand{\fppf}{\mathrm{fppf}}

\newcommand{\Ab}{\mathrm{Ab}}

\newcommand{\Zar}{\mathrm{Zar}}

\newcommand{\alg}{\mathrm{alg}}

 \DeclareMathOperator{\Spec}{Spec}
 \DeclareMathOperator{\Coker}{Coker} 
 \DeclareMathOperator{\Gr}{Gr}

\DeclareMathOperator{\Ker}{Ker}
\DeclareMathOperator{\Pic}{Pic}

\DeclareMathOperator{\bPic}{\mathbf{Pic}}

\DeclareMathOperator{\Sym}{Sym}

\newcommand{\Id}{\mathrm{Id}}
\newcommand{\Sch}{\mathrm{Sch}}

\DeclareMathOperator{\Aut}{Aut}

\DeclareMathOperator{\Proj}{Proj}
\newcommand{\bA}{\mathbf{A}}

\newcommand{\bP}{\mathbf{P}}

\DeclareMathOperator{\Hom}{Hom}
\DeclareMathOperator{\NS}{NS}
\DeclareMathOperator{\Alb}{Alb}
\DeclareMathOperator{\Gal}{Gal}
\DeclareMathOperator{\Fitt}{Fitt}
\newcommand{\Z}{\mathbf{Z}}

\newcommand{\Res}{\mathrm{Res}}
\DeclareMathOperator{\CH}{CH}
\DeclareMathOperator{\bCH}{\mathbf{CH}}

\newcommand{\torsion}{\mathrm{torsion}}

\newcommand{\cl}{\mathrm{cl}}
\newcommand{\red}{\mathrm{red}}
\newcommand{\rk}{\mathrm{rk}}

\newcommand{\RR}{{\mathrm{R}}}
\newcommand{\LL}{{\mathrm{L}}}

\newcommand{\p}{{\mathrm{p}}}
\newcommand{\op}{{\mathrm{op}}}
\newcommand{\ok}{\overline{k}}

\newcommand{\Gm}{\mathbf{G}_\mathrm{m}}
\newcommand{\subrat}{{\mathrm{rat}}}
\newcommand{\subirrat}{{\mathrm{irrat}}}

\usepackage{color}

\advance\textheight -1.4mm
\advance\topmargin .7mm

\begin{document}

\date{December 11th, 2019; revised on July 2nd, 2021; incorporated an erratum into the proof
of Theorem~\ref{thuniratcub}
on January 23rd, 2025}
\title[Intermediate Jacobians and rationality over arbitrary fields]{Intermediate Jacobians and rationality\\over arbitrary fields}

\author{Olivier Benoist}
\address{D\'epartement de math\'ematiques et applications, \'Ecole normale sup\'erieure et CNRS, 45~rue d'Ulm, 75230 Paris Cedex 05, France}
\email{olivier.benoist@ens.fr}

\author{Olivier Wittenberg}
\address{Institut Galil\'ee, Universit\'e Sorbonne Paris Nord, 99~avenue Jean-Baptiste Cl\'ement, 93430 Villetaneuse, France}
\email{wittenberg@math.univ-paris13.fr}

\renewcommand{\abstractname}{Abstract}
\begin{abstract}
We prove that a three-dimensional smooth complete intersection of two quadrics over a field $k$ is $k$-rational if and only if it contains a line defined over $k$. To do so, we develop a theory of intermediate Jacobians for geometrically rational threefolds over arbitrary, not necessarily perfect, fields. As a consequence, we obtain the first examples of smooth projective varieties over a field $k$ which have a $k$-point, and are rational over a purely inseparable field extension of $k$, but not over $k$.
\end{abstract}

\maketitle

\section*{Introduction}
{\renewcommand*{\thethm}{\Alph{thm}}

Let $k$ be a field. A variety $X$ of dimension $n$ over $k$ is said to be $k$-\textit{rational} (resp.\ $k$-\textit{unirational}, resp.\ \textit{separably} $k$-\textit{unirational}) if there exists a birational map (resp.\ a dominant rational map, resp.\ a dominant and separable rational map) $\bA^n_k\dashrightarrow X$.

This article is devoted to studying the $k$-rationality of threefolds over~$k$. Our main result answers positively a conjecture of Kuznetsov and
Prokhorov.

\begin{thm}[Theorem \ref{thrat}]
\label{main}
Let $X\subset\bP^5_k$ be a smooth complete intersection of two quadrics. Then $X$ is $k$-rational if and only if it contains a line defined over $k$.
\end{thm}

The question of the validity of Theorem~\ref{main} goes back
to Auel, Bernardara and Bolognesi \cite[Question~5.3.2~(3)]{abb}, who raised
it when~$k$ is a rational function field in one variable over an algebraically closed field.

Using the fact that varieties $X$ as in Theorem \ref{main} are separably $k$-unirational if and only if they have a $k$-point (see Theorem \ref{thunirat}), we obtain new counterexamples to the L\"uroth problem over non-closed fields.

\begin{thm}[Theorem \ref{thmC((t))}]
\label{mainB}
For any algebraically closed field $\kappa$, there exists a three-dimensional smooth complete intersection of two quadrics $X\subset\bP^5_{\kappa(\!(t)\!)}$ which is separably $\kappa(\!(t)\!)$\nobreakdash-unirational, $\kappa(\!(t^{\frac{1}{2}})\!)$-rational, but not $\kappa(\!(t)\!)$-rational.
\end{thm}
When $\kappa$ has characteristic $2$, Theorem \ref{mainB} yields the first examples of smooth projective varieties over a field $k$ which have a $k$-point and are rational over the perfect closure of $k$, but which are not $k$-rational (see Remarks~\ref{remsLuroth}~(iii) and~(iv)).

\vspace{1em}

Theorem \ref{main} may be compared to the classical fact that a smooth quadric over~$k$ is $k$-rational if and only if it has a $k$-point. However, although it is easy to check that a smooth projective $k$-rational variety has a $k$-point, the fact that a $k$-rational three-dimensional smooth complete intersection of two quadrics $X$ necessarily contains a $k$-line is highly non-trivial.
To prove it, we rely on obstructions to the $k$\nobreakdash-rationality of $X$ arising from a study of its intermediate Jacobian.

Such obstructions go back to the seminal work of Clemens and Griffiths \cite{CG}: if a smooth projective threefold over $\bC$ is $\bC$\nobreakdash-rational, then its intermediate Jacobian is isomorphic, as a principally polarized abelian variety over $\bC$, to the Jacobian of a (not necessarily connected) smooth projective curve.
This implication was used in \cite{CG} to show that smooth cubic threefolds over $\bC$ are never $\bC$-rational,
and was later applied to show the irrationality of several other classes of complex threefolds (see for instance \cite{Beauville}).
The work of Clemens and Griffiths was extended by Murre
\cite{Murrecubic} to algebraically closed fields of any
characteristic different from~$2$.

More recently, we implemented the arguments of Clemens and Griffiths over arbitrary perfect
fields $k$ \cite{CGBW}. By exploiting the fact that the intermediate Jacobian may be isomorphic to
the Jacobian of a smooth projective curve over $\ok$ while not being so
over $k$, we produced new examples of varieties over~$k$ that are
$\ok$-rational but not $k$-rational.

Hassett and Tschinkel \cite{HT1} subsequently noticed that over a
non-closed field~$k$, the intermediate Jacobian carries further
obstructions to $k$\nobreakdash-rationality:
if~$X$ is a smooth projective $k$\nobreakdash-rational threefold,
then not only is its intermediate
Jacobian
isomorphic to the Jacobian $\bPic^0(C)$ of a smooth projective
curve $C$ over~$k$, but in addition, assuming for simplicity that~$C$ is geometrically connected of genus~$\geq 2$,
the $\bPic^0(C)$\nobreakdash-torsors associated
with $\Aut(\ok/k)$\nobreakdash-invariant algebraic equivalence classes of
codimension $2$ cycles on $X$ are of the form $\bPic^{\alpha}(C)$ for
some $\alpha\in\Z$.
When~$X$ is a smooth
three-dimensional complete intersection of two quadrics, they used these
obstructions in combination with the natural identification of the variety
of lines of~$X$ with a torsor under the intermediate Jacobian of $X$,
and with the work of Wang~\cite{Wang},
to
prove Theorem~\ref{main} when $k=\R$ \cite[Theorem 36]{HT1} (and later
\cite{HT2} over subfields of~$\bC$).

The aim of the present article is to extend these arguments to arbitrary
fields.

Applications to $k$-rationality criteria for other classes of $\ok$-rational threefolds
appear in the work of Kuznetsov and Prokhorov \cite{KP}.

\vspace{1em}

So far, we have been imprecise about what we call the intermediate Jacobian of a smooth projective threefold $X$ over $k$.

If $k=\bC$, one can use Griffiths' intermediate Jacobian $J^3X$ constructed by transcendental means. This is the original path taken by Clemens and Griffiths \cite{CG}. The algebraic part of Griffiths' intermediate Jacobian has been shown to descend to subfields $k\subseteq \bC$ by Achter, Casalaina-Martin and Vial \cite[Theorem B]{ACMV1}; the resulting $k$-structure on $J^3X$ is the one used in \cite{HT2}.

Over algebraically closed fields~$k$ of arbitrary characteristic, a different construction of an intermediate Jacobian $\Ab^2(X)$, based on codimension $2$ cycles, was provided by Murre \cite[Theorem A p.\,226]{Murre} (see also \cite{KahnMurre}). This cycle-theoretic approach to intermediate Jacobians had already been applied by him to rationality problems (see \cite{Murrecubic}).

Over a perfect field $k$, the universal property satisfied by Murre's intermediate Jacobian $\Ab^2(X_{\ok})$ induces a Galois descent datum on $\Ab^2(X_{\ok})$, thus yielding a $k$\nobreakdash-form $\Ab^2(X)$ of $\Ab^2(X_{\ok})$ \cite[Theorem 4.4]{ACMV1}. It is this intermediate Jacobian $\Ab^2(X)$, which coincides
with~$J^3X$ when~$k\subseteq \bC$, that we used in \cite{CGBW}.

Over an imperfect field $k$, one runs into the difficulty that Murre's definition of $\Ab^2(X_{\ok})$ does not give rise to a $\ok/k$-descent datum on $\Ab^2(X_{\ok})$. 
Achter, Casalaina-Martin and Vial still prove, in \cite{ACMVfunctorial}, the existence of an
algebraic representative $\Ab^2(X)$ for algebraically trivial codimension $2$ cycles on~$X$
(see \textsection1.2 of \emph{op.\ cit.}\ for the definition). However, when~$k$ is imperfect,
it is not known whether $\Ab^2(X)_{\ok}$ is isomorphic to $\Ab^2(X_{\ok})$. For this reason, we do not know how to construct on $\Ab^2(X)$ the principal polarization that is so crucial to the Clemens--Griffiths method.
 
To overcome this difficulty and prove Theorem \ref{main} in full generality, we provide, over an arbitrary field $k$, an entirely new construction of an intermediate Jacobian.

\vspace{1em}

Our point of view is inspired by Grothendieck's definition of the Picard scheme (for which see \cite{FGA}, \cite[Chapter~8]{BLR}, \cite{FGAPic}). With any smooth projective $\ok$-rational threefold $X$ over~$k$, we associate a functor $\CH^2_{X/k,\fppf}:(\Sch/k)^{\op}\to(\Ab)$ endowed with a natural bijection $\CH^2(X_{\ok})\isoto\CH^2_{X/k,\fppf}(\ok)$ (see Definition \ref{defch2xk} and (\ref{okpointsbis})).  The functor $\CH^2_{X/k,\fppf}$ is an analogue, for codimension $2$ cycles, of the Picard functor $\Pic_{X/k,\fppf}$.

Too naive attempts to define the functor $\CH^2_{X/k,\fppf}$ on the category of $k$\nobreakdash-schemes,
such as the formula ``$T\mapsto \CH^2(X_T)$'', fail as Chow groups of possibly singular schemes are not even contravariant with respect to arbitrary morphisms: one would need to use a contravariant variant of Chow groups (see Remark \ref{remvariantech2xk}~(ii)).
To solve this issue, we view Chow groups of codimension~$\leq 2$ as subquotients of $K$-theory by means of the Chern character, and we define $\CH^2_{X/k,\fppf}$ as an appropriate subquotient of (the fppf sheafification of)
the functor $T\mapsto \K_0(X_T)$. That this procedure gives rise to the correct functor, even integrally, is a consequence of the Riemann--Roch theorem without denominators \cite{RRsansden}.

We show that $\CH^2_{X/k,\fppf}$ is represented by a smooth $k$-group scheme $\bCH^2_{X/k}$ (Theorem~\ref{threp}~(i)). Our functorial approach is crucial for this,
as it allows us to argue by fppf descent from a possibly inseparable finite extension $l$ of $k$ such that $X$ is $l$-rational. By construction, there is a natural isomorphism $\bCH^2_{X_l/l}\simeq(\bCH^2_{X/k})_l$ for all field extensions $l$ of $k$.

The $k$-group scheme that we use as a substitute for the intermediate Jacobian of $X$ is then the identity component~$(\bCH^2_{X/k})^0$ of $\bCH^2_{X/k}$, which is an abelian variety (Theorem~\ref{threp}~(ii)).
 We hope that this functorial perspective on intermediate Jacobians may have other applications (to intermediate Jacobians in families, to deformations of algebraic cycles).

 Establishing an identification $(\bCH^2_{X/k})^0_{\ok}\simeq \Ab^2(X_{\ok})$ (Theorem \ref{threp}~(vi)) and using the principal polarization on $\Ab^2(X_{\ok})$ constructed in \cite{CGBW}, we endow $(\bCH^2_{X/k})^0$ with a canonical principal polarization, which paves the way for applications to rationality questions.
Let us now state the most general obstruction to the $k$-rationality of a smooth projective threefold that we obtain by analyzing~$\bCH^2_{X/k}$.

\begin{thm}[Theorem \ref{threp} (vii)]
\label{main2}
Let $X$ be a smooth projective $k$-rational threefold over $k$. Then there exists a smooth projective curve $B$ over $k$ such that the $k$-group scheme $\bCH^2_{X/k}$ can be realized as a direct factor of $\bPic_{B/k}$ in a way that respects the canonical principal polarizations.
\end{thm}

In Theorem~\ref{listobstructions}, we deduce from Theorem \ref{main2} more concrete obstructions to the $k$\nobreakdash-rationality of $X$, pertaining to the N\'eron--Severi group $\NS^2(X_{\ok})$ of algebraic equivalence classes of codimension $2$ cycles on $X_{\ok}$, to the principally polarized abelian variety $(\bCH^2_{X/k})^0$, and to the $(\bCH^2_{X/k})^0$-torsors that are of the form $(\bCH^2_{X/k})^{\alpha}$ for some $\alpha\in \big(\mkern-2.5mu\bCH^2_{X/k}/(\bCH^2_{X/k})^0\big)(k)=\NS^2(X_{\ok})^{\Aut(\ok/k)}$.

The principle of the proof of Theorem \ref{main2} goes back to Clemens and Griffiths. Since $X$ is $k$-rational,
it can be obtained from $\bP^3_k$ by a composition of blow-ups of regular curves and of closed points, followed by a contraction. The curve $B$ whose existence is predicted by Theorem \ref{main2} is roughly the union of the blown-up curves. This works perfectly well if $k$ is perfect. If $k$ is imperfect, however, some of the blown-up curves may be regular but not smooth over $k$. It is nevertheless very important, in view of the application to Theorem \ref{main}, that the curve $B$ appearing in the statement of Theorem~\ref{main2} be smooth over $k$. To prove Theorem \ref{main2} as stated, one thus has to show that the contributions of these non-smooth regular curves are cancelled out by the final contraction. This non-trivial fact relies on a complete understanding of which Jacobians of proper reduced curves over $k$ split as the product of an affine group scheme and of an abelian variety over~$k$ (Theorem~\ref{thmsplit}). 

\vspace{1em}

The text is organized as follows. Sections \ref{sec1} and \ref{sec2} gather preliminaries, concerning respectively group schemes and $K$-theory. Section \ref{sec2} contains in particular the definition of the above-mentioned functor $\CH^2_{X/k,\fppf}$ associated with a smooth projective $\ok$-rational threefold $X$ over $k$ (Definition \ref{defch2xk}). In Section \ref{sec3}, we prove that this functor is representable and study the $k$-group scheme $\bCH^2_{X/k}$ that represents it (Theorem~\ref{threp}).
A number of obstructions to the $k$-rationality of~$X$ are then derived (Theorem \ref{threp} (vii) and Theorem~\ref{listobstructions}).
Section \ref{sec4} is devoted to applications to three-dimensional smooth complete intersections of two quadrics: we compute
$\bCH^2_{X/k}$ entirely (Theorem \ref{thintjac}), deduce the irrationality criterion that is our main theorem (Theorem \ref{thrat}), and apply the criterion to examples over Laurent series fields (Theorem \ref{thmC((t))}).

\subsection*{Acknowledgements}

Over a perfect field, the results of this article were obtained during the
``Rationality of Algebraic Varieties'' conference held on Schiermonnikoog
Island in April 2019 after reading the arguments of Hassett and
Tschinkel~\cite{HT1} over the reals.  We are grateful to Jean-Louis
Colliot-Thélène and to Sasha Kuznetsov for encouraging us to write them up.
Finally, we thank Asher Auel, Marcello Bernardara and Michele Bolognesi for their comments and their interest,
and the referees for their work and their helpful suggestions.

\subsection*{Notation and conventions}

We fix a field $k$ and an algebraic closure $\ok$ of~$k$. Let~$k_\p$ be the 
perfect closure of $k$ in $\ok$ and $\Gamma_k:=\Gal(\ok/k_{\p})=\Aut(\ok/k)$ be the absolute Galois group of $k_{\p}$. 
 If $X$ and $T$ are two $k$-schemes, we set $X_T:=X\times_k T$.
A~\textit{variety} over~$k$ is a separated scheme of finite type over $k$;
a~\emph{curve} is a variety of pure dimension~$1$.
If $G$ is a group scheme locally of finite type over $k$, we denote by $G^0$ its identity component.
If $X$ is a smooth proper variety over $k$, we let $\CH^c(X_{\ok})_{\alg}\subseteq\CH^c(X_{\ok})$ be the subgroup
of algebraically trivial codimension~$c$ cycle classes and define $\NS^c(X_{\ok}):=\CH^c(X_{\ok})/\CH^c(X_{\ok})_{\alg}$.

We use \textit{qcqs} as a shorthand for quasi-compact and quasi-separated.
We denote by $(\Sch/ k)$ 
the category of qcqs 
$k$-schemes
and by $(\Ab)$ the category of abelian groups. 
If $X$ is a commutative $k$-group scheme, we let $\Phi_X:(\Sch/ k)^{\op}\to(\Ab)$ be the functor given by $\Phi_X(T)=\Hom_k(T,X)$. The functor $\Phi:X\mapsto \Phi_X$, from the category of commutative $k$-group schemes to the category of
functors $(\Sch/ k)^{\op}\to(\Ab)$, is fully faithful, by Yoneda's lemma and since all schemes are covered by affine (hence qcqs) open subschemes. 
We say that a functor $(\Sch/ k)^{\op}\to(\Ab)$ is \emph{representable} if it is isomorphic to~$\Phi_X$ for some commutative $k$\nobreakdash-group scheme~$X$, which need not be qcqs.
The functor $\Z:(\Sch/k)^{\op}\to (\Ab)$ sending $T\in (\Sch/k)$ to the group $\Z(T)$ of locally constant maps $T\to \Z$ is represented by the constant $k$-group scheme $\Z$.

 We refer to \cite[VII, D\'efinition 1.4]{SGA6} for the definition of a \textit{regular immersion} of schemes that we use, based on the Koszul complex.
A closed immersion of regular schemes is always regular. 
A morphism of schemes $f:X\to Y$ is said to be \textit{a local complete intersection} or \textit{lci} if
it can be factored, locally on $X$, as the composition of a regular immersion and of a smooth morphism \cite[VIII, D\'efinition~1.1]{SGA6}. 

If $\ell$ is a prime number and $M$ is a $\Z$-module, we let $M\{\ell\}\subseteq M$ be the $\ell$-primary torsion subgroup of $M$, and $T_{\ell}(M):=\Hom(\Ql/\Zl,M)$ and $V_{\ell}(M):=T_{\ell}(M)[1/\ell]$ be the $\ell$-adic Tate modules of $M$. If $G$ is a commutative $k$-group scheme, we set $T_{\ell}G:=T_{\ell}(G(\ok))$ and $V_{\ell}G:=V_{\ell}(G(\ok))$.
}

\section{Group schemes}
\label{sec1}

We first collect miscellaneous information concerning group schemes. The main new result of this section is Theorem \ref{thmsplit}.

\subsection{Chevalley's theorem}
If~$G$ is a connected smooth group scheme over~$k$, there
is a unique short exact sequence
\begin{equation}
\label{Chevalley}
0\to L(G_{k_\p})\to G_{k_\p}\to A(G_{k_\p})\to 0
\end{equation}
of group schemes over $k_\p$, where $L(G_{k_\p})$ is smooth, connected and affine and where $A(G_{k_\p})$ is an abelian variety.
This statement was proved by Chevalley~\cite{Chevalley},
see \cite[Theorem 1.1]{Conrad} for a modern proof.

\subsection{Principal polarizations}
\label{pp}

Recall that a \textit{principal polarization} on an abelian variety $A$ over $k_{\p}$ is an ample class $\theta\in \NS(A_{\ok})^{\Gamma_k}$ whose associated isogeny $A_{\ok}\to\hat{A}_{\ok}$ is an isomorphism, that a principally polarized abelian variety~$A$ over $k_\p$ is a product of indecomposable principally polarized abelian varieties over~$k_\p$ and that the factors of this decomposition are unique as subvarieties of $A$ (see \cite[\S 2.1]{CGBW}).

We define a \textit{polarization} (resp.\ a \textit{principal polarization})
of a smooth commutative group scheme $G$ over~$k$
to be a polarization (resp.\ a principal polarization) of the abelian variety $A(G^0_{k_\p})$ over~$k_\p$. If $G$ is endowed with a polarization $\theta$, we say that a smooth
subgroup scheme $H\subseteq G$ is a \textit{polarized direct factor} of $G$ if there exists a subgroup scheme $H'\subseteq G$ such that the canonical morphism $\iota:H\times H'\isoto G$ is an isomorphism
and such that
$\iota^*\theta=\pi^*\theta|_{A(H^0_{k_\p})}+\pi'^*\theta|_{A(H'^0_{k_\p})}$,
where~$\pi$ and~$\pi'$ denote the projections of $A(H^0_{k_\p})\times A(H'^0_{k_\p})$
onto $A(H^0_{k_\p})$ and $A(H'^0_{k_\p})$.
If in addition~$\theta$ is a principal polarization, then
so are $\pi^*\theta|_{A(H^0_{k_\p})}$ and $\pi'^*\theta|_{A(H'^0_{k_\p})}$;
in this case, we also speak of a \textit{principally polarized direct factor}.

\subsection{Locally constant functions}
\label{locconstpar}

For a variety $X$ over $k$, consider the functor
\begin{alignat*}{4}
\Z_{X/k}:\hspace{1em}(\Sch&/k)^{\op}&\hspace{1em}\to\hspace{1em}&(\Ab)\\
&T&\hspace{1em}\mapsto\hspace{1em}&\Z(X_T)\rlap{.}
\end{alignat*}
Equivalently $\Z_{X/k}$ is the push-forward of the constant sheaf~$\Z$ by
the structural morphism $X \to\Spec(k)$ (and hence is an fpqc sheaf).
Let $\pi_0(X/k)$ denote the \'etale $k$-scheme of connected components of $X$, defined in \cite[\S I.4, D\'efinition~6.6]{DemaGa}.
The Weil restriction of scalars $\Res_{\pi_0(X/k)/k}\Z$ exists
as a $k$\nobreakdash-scheme by \cite[7.6/4]{BLR}.

\begin{prop}
\label{locconsrep}
Let $X$ be a variety over $k$. Then the functor $\Z_{X/k}$ is canonically represented by the 
Weil restriction of scalars $\Res_{\pi_0(X/k)/k}\Z$.
\end{prop}

\begin{proof}
Recall that there is a canonical faithfully flat
morphism $q_X:X \to \pi_0(X/k)$ whose fibres are geometrically connected
 \cite[\S I.4, Propositions~6.5 and~6.7]{DemaGa}.
For any $T\in(\Sch/k)$,
the resulting morphism $X_T \to \pi_0(X/k)_T$ is
surjective, has connected fibres, and is open \cite[Th\'eor\`eme~2.4.6]{EGA42};
therefore~$q_X$ induces a bijection between the sets of connected components
of~$X_T$ and of~$\pi_0(X/k)_T$.
As a consequence, the homomorphism 
$\Z(\pi_0(X/k)_T) \to \Z(X_T)$
is an isomorphism,
and hence so is the morphism of functors
$\Res_{\pi_0(X/k)/k}\Z \to \Z_{X/k}$
that it induces when~$T$ varies.
\end{proof}

\begin{rems}
\label{isoloc}
(i) 
We will usually denote by $\Z_{X/k}$, rather than by $\Res_{\pi_0(X/k)/k}\Z$, the group scheme that $\Z_{X/k}$ represents.

(ii) Proposition \ref{locconsrep} shows that a morphism $p:X'\to X$ of varieties over $k$ that induces a bijection between the sets of connected components of $X_{\ok}$ and of $X'_{\ok}$ gives rise to an isomorphism $\Z_{X/k}\isoto\Z_{X'/k}$.
\end{rems}

\subsection{Picard schemes}
\label{Pic}

The \textit{absolute Picard functor} of a proper variety $X$ over $k$~is
\begin{alignat*}{4}
\Pic_{X/k}:\hspace{1em}(\Sch&/k)^{\op}&\hspace{1em}\to\hspace{1em}&(\Ab)\\
&T&\hspace{1em}\mapsto\hspace{1em}&\Pic(X_T).
\end{alignat*}
Beware that our notation differs slightly from that of \cite[\S 9.2]{FGAPic}.

If $\tau\in\{\Zar,\et,\fppf\}$, we denote the sheafification of $\Pic_{X/k}$ for the corresponding (Zariski, \'etale or fppf) topology by $\Pic_{X/k,\tau}$.
The functors $\Pic_{X/k,\et}$ and $\Pic_{X/k,\fppf}$ are equal \cite[8.1 p.\,203]{BLR}
and are represented by a group scheme locally of finite type over~$k$
\cite[8.2/3]{BLR} which we denote by $\bPic_{X/k}$: the \textit{Picard scheme} of~$X$.
These two functors contain $\Pic_{X/k,\Zar}$
and $T\mapsto\Pic(X_T)/\Pic(T)$ as subfunctors if $H^0(X,\sO_X)=k$, and they coincide with them
if in addition $X(k)\neq\varnothing$
(see \cite[Theorem~9.2.5]{FGAPic} and \cite[Proposition~7.8.6]{EGA32}), for instance if $X$ is connected and reduced and~$k=\ok$.

\subsubsection{Picard schemes of blow-ups}
\label{parpic}

In \S\ref{parpic}, we consider the following situation.
We fix a regular closed immersion $i:Y\to X$ of qcqs schemes of pure codimension $c\geq 2$, we let $p:X'\to X$ be the blow-up of $X$ along $Y$ with exceptional divisor~$Y'$, and we let $p':Y'\to Y$ and $i':Y'\to X'$ be the  natural morphisms. The morphism $p':Y'\to Y$ is a projective bundle of relative dimension $c-1\geq 1$
(see \cite[\S 1.2]{Thomason}).
In this setting,  we study the group morphism
\begin{alignat}{4}
\label{formulaPicblup}
\begin{aligned}
\Pic(X&)\times\Z(Y)&\hspace{1em}\to\hspace{1em}&\hspace{1.5em}\Pic(X')\\
&(\sL,\psi)&\hspace{1em}\mapsto\hspace{1em}& p^*\sL\otimes\sO_{X'} \Big(-\sum_{n\in\Z}n\mkern-1mu\Big[p^{-1}(\psi^{-1}(n))\Big]\Big).
\end{aligned}
\end{alignat}

\begin{prop}
\label{propblup}
Under the hypotheses of \S\ref{parpic}, the map (\ref{formulaPicblup}) is bijective.
\end{prop}

\begin{proof}
By absolute noetherian approximation \cite[Theorem C.9]{TT} and the limit arguments of \cite[\S8]{EGA43}, we may assume that $X$ is noetherian.
If $\sN\in \Pic(X')$, the function $\psi:Y\to\Z$ such that $\sN|_{X'_y}\simeq\sO_{X'_y}(\psi(y))$ for all $y\in Y$ is locally constant.
(Indeed, for $n\gg 0$, the $\sO_Y$\nobreakdash-module $p'_*((\sN|_{Y'})(n))$ is locally free and its formation commutes with base change, by \cite[III, Theorems 8.8 and 12.11]{Hartshorne}.)
Since $\sO_{X'}(-Y')|_{X'_y}\simeq\sO_{X'_y}(1)$ for all $y\in Y$ (see \cite[\S 1.2]{Thomason}), it follows that $\sN\otimes\sO_{X'}(\sum_{n\in\Z}n[p^{-1}(\psi^{-1}(n))])$ is trivial on the fibers of~$p$, and that $\psi$ is the unique function with this property.
It remains to show that $p^*:\Pic(X)\to\Pic(X')$ is injective with image the subgroup of isomorphism classes of line bundles that are trivial on the fibers of~$p$. The injectivity follows from \cite[Lemme 2.3 (a)]{Thomason}, and the description of the image from Lemma \ref{lemlbbl}~(iii) below.
\end{proof}

\begin{lem}
\label{lemlbbl}
Under the hypotheses of \S\ref{parpic}, assume that $X$ is noetherian and let $\sN$ be a line bundle on $X'$ such that $\sN|_{X'_y}\simeq \sO_{X'_y}$ for all $y\in Y$. For any integer~$n$, set $\sN(n)=\sN \otimes \sO_{X'}(-nY')$. Then:
\begin{enumerate}[(i)]
\item For all $j\geq 1$ and $n\geq 0$, the sheaf $R^jp_*(\sN(n))$ vanishes.
\item For all $n\geq 0$, the natural morphism $p^*p_*(\sN(n))\to\sN(n)$ is surjective.
\item The sheaf $p_*\sN$ is invertible and $p^*p_*\sN\to\sN$ is an isomorphism.
\end{enumerate}
\end{lem}

\begin{proof}
By \cite[III, Theorem 8.8 (c)]{Hartshorne},
assertion (i) holds for $n\gg 0$. To prove (i) for all $n\geq 0$ by descending induction on $n$, we consider for $j\geq 1$ the  exact sequence
$$R^jp_*\sN(n)\to R^jp_*\sN(n-1)\to R^jp_*(\sN(n-1)|_{Y'})$$
and note that $R^jp'_*(\sN(n-1)|_{Y'})=0$ for $n\geq 1$ by cohomology and base change \cite[III, Theorem 12.11]{Hartshorne}.

Assertion (ii) holds for all $n\gg 0$ by \cite[III, Theorem 8.8 (a)]{Hartshorne}.
To prove~(ii) for all $n\geq 0$ by descending induction on $n$, we consider the natural commutative diagram with exact rows
\begin{equation*}
\begin{aligned}
\xymatrix
@R=0.4cm
{
&p^*p_*\sN(n)\ar^{}[r]\ar[d]^{} &p^*p_*\sN(n-1)\ar^{}[r]\ar[d]^{} & p^*p_*(\sN(n-1)|_{Y'})\ar[d]^{}  \ar^{}[r]&0\\
0\ar[r]&\sN(n)\ar^{}[r]&\sN(n-1) \ar^{}[r]&\sN(n-1)|_{Y'}\ar^{}[r]&0,
}
\end{aligned}
\end{equation*}
in which the exactness of the upper row follows from the vanishing of $R^1p_*(\sN(n))$ proved in (i), and note that $p^*p_*(\sN(n-1)|_{Y'})\to\sN(n-1)|_{Y'}$ is surjective for $n\geq 1$ in view of Nakayama's lemma,
since its restriction to the fibers of~$p$ is surjective by cohomology and base change \cite[III, Theorem 12.11]{Hartshorne}.

To prove (iii), we work Zariski-locally around a point $y\in X$.  In view of (ii) for $n=0$, we can assume after shrinking $X$ the existence of a section $\sigma\in H^0(X',\sN)$ that does not vanish identically on $X'_y$. Since $\sN|_{X'_y}\simeq \sO_{X'_y}$ and $X'_y$ is a projective space, the section $\sigma$ vanishes nowhere on $X'_y$ and, after shrinking $X$ again, it induces an isomorphism $\sigma:\sO_{X'}\isoto\sN$. Assertion (iii) now follows from the fact that the natural morphism $\sO_{X}\to p_*\sO_{X'}$ is an isomorphism \cite[VII, Lemme 3.5]{SGA6}.
\end{proof}

\begin{cor}
\label{corblup}
Under the hypotheses of \S\ref{parpic}, if $X$ is moreover a proper variety over~$k$, the formula (\ref{formulaPicblup}) induces an isomorphism of functors
\begin{equation}
\label{decbl}
\Pic_{X/k} \times\mkern2.5mu \Z_{Y/k}\isoto\Pic_{X'/k}\rlap{.}
\end{equation}
\end{cor}

\begin{proof}
Since the formation of the blow-up of a regular closed immersion commutes with flat base change (see \cite[VII, Propositions 1.5 et 1.8 i)]{SGA6}), we can apply Proposition~\ref{propblup} to the morphisms $i_T:Y_T\to X_T$ for $T\in (\Sch/k)$.
\end{proof}

\subsubsection{Picard schemes of curves}
\label{Piccurves}

If $C$ is a proper curve over $k$, then $\bPic_{C/k}$ is smooth over $k$ by \cite[8.4/2]{BLR}. Moreover, letting $D:=\widetilde{C_{k_{\p}}^{\red}}$ be the normalization of the reduction of $C_{k_{\p}}$, which is a smooth proper curve over $k_{\p}$, the 
pull-back morphism $\bPic^0_{C_{k_{\p}}/k_{\p}}\to \bPic^0_{D/k_{\p}}$ induces an isomorphism 
\begin{equation}
\label{Chevanorm}
A(\bPic^0_{C_{k_{\p}}/k_{\p}})\isoto\bPic^0_{D/k_{\p}},
\end{equation}
as \hbox{\cite[9.2/11]{BLR}} shows. The principal polarization of $\bPic^0_{D/k_{\p}}$ given by the theta divisor thus induces a canonical principal polarization on $\bPic_{C/k}$ in the sense of~\S\ref{pp}. If $C$ is irreducible, then so is $D$ \cite[Proposition~2.4.5]{EGA42}
and the principally polarized abelian variety $\bPic^0_{D/k_{\p}}$ over $k_{\p}$ is thus indecomposable if it is non-zero (see \cite[\S 2.1]{CGBW}).

In preparation for the proof of Proposition~\ref{NS} and, later, for the statement of Corollary~\ref{corsplit},
let us recall
that if~$V$ is a variety over~$k$ and~$V'$ denotes the normalization of~$V^{\red}$,
then~$V$ is \emph{geometrically unibranch} if and only if the natural morphism
$V' \to V$ is a universal homeomorphism
\cite[Proposition~2.4.5, (6.15.1), end of (6.15.3)]{EGA42}.
In particular, normal varieties are geometrically unibranch.
We also recall that the property of being geometrically unibranch is invariant under extension of scalars
\cite[Proposition~6.15.7]{EGA42}.

\begin{prop}
\label{NS}
Let $C$ be a proper curve over $k$ and let $C':=\widetilde{C^{\red}}$ be the normalization of its reduction. Then there is a short exact sequence
\begin{align}
\label{se:pic0picz}
0\to \bPic^0_{C/k}\to\bPic_{C/k}\to \Z_{C'/k}\to 0\rlap{.}
\end{align}
\end{prop}

\begin{proof}
Both $\Z_{C'/k}$ and $\bPic_{C/k}/\bPic^0_{C/k}$ are \'etale group schemes over $k$. They are thus isomorphic if and only if so are their base changes $G:=\Z_{C'_{k_{\p}}/k_{\p}}$ and $H:=\bPic_{C_{k_{\p}}/k_{\p}}/\bPic^0_{C_{k_{\p}}/k_{\p}}$ to $k_{\p}$.
Letting $D:=\widetilde{C_{k_{\p}}^{\red}}$ be the normalization of the reduction of~$C_{k_{\p}}$, which is a smooth proper curve over $k_{\p}$, one has $G=\Z_{D/k_{\p}}$ by Remark~\ref{isoloc}~(ii)
(indeed
the map $D\to C'_{k_{\p}}$ is a universal homeomorphism
since $C'_{k_{\p}}$ is geometrically unibranch),
and $H=\bPic_{D/k_{\p}}/\bPic^0_{D/k_{\p}}$ by \cite[9.2/11]{BLR}.
That $G\simeq H$ now follows from the fact that $G(\ok)$ and $H(\ok)$ are both isomorphic, as $\Gamma_k$-modules, to $\Z(D_{\ok})$.
\end{proof}

\subsection{When do Jacobians split?}

We now provide, in Theorem~\ref{thmsplit}, a criterion for the Jacobian of a proper reduced curve
to be the product of an abelian variety and of an affine group scheme.
We will use Theorem~\ref{thmsplit}
in Lemma~\ref{inducedmapzero}, which plays a key role in the proof of Theorem~\ref{main2}.

\subsubsection{Statement}

Let us introduce some notation.
Whenever~$D$ is a smooth proper integral curve over~$k$,
the \emph{genus} of~$D$ is the dimension of the abelian variety $\bPic^0_{D/k}$.
We note that~$D$ has genus~$0$ if and only if the irreducible components of~$D_{\ok}$ (which are all isomorphic)
are rational.
Given a proper reduced curve~$C$ over~$k$, we denote by $C_{\subrat}$ (resp.~$C_{\subirrat}$)
the union of those irreducible
components~$B$ of~$C$ such that the normalization of $(B_{k_{\p}})^{\red}$ has genus~$0$
(resp.\ has genus~$\geq 1$).  We  view~$C_{\subrat}$ and~$C_{\subirrat}$ as reduced closed subschemes
of~$C$.
We define a \emph{strict cycle of components of~$C_{\ok}$}
to be a sequence
of pairwise distinct irreducible components
 $B_1,\dots,B_n$ of~$C_{\ok}$ for some integer $n\geq 2$, such that there exist
pairwise distinct points $x_1,\dots,x_n$ of~$C_{\ok}$ with $x_i \in B_i \cap B_{i+1}$
for all $i\in \{1,\dots,n-1\}$ and $x_n \in B_n \cap B_1$.
Finally, we recall that a reduced curve over~$\ok$ is \emph{seminormal} if it is \'etale locally isomorphic
to the union of the coordinate axes in an affine space over~$\ok$ \cite[Chapter~I, 7.2.2]{Kollarbook}.

\begin{thm}
\label{thmsplit}
Let $C$ be a proper reduced curve over $k$. The group
scheme $\bPic^0_{C/k}$ is the product of an abelian variety and of an affine group scheme over~$k$
if and only if the following conditions all hold:
\begin{enumerate}[(i)]
\item the scheme $(C_\subirrat)_{\ok}$ is reduced and seminormal and its irreducible components are smooth;
\item any strict cycle of components of~$C_{\ok}$ is contained in~$(C_{\subrat})_{\ok}$;
\item for every connected component~$B$ of~$C_{\subrat}$,
either the scheme $B \cap C_{\subirrat}$ is \'etale over~$k$,
or it is of the form $\Spec(k')$ for some field~$k'$
and the restriction map
$H^0(B,\sO_B) \to k'$ is bijective.
\end{enumerate}
In this case, the natural map $\bPic_{C/k} \to \bPic_{C_{\subrat}/k} \times \bPic_{C_{\subirrat}/k}$
is an isomorphism
and
 $\bPic^0_{C_{\subrat}/k}$ is affine while
$\bPic^0_{C_{\subirrat}/k}$
is an abelian variety.
\end{thm}

Condition~(iii) holds if $(C_{\subrat} \cap C_{\subirrat})_{\ok}$ is reduced,
and it implies, in turn, that
$C_{\subrat} \cap C_{\subirrat}$ is reduced.  The reverse implications are true if~$k$ is perfect.

When~$C$ is integral and geometrically locally irreducible,
for instance when~$C$ is integral and geometrically unibranch,
Theorem~\ref{thmsplit}
 takes on a particularly simple form,
which we now state.
We recall that normal varieties are geometrically unibranch
(see \textsection\ref{Piccurves} for more reminders on this property).
In the sequel, we shall only apply Theorem~\ref{thmsplit}
to normal curves,
through Corollary~\ref{corsplit}.

\begin{cor}
\label{corsplit}
Let $C$ be a proper integral curve over~$k$.
Assume that the connected components of $C_{\ok}$ are irreducible;
such is the case, for instance, if~$C$ is geometrically unibranch.
Then the group
scheme $\bPic^0_{C/k}$ is the product of an abelian variety and of an affine group scheme over~$k$
if and only if at least one of the following two conditions holds:
\begin{enumerate}[(i)]
\item $C$ is smooth over $k$
(in which case $\bPic^0_{C/k}$ is an abelian variety);
\item the normalization $D$ of $(C_{k_{\p}})^{\red}$ has genus $0$
(in which case  $\bPic^0_{C/k}$ is affine).
\end{enumerate}
\end{cor}

\begin{proof}
Our assumptions imply that $C_{\subrat}=\varnothing$ or $C_{\subirrat}=\varnothing$,
that there is no strict cycle of components of~$C_{\ok}$
and that~$C$ is smooth over~$k$ if and only if~$C_{\ok}$ is reduced with smooth irreducible components.
Thus, the conditions of Theorem~\ref{thmsplit} are all met
if and only if at least one of~(i) and~(ii) holds.
\end{proof}

\begin{rem}
Corollary~\ref{corsplit} applies to integral curves that may not be geometrically reduced. It would however fail in general for irreducible but non-reduced curves, as the following example shows. Let $E$ be an elliptic curve over $\ok$, and let $\sL$ be an ample line bundle on $E$. Define $C:=\Proj_{E}(\sO_E\oplus\sL)$, where sections of~$\sL$ square to~$0$. The natural closed immersion $i:E=C^{\red}\to C$ then induces an isomorphism $i^*:\bPic^0_{C/k}\isoto \bPic^0_{E/k}\simeq E$. Indeed, in view of (\ref{Chevanorm}), it suffices to show that the kernel of $(i^*)_{k_{\p}}$ has trivial tangent space at the identity, which follows from the fact that the pull-back $i^*:H^1(C,\sO_C)\isoto H^1(E,\sO_E)$ is an isomorphism.
We note that~$C$ is geometrically unibranch since~$C^{\red}$ is normal.
\end{rem}

\subsubsection{A few general lemmas}

We establish a series of lemmas on which the proof of Theorem~\ref{thmsplit} will rely.
The first one is due to Tanaka, see \cite[Lemma~3.3]{Tanaka}.

\begin{lem}
\label{lem:tanaka}
Let~$F$ be a field extension of~$k$.
If~$F/k$ is not separable,
there exist finite purely inseparable field extensions $k \subseteq k' \subseteq k''$
and a $k'$\nobreakdash-linear embedding $k'' \hookrightarrow F \otimes_k k'$
such that $F \otimes_k k'$ is a field and $k''\neq k'$.
\end{lem}

\begin{proof}
Let~$k''$ be a minimal finite purely inseparable field extension of~$k$ such that $F \otimes_k k''$ fails to be reduced.
Let~$p$ denote the characteristic of~$k$.
As $k''/k$ is finite, purely inseparable and non-trivial,
there exists a subextension $k'/k$ such that $[k'':k']=p$.
Fix $x \in k'$ such that $k''=k'(x^{1/p})$.
By the minimality of~$k''$, the finite connected non-zero $F$\nobreakdash-algebra $F'=F \otimes_k k'$ is reduced,
hence is a field.
On the other hand, as $F \otimes_k k'' = F'[t]/(t^p-x)$ is not reduced,
we see that $x$ must be a $p$\nobreakdash-th power in~$F'$, so that $k''$ embeds $k'$\nobreakdash-linearly
into~$F'$.
\end{proof}

\begin{lem}
\label{lem:quotientchevalley}
Let $f:G \to G'$ be a surjective morphism between connected smooth group schemes over~$k$
such that the kernel of $A(f_{k_{\p}}):A(G_{k_{\p}}) \to A(G'_{k_{\p}})$ is smooth and connected.
If~$G$ is isomorphic to the product of an abelian variety and of an affine group scheme over~$k$,
then so is~$G'$.
\end{lem}

\begin{proof}
Suppose $G=L\times A$ with~$L$ affine and~$A$ an abelian variety.
Let $K=\Ker(f)$.
Let $p:G\to A$ be the projection, let $i:A\to G$ be the inclusion, 
and let $p(K)$ denote the scheme-theoretic image of~$K$ by~$p$;
it is a subgroup scheme of~$A$ \cite[VI$_{\mathrm{A}}$, Proposition~6.4]{SGA31r}.
We note that $A(G_{k_{\p}})=A_{k_{\p}}$
and that $p(K)_{k_{\p}} \subseteq \Ker(A(f_{k_{\p}}))$.
The morphism $\Ker(A(f_{k_{\p}})) \to G'_{k_{\p}}$
induced by~$f \circ i$
factors through $L(G'_{k_{\p}})$ and hence vanishes
since $\Ker(A(f_{k_{\p}}))$ is an abelian variety
and $L(G'_{k_{\p}})$ is affine.
Therefore the morphism $p(K) \to G'$ induced by~$f\circ i$ also vanishes.
We deduce that $i(p(K)) \subseteq K$, hence $K=(K \cap L) \times i(p(K))$.
On the other hand, the closed immersion $G/K \to G'$ induced by~$f$ is an isomorphism as it
is surjective and~$G'$ is reduced.
It follows that $G'=(L/(K\cap L)) \times (A/p(K))$, and the lemma is proved.
\end{proof}

\begin{lem}
\label{lem:picaddcomponent}
Let~$C$ be a proper reduced curve over~$k$ and~$B$ be a geometrically connected reduced closed subscheme
of~$C$ of pure dimension~$1$.  View the union of the irreducible components of~$C$ not contained in~$B$
as a reduced closed subscheme~$B'$ of~$C$.
If the natural morphism $B \cap B' \to \Spec\mkern1mu(H^0(B',\sO_{B'}))$ is \'etale
and if any strict cycle of components of~$C_{\ok}$ is contained in~$B'_{\ok}$,
then the natural morphism
 $\bPic_{C/k}\to \bPic_{B/k} \times \bPic_{B'/k}$
is an isomorphism.
\end{lem}

\begin{proof}
If $i:B_T \hookrightarrow C_T$, $i':B'_T\hookrightarrow C_T$, $i'':B_T\cap B'_T \hookrightarrow C_T$
denote the inclusions, the short exact sequence $1 \to \Gm \to i_*\Gm \times i'_*\Gm \to i''_*\Gm \to 1$
of sheaves for the Zariski topology on~$C_T$ induces,
for any $T\in(\Sch/k)$,
 an exact sequence of groups
\begin{multline}
\label{se:mayervietoris}
1 \to \Gm(C_T) \to \Gm(B_T) \times \Gm(B'_T) \to \Gm(B_T \cap B'_T)\\ \to \Pic(C_T)
\to \Pic(B_T) \times \Pic(B'_T) \to \Pic(B_T \cap B'_T)\rlap{.}
\end{multline}
The natural morphism $B \cap B' \to \Spec\mkern1mu(H^0(B',\sO_{B'}))$
is an \'etale morphism between finite $k$\nobreakdash-schemes
that
induces an injection on $\ok$\nobreakdash-points,
in view of the assumption about strict cycles of components.
(Note that $\Spec\mkern1mu(H^0(B',\sO_{B'}))(\ok)$ is the set of connected components of~$B'_{\ok}$.)
It is therefore an open and closed immersion \cite[I, Th\'eor\`eme~5.1]{SGA1}.
In particular, the
resulting restriction map
\begin{align*}
\Gm\big(\Spec\mkern1mu(H^0(B',\sO_{B'})) \times_k T\big) \to \Gm\big(B_T \cap B'_T\big)
\end{align*}
is onto, and hence so is
the restriction map $\Gm(B'_T) \to \Gm(B_T \cap B'_T)$.
Noting that $\bPic_{B\cap B'/k}=0$,
the conclusion of the lemma now results from
the exact sequence of fppf sheaves
obtained by
sheafifying~\eqref{se:mayervietoris}
with respect to~$T$.
\end{proof}

\begin{lem}
\label{lem:k0=k}
Let~$k'$ be a finite purely inseparable extension of~$k$.
Let~$C$ be a proper integral curve over~$k'$. We view~$C$ as a curve over~$k$.
If  $\bPic^0_{C/k}$ is the product
 of an abelian variety and of an affine group scheme over~$k$,
then $k'=k$
or $C=C_{\subrat}$.
\end{lem}

\begin{proof}
For all $T\in (\Sch/k)$, pull-back induces an equivalence between the categories of \'etale $T$\nobreakdash-schemes and of \'etale $T_{k'}$\nobreakdash-schemes \cite[IX, Th\'eor\`eme~4.10]{SGA1}.
 It follows at once that the natural morphism
$$\Res_{k'/k}(\Pic_{C/k'})\to \Res_{k'/k}(\Pic_{C/k',\et})$$
of functors $(\Sch/k)^{\op}\to(\Ab)$ becomes an isomorphism after \'etale sheafification.
On the other hand,
by the very definition
of the absolute Picard functor (see~\textsection\ref{Pic})
and of Weil restriction of scalars,
one has
$\Pic_{C/k}=\Res_{k'/k}(\Pic_{C/k'})$.
By these two remarks, we get, after sheafification, an isomorphism $\bPic_{C/k}= \Res_{k'/k}(\bPic_{C/k'})$, which restricts to an isomorphism $\bPic^0_{C/k}= \Res_{k'/k}(\bPic^0_{C/k'})$ in view of \cite[Proposition~A.5.9]{CGP}.
Let us write  $\bPic^0_{C/k}=L\times A$ with~$L$ affine and~$A$ an abelian variety.
Applying \cite[XVII, Appendice~III, Proposition 5.1]{SGA32} with $U=L$
(or \cite[Proposition~A.7.8]{CGP})
now shows that if $k'\neq k$, then $\bPic^0_{C/k}=L$, so that $A((\bPic^0_{C/k})_{k_{\p}})=0$
and $C=C_{\subrat}$.
\end{proof}

\begin{lem}
\label{lem:D to Pic factors through C}
Let~$D$ be the disjoint union of smooth proper geometrically integral curves $D_1,\dots,D_n$ over~$k$.
Let~$C$ be a proper curve over~$k$ such that $H^0(C,\sO_C)=k$
and $C(k)\neq\varnothing$.
Let $\nu:D\to C$ be a morphism.
If $\nu^*:\bPic_{C/k} \to \bPic_{D/k}$
admits a section,
then there exist morphisms $\rho_i:C \to \bPic_{D_i/k}$ for $i \in \{1,\dots,n\}$
 such that
 $\rho_i\circ\nu|_{D_i}$ is the canonical morphism for all~$i$
while $\rho_i\circ\nu|_{D_j}$ is constant for all $j\neq i$.
\end{lem}

\begin{proof}
We let $\nu_i:D_i\to C$ be the restriction of~$\nu$,
fix
a section $\sigma:\bPic_{D/k}\to \bPic_{C/k}$ of~$\nu^*$
and,
noting that $\bPic_{D/k}=\prod_{i=1}^n \bPic_{D_i/k}$,
 let $\sigma_i:\bPic_{D_i/k}\to \bPic_{C/k}$ be the restriction of~$\sigma$,
so that $\nu_i^* \circ \sigma_j=\delta_{ij}$.
Let $\iota_i:D_i\to \bPic_{D_i/k}$ be the canonical morphism.
Then~$\nu_j^*$ maps $\sigma_i \circ \iota_i \in \bPic_{C/k}(D_i)$
to $\iota_i \in \bPic_{D_i/k}(D_i)$ if $j=i$, to $0 \in \bPic_{D_j/k}(D_i)$ otherwise.
As $H^0(C,\sO_C)=H^0(D_i,\sO_{D_i})=k$
and $C(k)\neq\varnothing$,
there are a canonical bijection
$\Pic(C_T)/\Pic(T) \isoto \bPic_{C/k}(T)$ and a canonical injection
$\Pic((D_i)_T)/\Pic(T) \hookrightarrow \bPic_{D_i/k}(T)$
for all~$T$.
We deduce, for each~$i$, the existence of $\alpha_i \in \Pic(C\times_k D_i)$ such that
$(\nu_j \times 1)^*\alpha_i \in \Pic(D_j \times_k D_i)$ is the class of the diagonal
for $j=i$ and comes by pull-back from $\Pic(D_i)$ for all $j\neq i$.
Switching the factors, the class~$\alpha_i$ gives rise to the desired element
$\rho_i \in \bPic_{D_i/k}(C)$.
\end{proof}

\begin{lem}
\label{Picnorm}
Let $C$ be a proper curve over $k$.  Let $\nu:C'\to C$ be the normalization of $C^{\red}$. The pull-back morphism
$\nu^*:\bPic_{C/k}^0\to \bPic_{C'/k}^0$ is surjective.
\end{lem}

\begin{proof}
We may assume that~$k$ is separably closed.
We then claim that the map $\nu^*:\bPic_{C/k}(k)\to \bPic_{C'/k}(k)$ is onto.
As $\bPic_{C'/k}$ is smooth over~$k$,
it will follow  that the morphism
$\nu^*:\bPic_{C/k}\to \bPic_{C'/k}$
is dominant,
by \cite[2.2/13]{BLR},
so that the morphism~$\nu^*:\bPic_{C/k}^0\to \bPic_{C'/k}^0$
is dominant, by Proposition~\ref{NS}, and hence surjective,
by \cite[VI$_\mathrm{A}$, Corollaire~6.2~(i)]{SGA31r}.
To prove the claim, we may assume that~$C$ is reduced, by \cite[9.2/5]{BLR}.
Letting $C^0 \subseteq C$ be a dense open normal subset,
we then remark that $\nu^*:\Pic(C) \to \Pic(C')$ is onto
since any divisor on~$C'$ is linearly equivalent to a divisor
supported on $\nu^{-1}(C^0)$.
As~$k$ is separably closed, it follows that $\nu^*:\Pic_{C/k,\et}(k)\to \Pic_{C'/k,\et}(k)$ is onto as well, as
desired.
\end{proof}

\subsubsection{Proof of Theorem~\ref{thmsplit}}

As the formation of~$C_{\subrat}$ and~$C_{\subirrat}$ is compatible
with separable extensions of scalars and as the assertions of the last sentence
of the theorem
are of a geometric nature, we may, and will henceforth, assume that~$k$ is separably closed.
For later use, we note that thanks to this assumption, if condition~(i)
of Theorem~\ref{thmsplit} holds, then
the irreducible components of~$C_{\subirrat}$, viewed as reduced schemes,
are geometrically reduced (being both reduced and generically geometrically reduced)
and therefore smooth over~$k$, and
the non-smooth locus of~$C_{\subirrat}$ over~$k$ consists of $k$\nobreakdash-points
(as the intersection of any two irreducible components of~$C_{\subirrat}$ is \'etale).

\medskip
\emph{Step 1}: we assume that (i)--(iii) hold and deduce the remaining assertions.

From~\eqref{Chevanorm} applied to $C_{\subrat}$,
we deduce that
 $(\bPic^0_{C_{\subrat}/k})_{k_{\p}}$ is affine;
hence $\bPic^0_{C_{\subrat}/k}$ is also affine.
To prove that
 $\bPic^0_{C_{\subirrat}/k}$ is an abelian variety
and that
the natural map $\bPic_{C/k} \to \bPic_{C_{\subrat}/k} \times \bPic_{C_{\subirrat}/k}$
is an isomorphism,
we argue by induction on the number of irreducible components of~$C_{\subirrat}$.
When $C_{\subirrat}=\varnothing$, there is nothing to prove.
Otherwise,
let us choose an irreducible component~$B$ of $C_{\subirrat}$
and denote by~$B'$ (resp.~$B'_{\subirrat}$) the union of the irreducible components of~$C$
(resp.~of $C_{\subirrat}$) that are distinct from~$B$.
We view~$B$, $B'$ and $B'_{\subirrat}$ as reduced closed subschemes of~$C$.
The equalities 
 $C = B \cup B' = C_{\subrat} \cup C_{\subirrat} = B \cup C_{\subrat} \cup B'_{\subirrat}$
and $B' = C_{\subrat} \cup B'_{\subirrat}$ induce a commutative square
\begin{align}
\begin{aligned}
\label{eq:commsquare induction pic}
\xymatrix@R=3ex{
\bPic_{C/k} \ar[d] \ar[r] & \bPic_{B/k} \times \bPic_{B'/k} \ar[d]^\wr \\
\bPic_{C_{\subrat}/k} \times \bPic_{C_{\subirrat}/k} \ar[r] & \bPic_{B/k} \times \bPic_{C_{\subrat}/k} \times \bPic_{B'_{\subirrat}/k}
}
\end{aligned}
\end{align}
whose right vertical arrow is an isomorphism by the induction hypothesis.
To conclude the proof, we need only verify that
Lemma~\ref{lem:picaddcomponent}
can be applied, on the one hand, to~$C$, $B$ and~$B'$,
and on the other hand, to $C_{\subirrat}$, $B$ and $B'_{\subirrat}$.
Indeed,  it will  follow that the horizontal
arrows of~\eqref{eq:commsquare induction pic} are isomorphisms, which 
implies,
on the one hand, that the left vertical arrow
of this square is an isomorphism,
and on the other hand, that  $\bPic^0_{C_{\subirrat}/k}$ is an abelian variety,
since so are $\bPic^0_{B'_{\subirrat}/k}$ (by the induction hypothesis)
and $\bPic^0_{B/k}$ (as~$B$ is smooth over~$k$ \cite[9.2/3]{BLR}).

To this end, letting~$F$ stand either for~$B'$ or for~$B'_{\subirrat}$, we fix $x \in B \cap F$,
and denote by~$E$ the connected component of~$x$ in~$F$
and by $C_1,\dots,C_m$ the irreducible components of~$C_{\subirrat}$ containing~$x$,
numbered so that $B=C_1$.
The finite $k$\nobreakdash-algebra $H^0(E,\sO_E)$, being non-zero, connected and reduced, is a field;
it embeds into
the residue field~$k(x)$ of~$x$.
We now have to prove that
the natural morphism $B \cap E \to \Spec\mkern1mu(H^0(E,\sO_{E}))$ is \'etale at~$x$.

If $k(x)\neq k$, then $m=1$ and therefore
$C_{\subrat}\cap C_{\subirrat}$ is not \'etale over~$k$ at~$x$.
We deduce from~(iii)
that~$E$ coincides with the connected component of~$x$ in~$C_{\subrat}$,
that $B \cap E$ is reduced at~$x$ and that $H^0(E,\sO_E)=k(x)$;
hence the desired result.

If $k(x)=k$, then $H^0(E,\sO_E)=k$ and it suffices to see
that $B \cap E$ is reduced at~$x$.
In the Zariski tangent space $T_xC$ of~$C$ at~$x$,
we have $T_xE=V + T_xC_2 + \dots + T_xC_m$
where $V=T_xC_{\subrat}$ if $x \in C_{\subrat}$ and $F=B'$,
and $V=0$ otherwise.
It follows from~(iii) that $C_{\subrat}\cap C_{\subirrat}$ is reduced,
so that $V \cap (T_xC_1+\dots+T_xC_m)=0$,
and from~(i) that the lines $T_xC_1,\dots,T_xC_m$
are in direct sum.  Hence $T_x(B\cap E)=T_xC_1 \cap T_xE = 0$, and $B \cap E$ is indeed reduced at~$x$.
Step~1 is complete.

\medskip
We now assume that $\bPic^0_{C/k}$
is the product of an abelian variety and of an affine group scheme over~$k$, and prove (i)--(iii) in four more steps.
By Lemma~\ref{lem:quotientchevalley}, we may assume that~$C$ is connected.
In addition, we may assume that $C_{\subirrat}\neq\varnothing$.
Let $C_1,\dots,C_n$ be the irreducible components of $C_{\subirrat}$, viewed as reduced schemes.
Let~$D_i$ be the normalization of~$C_i$,
let~$D$ be the disjoint union of the~$D_i$
and let $\nu_i:D_i\to C$
and $\nu:D\to C$ be the natural morphisms.

\medskip
\emph{Step 2}: we prove that $C_1,\dots,C_n$ and $C_{\subirrat}$ are geometrically reduced over~$k$.

As $C_{\subirrat}=C_1 \cup \dots \cup C_n$, it suffices to check that the~$C_i$
are geometrically reduced.
Assume that some $B \in \{C_1,\dots,C_n\}$
is not geometrically reduced.
By Lemma~\ref{lem:tanaka} applied to~$k(B)$,
there exist subfields $k \subseteq k' \subseteq k'' \subseteq k_{\p}$,
with $k''\neq k'$ and $k''/k$ finite,
such that $B_{k'}$ is integral and such that
if $B'$ denotes its normalization, the natural morphism
$B' \to \Spec(k')$ factors through $\Spec(k'')$.
Let $\omega:C'\to C_{k'}$ be the normalization of $(C_{k'})^{\red}$.
By Lemma~\ref{Picnorm},
the pull-back map $\omega^*:\bPic^0_{C_{k'}/k'} \to \bPic^0_{C'/k'}$
is surjective,
and~\eqref{Chevanorm} shows that
$A(\omega^*_{k_{\p}})$ is an isomorphism.
As~$B'$ is a connected component of~$C'$,
we deduce, thanks to Lemma~\ref{lem:quotientchevalley}, that $\bPic^0_{B'/k'}$
is the product of an abelian variety and of an affine group scheme over~$k'$.
Lemma~\ref{lem:k0=k} implies that $k''=k'$ or $B'=B'_{\subrat}$, a contradiction.

\medskip
\emph{Step 3}: assuming that~$D$ is smooth over~$k$, we construct
$\rho_i:C \to \bPic^1_{D_i/k}$
 such that
 $\rho_i\circ\nu_i$ is a closed immersion 
while $\rho_i(C_{\subrat} \cup \bigcup_{j\neq i} C_j)$ is finite,
for all~$i$.

It suffices to check that the hypotheses of Lemma~\ref{lem:D to Pic factors through C} are satisfied.
Indeed,
the canonical morphism $D_i \to \bPic^1_{D_i/k}$ is a closed immersion if~$D_i$ is a smooth
proper integral curve of genus~$\geq 1$ over~$k$  \cite[Propositions 6.1 and 2.3]{MilneJac},
and $\rho_i(C_{\subrat})$ is finite
since any morphism from a rational curve
to an abelian variety is constant.

We recall that $C_{\subirrat}\neq\varnothing$. As~$C$ is proper, connected and reduced, the restriction map
$H^0(C,\sO_C) \to H^0(C_1,\sO_{C_1})$ has to be injective;
as~$C_1$ is geometrically integral, we deduce that $H^0(C,\sO_C)=k$.
In addition, as~$k$ is separably closed,
we have $C_1(k)\neq\varnothing$
\cite[2.2/13]{BLR}, hence $C(k)\neq\varnothing$.

As~$D$ is smooth,
the group scheme $\bPic^0_{D/k}$ is an abelian variety
and the morphism $\nu^*:\bPic^0_{C/k} \to \bPic^0_{D/k}$
induces an isomorphism $A((\bPic^0_{C/k})_{k_{\p}}) \isoto (\bPic^0_{D/k})_{k_{\p}}$
(see~\eqref{Chevanorm}).
Our assumption on
$\bPic^0_{C/k}$
therefore implies that
 $\nu^*:\bPic^0_{C/k} \to \bPic^0_{D/k}$ admits a section.
Letting~$C'$ be the normalization of $C^{\red}$,
the natural map $\nu^*:\Z_{C'/k} \to \Z_{D/k}$ also admits a section.
In addition, the sequence~\eqref{se:pic0picz}
splits as~$k$ is separably closed and $\bPic^0_{C/k}$ is smooth,
and the choice of a splitting of~\eqref{se:pic0picz} and of
a section of $\nu^*:\Z_{C'/k} \to \Z_{D/k}$
determines a splitting of the sequence~\eqref{se:pic0picz}
associated with~$D$.
Applying Proposition~\ref{NS} to~$C$ and to~$D$,
we now conclude that $\nu^*:\bPic_{C/k} \to \bPic_{D/k}$ admits a section.
This completes Step~3.

\medskip
\emph{Step 4}: we prove~(i) and~(ii) of Theorem~\ref{thmsplit}.

Step~2 ensures that
$((C_{\ok})^{\red})_{\subirrat}=(C_{\subirrat})_{\ok}$
and it follows from Lemma~\ref{lem:quotientchevalley}, in view of \cite[9.2/5]{BLR},
that $\bPic^0_{(C_{\ok})^{\red}/\ok}$ is the product of an abelian variety and
of an affine group scheme over~$\ok$.
Thus,
after replacing~$C$ with $(C_{\ok})^{\red}$,
we may, and will until the end of Step~4, assume that $k=\ok$.
In this case~$D$ is smooth over~$k$ and Step~3 becomes applicable.

To complete the proof of~(i), it suffices to show that
for all $i\in \{1,\dots,n\}$,
the curve~$C_i$ is smooth
and the scheme $C_i \cap \big(\bigcup_{j\neq i} C_j\big)$ is reduced.
We fix~$i$.
As $\rho_i \circ \nu_i$ is an immersion,
so is~$\nu_i$; as $\nu_i(D_i)=C_i$ is a reduced closed subscheme of~$C$,
we deduce that~$\nu_i$ restricts to an isomorphism $D_i \to C_i$,
so that~$C_i$ is smooth.
Now the restriction of~$\rho_i$
to the subscheme
 $C_i \cap \big(\bigcup_{j\neq i} C_j\big)$
is both a closed immersion (since so is $\rho_i|_{C_i}$)
and a morphism whose scheme-theoretic image
is finite and reduced (since so is $\rho_i|_{\bigcup_{j\neq i} C_j}$),
hence this scheme is reduced.

To prove~(ii), we pick
a strict cycle of components $B_1,\dots,B_n$ of~$C$
and pairwise distinct points $x_i \in B_i \cap B_{i+1}$
for $i\in \{1,\dots,n-1\}$ and $x_n \in B_n \cap B_1$.
If one of the~$B_i$ were contained in~$C_{\subirrat}$,
say $B_1=C_1$, then $\rho_1(x_1)$ and~$\rho_1(x_n)$ would have to be distinct,
since $\rho_1|_{C_1}$ is injective, and equal, since $\rho_1(B_2 \cup \dots \cup B_n)$ is
a point (being finite and connected).  This is absurd.

\medskip
\emph{Step 5}: we prove~(iii) of Theorem~\ref{thmsplit}.

We now know that~(i) holds, hence~$D$ is smooth over~$k$:
we can apply Step~3~again.

\begin{lem}
\label{lem:droledelemme}
Let $i\in \{1,\dots,n\}$.
For any purely $1$\nobreakdash-dimensional connected reduced closed subscheme~$E$ of~$C$
such that $E\cap C_i$ is finite and non-empty,
the restriction map
 $H^0(E,\sO_{E}) \to H^0(E \cap C_i, \sO_{E \cap C_i})$
is an isomorphism of fields.
\end{lem}

\begin{proof}
The morphism~$\rho_i|_E$ has finite image (by Step~3),
hence it factors through an affine open $\Spec(R) \subset \bPic^1_{D_i/k}$.
To see that the map of the lemma is surjective,
we note that
its composition
with
$R \to H^0(E,\sO_{E})$
is surjective
as $\rho_i|_{E \cap C_i}$ is a closed immersion.
It is also injective,
 as $H^0(E,\sO_E)$ is a field
(being a
non-zero, connected, reduced, finite $k$\nobreakdash-algebra)
and $E \cap C_i \neq\varnothing$.
\end{proof}

To prove~(iii), let~$B$ be a connected component of~$C_{\subrat}$
such that $B\cap C_{\subirrat}$ is not \'etale over~$k$, say at a point~$x$.

After renumbering, we may assume that the~$C_i$ containing~$x$
are $C_1,\dots,C_m$,
for some $m \in \{1,\dots,n\}$.
If~$x$ were a $k$\nobreakdash-point, the subspace $T_xB \cap T_xC_{\subirrat}$ of the Zariski tangent space $T_xC$ would be non-zero.
After renumbering $C_1,\dots,C_m$ appropriately
and setting $E=B \cup C_2 \cup \dots \cup C_m$,
the vector space $T_xE \cap T_xC_1$
would be non-zero.  The scheme $E \cap C_1$ would then be non-reduced;
this would contradict Lemma~\ref{lem:droledelemme}.
Thus $k(x)\neq k$ and hence $m=1$.

Let~$E'$ be the connected component of~$x$ in $C_{\subrat}\cup C_2\cup \dots\cup C_n$.
As the evaluation map
$H^0(E' \cap C_1, \sO_{E' \cap C_1}) \to k(x)$
is surjective,
 Lemma~\ref{lem:droledelemme} shows
that
$H^0(E',\sO_{E'})=H^0(E' \cap C_1, \sO_{E' \cap C_1})=k(x)$;
hence, by Step~2, there is no $k$\nobreakdash-algebra morphism
$H^0(E',\sO_{E'})\to H^0(C_j,\sO_{C_j})$ for any~$j$.
Therefore $E'=B$ and $B \cap C_{\subirrat}=E' \cap C_1$.
This completes Step~5
as well as the proof of Theorem~\ref{thmsplit}.

\subsection{Murre's intermediate Jacobian}
\label{parMurre}

If $X$ is a smooth projective variety over~$k$, 
Murre defines an \textit{algebraic representative for algebraically trivial codimension $2$ cycles on $X_{\ok}$} to be an abelian variety $\Ab^2(X_{\ok})$ over $\ok$ endowed with a morphism
\begin{equation}
\label{AJk}
\phi^2_X:\CH^2(X_{\overline{k}})_{\alg}\to \Ab^2(X_{\ok})(\overline{k})
\end{equation}
that is initial among regular homomorphisms with values in an abelian variety (see \cite[Definition 1.6.1, \S 1.8]{Murre}). It is obviously unique up to a unique isomorphism, Murre has shown its existence (see \cite[Theorem~A~p.\,226]{Murre} and \cite{KahnMurre}), and Achter, Casalaina-Martin and Vial have shown that it descends uniquely to an abelian variety $\Ab^2(X_{k_\p})$ over $k_\p$ in such a way that (\ref{AJk}) is $\Gamma_k$\nobreakdash-equivariant \cite[Theorem~4.4]{ACMV1}.
If $X$ is a smooth projective $\ok$-rational threefold over $k$, we endow $\Ab^2(X_{k_\p})$ with the principal polarization $\theta_X\in\NS^1(\Ab^2(X_{\ok}))^{\Gamma_k}$ 
constructed in \cite[Property~2.4, Corollary 2.8]{CGBW}. 

\subsection{Representability lemmas}

Here are three lemmas to be used later.

\begin{lem}
\label{repfact}
Let $F$ and $F'$ be two functors $(\Sch/k)^{\op}\to(\Ab)$. If $F\times F'$ is represented by a group scheme locally of finite type over $k$, then so is $F$.
\end{lem}

\begin{proof}
Let $G$ be the group scheme represented by $F\times F'$ and let $\mu:G\to G$ be the morphism induced 
by the endomorphism $(x,y)\mapsto (0,y)$ of $F\times F'$. Then $F$ is represented by $\Ker(\mu)$, which is a group scheme locally of finite type over~$k$.
\end{proof}

\begin{lem}
\label{repquotZ}
Let $G$ be a commutative group scheme locally of finite type over $k$.
 Let $\nu:\Z\to G$ be a morphism of $k$-group schemes such that $\nu(n)\notin G^0(k)$ if $n\neq 0$. Then the cokernel functor $Q:(\Sch/k)^{\op}\to(\Ab)$ defined by $Q(T)=G(T)/\nu(\Z(T))$ is represented by a group scheme locally  of finite type over $k$.
\end{lem}

\begin{proof}
Translation by $\nu(n)$ for $n\in\Z$ induces an action of the group $\Z$ on the set of connected components of $G$. Choose one connected component of $G$ in each orbit of this action. Their disjoint union represents $Q$.
\end{proof}

\begin{lem}
\label{lem:locrep=>rep}
Let $F:(\Sch/k)^{\op}\to(\Ab)$ be a functor and~$k'$ be a finite extension of~$k$.
Let $\tau\in\{\et,\fppf\}$.  Assume that $k'/k$ is separable if $\tau=\et$.
If~$F$ is a sheaf for the topology~$\tau$
and the functor $(\Sch/k')^{\op}\to(\Ab)$ obtained by restricting~$F$
is represented by a group scheme locally of finite type over~$k'$,
then~$F$ is represented by a group scheme locally of finite type over~$k$.
\end{lem}

\begin{proof}
Let $F':(\Sch/k')^{\op}\to(\Ab)$ denote the restriction of~$F$
and~$G'$ the $k$\nobreakdash-group scheme that represents~$F'$.
As~$F'$ is the restriction of~$F$, there is a canonical
descent datum on~$F'$ with respect to $p:\Spec(k')\to \Spec(k)$.
This descent datum
induces a descent datum on~$G'$.
By \cite[Lemma~I.8.6]{MurreContravariant},
the latter is effective, and~$G'$ descends to a group scheme~$G$ locally of finite type over~$k$.
It remains to note that~$F$ and~$\Phi_G:(\Sch/k)^{\op}\to(\Ab)$ are two $\tau$\nobreakdash-sheaves,
that~$p$ is a $\tau$\nobreakdash-covering
and that there is by construction an isomorphism of $\tau$\nobreakdash-sheaves
$p^*F \simeq p^*\Phi_G$ that satisfies the cocycle condition;
from this, it follows that $F\simeq \Phi_G$.
\end{proof}

\section{\texorpdfstring{$K$}{K}-theory functors}
\label{sec2}

We now define and study several functors $(\Sch/k)^{\op}\to(\Ab)$ built from $K$-theory. Our goal, met in \S\ref{CH2def}, is to define, for a smooth projective $\ok$-rational threefold $X$ over $k$, a functor $\CH^2_{X/k,\fppf}:(\Sch/k)^{\op}\to(\Ab)$ that will serve as a substitute for its intermediate Jacobian.

\subsection{\texorpdfstring{$K$}{K}-theory}

We follow the conventions of \cite{SGA6,TT}.

\subsubsection{Definition}
\label{defK0}

 If $X$ is a qcqs scheme, we let $\K_0(X)$ be the Grothendieck group of the triangulated
category of perfect complexes of $\sO_X$-modules.
 (This group is denoted by $\K^{\boldsymbol{\cdot}}(X)$ in \cite[IV, D\'efinition 2.2]{SGA6} and coincides with the one defined in \cite[\S 3.1]{TT}, as indicated in \cite[\S 3.1.1]{TT}.) We endow $\K_0(X)$ with the ring structure induced by the tensor product (\cite[IV, \S 2.7 b)]{SGA6}, \cite[\S 3.15]{TT}). 

If $X$ admits an ample family of line bundles \cite[\S 2.1.1]{TT} (for instance, if $X$ is quasi-projective over an affine scheme \cite[\S 2.1.2 (c)]{TT}),
then $\K_0(X)$ is naturally isomorphic to the Grothendieck group of the exact category of vector bundles on~$X$ (combine \hbox{\cite[Corollary 3.9]{TT}} and \cite[Theorem 1.11.7]{TT}).

\subsubsection{Functoriality}
\label{funct}

 Let $f:X\to Y$ be a morphism of qcqs schemes. 

The derived pull-back $\LL f^*$ of perfect complexes along $f$
 induces a morphism $f^*:\K_0(Y)\to \K_0(X)$ (\hbox{\cite[IV, \S 2.7 b)]{SGA6}}, \cite[\S 3.14]{TT}).

  The derived push-forward $\RR f_*$ induces a morphism $f_*:\K_0(X)\to \K_0(Y)$ if $\RR f_*$ preserves perfect complexes \cite[\S 3.16]{TT}.
This condition is satisfied if $f$ is a proper and perfect morphism
\cite[Proposition 2.1, Example 2.2~(a)]{LipNee},
for instance if $f$ is a proper lci morphism 
(see \cite[VIII, Proposition 1.7]{SGA6}).
In this case, the projection formula \hbox{\cite[IV, (2.12.4)]{SGA6}}
stemming from \cite[III, Proposition~3.7]{SGA6} shows that 
\begin{equation}
\label{projformula}
f_*(x\otimes f^*y)=f_*x\otimes y
\end{equation}
for all $x\in \K_0(X)$ and $y\in \K_0(Y)$.

\subsubsection{Rank and determinant} 
\label{rkdet}

The \textit{rank} $\rk:\K_0(X)\to \Z(X)$ and the \textit{determinant}
$\det:\K_0(X)\to \Pic(X)$ are group homomorphisms that are functorial with
respect to pull-backs and are such that if $x\in \K_0(X)$ is represented by
a bounded complex of vector bundles on $X$, then $\rk(x)$ is the
alternating sum of the ranks of its terms and $\det(x)$ is the alternating
tensor product of the determinants of its terms (see
\cite[Theorem~2~p.\,42]{MumKnu}).

We define $\SK_0(X)$ to be the kernel of $(\rk,\det):\K_0(X)\to \Z(X)\times\Pic(X)$.

\subsubsection{Projective bundles and blow-ups}

If $X$ is a qcqs scheme and $\pi:\bP_X\sE\to X$ is the projective bundle associated with a vector bundle $\sE$ of rank $r$ on $X$, the morphism $\K_0(X)^r\to \K_0(\bP_X\sE)$ given by the formula
\begin{equation}
\label{proj}
(x_0,\dots,x_{r-1})\mapsto \sum_{j=0}^{r-1} \pi^*x_j\otimes [\sO_{\bP_X\sE}(-j)]
\end{equation}
is an isomorphism of $\K_0(X)$-modules (\cite[VI, Th\'eor\`eme 1.1]{SGA6}, \cite[Theorem~4.1]{TT}).

If $i:Y\to X$ is a regular closed immersion of qcqs schemes of pure codimension $c\geq 1$, if $p:X'\to X$ is the blow-up of $X$ along $Y$ with exceptional divisor $Y'$, and if we denote by $p':Y'\to Y$ and $i':Y'\to X'$ the natural morphisms, Thomason has shown that the morphism $\K_0(X)\times \K_0(Y)^{c-1}\to \K_0(X')$ given by the formula
\begin{equation}
\label{blowup}
(x,y_1,\dots,y_{c-1})\mapsto p^*x+\sum_{j=1}^{c-1}i'_*p'^*y_j\otimes[\sO_{X'}(jY')]
\end{equation}
is an isomorphism of $\K_0(X)$-modules \cite[Th\'eor\`eme 2.1]{Thomason}.

\subsubsection{Coherent sheaves}
\label{coherent}

If $X$ is a qcqs scheme, we let $\GG_0(X)$ be the Grothendieck group of the triangulated category of pseudo-coherent complexes of $\sO_X$\nobreakdash-modules with bounded cohomology. (This group is denoted by $K_{\boldsymbol{\cdot}}(X)$ in \cite[IV, D\'efinition~2.2]{SGA6}, see \cite[\S 3.3]{TT}.)

If $X$ is noetherian, the group $\GG_0(X)$ is naturally isomorphic to the Grothendieck group of the abelian category of coherent $\sO_X$-modules \cite[IV, \S 2.4]{SGA6}.
In this case, letting $F_d\GG_0(X)\subseteq \GG_0(X)$ be the subgroup generated by classes of coherent sheaves whose support has dimension $\leq d$ \cite[X, D\'efinition 1.1.1]{SGA6} defines a filtration $F_{\bullet}$ on $\GG_0(X)$. 

If $X$ is noetherian and regular,
the natural morphism $\K_0(X)\to \GG_0(X)$ is an isomorphism \cite[Theorem 3.21]{TT}. 
This allows one to speak of the class of a coherent $\sO_X$\nobreakdash-module in $\K_0(X)$, and of the filtration $F_\bullet \K_0(X)$ of $\K_0(X)$ by dimension of the support.

If $X$ is a smooth variety of pure dimension $n$ over $k$, we use the notation $F^c\K_0(X)=F_{n-c}\K_0(X)$
and $\Gr^c_F\K_0(X)=F^c\K_0(X)/F^{c+1}\K_0(X)$.
According to \cite[Example~15.1.5]{Fulton}, associating with an integral closed subscheme $Z\subseteq X$ of codimension~$c$ the class $[\sO_Z]\in \K_0(X)$ induces a surjective morphism 
\begin{equation}
\label{CHK}
\varphi^c:\CH^c(X)\to \Gr^c_F\K_0(X).
\end{equation}
As explained in \cite[Example~15.3.6]{Fulton}, Jouanolou's Riemann--Roch theorem without denominators \cite{RRsansden} shows that $\varphi^0$, $\varphi^1$ and $\varphi^{2}$ are isomorphisms, with inverses given by the rank $\rk$, the determinant $\det$ and the opposite $-c_2$ of the second Chern class. In particular,
$F^2\K_0(X)=\SK_0(X)$.

\begin{lem}
\label{lemCH2SK0}
Let $f:X\to Y$ be a morphism of smooth equidimensional varieties over $k$. There exists a commutative diagram
\begin{equation}
\label{pullbackCH2SK0}
\begin{aligned}
\xymatrix
@C=1.2cm
@R=0.4cm
{
\CH^2(Y)\ar_{\sim}^{\varphi^2}[r]\ar[d]^(.45){f^*} & \ar[d]^{}  \Gr^2_F\K_0(Y) \\
\CH^2(X)\ar_{\sim}^{\varphi^2}[r]& \Gr^2_F\K_0(X)
}
\end{aligned}
\end{equation}
in which the right vertical arrow is induced by $f^*: \K_0(Y)\to \K_0(X)$.
\end{lem}

\begin{proof}
The pull-back $f^*: \K_0(Y)\to \K_0(X)$ restricts to $f^*: \SK_0(Y)\to \SK_0(X)$, hence induces $f^*:F^2\K_0(Y)\to F^2\K_0(X)$. Since the morphisms $\varphi^2$ are bijective with inverse given by $-c_2$, and since the second Chern class is functorial, we see that $f^*:F^2\K_0(Y)\to F^2\K_0(X)$ induces a morphism $f^*:\Gr^2_F\K_0(Y)\to \Gr^2_F\K_0(X)$ making the diagram (\ref{pullbackCH2SK0}) commute.
\end{proof}

\subsection{The functor \texorpdfstring{$K_{0,X/k}$}{K\_\{0,X/k\}} and its sheafifications}
\label{functK0}

\subsubsection{Definition}

If $X$ is a proper variety over $k$, the \textit{absolute $\K_0$ functor} of $X$ is
\begin{alignat*}{4}
\K_{0,X/k}:\hspace{1em}(\Sch&/k)^{\op}&\hspace{1em}\to\hspace{1em}&(\Ab)\\
&T&\hspace{1em}\mapsto\hspace{1em}&\K_0(X_ T).
\end{alignat*}
The rank and the determinant (see \S\ref{rkdet}) give rise to morphisms of functors $\rk:\K_{0,X/k}\to\Z_{X/k}$ and 
$\det:\K_{0,X/k}\to\Pic_{X/k}$.
We let $\SK_{0,X/k}$
be the kernel of $(\rk,\det):\K_{0,X/k}\to\Z_{X/k}\times\Pic_{X/k}$.
If $\tau\in\{\Zar,\et,\fppf\}$, 
we let $\K_{0,X/k,\tau}$ (resp.\ $\SK_{0,X/k,\tau}$)
be the sheafification of $\K_{0,X/k}$ (resp.\ $\SK_{0,X/k}$)
for the corresponding (Zariski, \'etale, fppf) topology.

\begin{rem}
We do not know whether $\K_{0,X/k,\et}$ and $\K_{0,X/k,\fppf}$ coincide.
\end{rem}

\subsubsection{Functoriality}
\label{functfunct}

Let $f:X\to Y$ be a morphism of proper varieties over $k$. The pull-backs $(f_T)^*$ for $T\in (\Sch/k)$ induce a natural transformation of functors $f^*:\K_{0,Y/k}\to \K_{0,X/k}$. 

Similarly, if $f:X\to Y$ is a perfect (for instance lci) morphism of proper varieties over $k$, the push-forwards 
$(f_T)_*$ for $T\in (\Sch/k)$ (which exist by \S\ref{funct} since $f_T$ is perfect by \cite[III, Corollaire 4.7.2]{SGA6})
induce a natural transformation of functors $f_*:\K_{0,X/k}\to \K_{0,Y/k}$,
by the base change theorem \cite[Theorem 3.10.3]{Lipman}
(which can be applied since $f_T:X_T\to Y_T$ and $Y_U\to Y_T$ are Tor-independent for all morphisms $U\to T$ in $(\Sch/k)$).

\begin{prop}
\label{pushfwdSK0}
Let $f:X\to Y$ be a perfect birational morphism between proper integral varieties over $k$.
\begin{enumerate}[(i)]
\item If $Y$ is normal, then $f_*$ restricts to a morphism $f_*:\SK_{0,X/k}\to \SK_{0,Y/k}$.
\item If $X$ and $Y$ are regular, then $f_*\circ f^*$ is the identity of $\K_{0,Y/k}$.
\end{enumerate}
\end{prop}

\begin{proof}[Proof of~(i)]
Let $U\subseteq Y$ be the biggest open subset above which $f$ is an isomorphism. Since $Y$ is normal, the  depth of $\sO_{Y,y}$ is $\geq 2$ for all $y\in Y\setminus U$.

Fix $T\in(\Sch/k)$ and a class $x\in \SK_0(X_T)$. Since $(\rk(x),\det(x))|_{U_T}$ is trivial, so is 
$(\rk((f_T)_*x),\det((f_T)_*x))|_{U_T}\in\Z(U_T)\times\Pic(U_T)$. To deduce the triviality of $(\rk((f_T)_*x),\det((f_T)_*x))$, it suffices to show the injectivity of the restriction morphisms $\Z(Y_T)\to \Z(U_T)$ and $\Pic(Y_T)\to \Pic(U_T)$. 

The morphism $\Z(Y_T)\to \Z(U_T)$ is actually bijective by Remark~\ref{isoloc} (ii) applied to the injection $U\hookrightarrow Y$.
To show the injectivity of $\Pic(Y_T)\to \Pic(U_T)$, we can assume that 
$T$ is noetherian by absolute noetherian approximation \cite[Theorem~C.9]{TT} and by the limit arguments of \cite[\S8.5]{EGA43}. It then suffices to combine \cite[XI, Lemme~3.4]{SGA2} and \cite[Proposition 6.3.1]{EGA42}.
\let\qed\relax\end{proof}
\begin{proof}[Proof of~(ii)]
Fix $T\in(\Sch/k)$. Chatzistamatiou and R\"ulling \cite[Theorem 1.1]{Charu} have shown that the natural morphism $\sO_Y\to \RR f_*\sO_{X}$ is a quasi-isomorphism. The base change theorem \cite[Theorem 3.10.3]{Lipman} (which can be applied as $Y_T$ is flat over~$Y$) implies that the natural morphism $\sO_{Y_T}\to \RR (f_T)_*\sO_{X_T}$ is also a quasi-isomorphism. That $(f_T)_*\circ(f_T)^*$ is the identity of $\K_0(Y_T)$ then follows from the projection formula~(\ref{projformula}).
\end{proof}

\subsubsection{Curves} 
\label{curves}

We can entirely compute $\K_{0,X/k,\tau}$ if $X$ is a curve.

\begin{prop}
\label{K0curve}
If $\tau\in\{\Zar,\et,\fppf\}$ and if $X$ is a projective variety of dimension~${\leq 1}$ over $k$, then $(\rk,\det):\K_{0,X/k,\tau}\to\Z_{X/k}\times\Pic_{X/k,\tau}$ is an isomorphism.
\end{prop}

\begin{proof}
It suffices to prove the proposition for $\tau=\Zar$.
The commutation of $\K_0$ and $\Pic$ with directed inverse limits of qcqs schemes with affine transition maps (see \hbox{\cite[Proposition 3.20]{TT}} and \cite[\S 8.5]{EGA43}),
applied to the system of affine neighbourhoods of a point in a qcqs $k$-scheme, shows that it suffices to prove the bijectivity of $(\rk,\det):\K_{0}(X_T)\to\Z(X_T)\times\Pic(X_T)$ for any local $k$-scheme~$T$. 
By absolute noetherian approximation \cite[Theorem C.9]{TT}, we can write $T$ as the limit of a directed inverse system $(T_i)_{i\in I}$ of noetherian $k$-schemes with affine transition maps. Replacing the $T_i$ by their localizations at the images of the closed point of~$T$, we may assume that they are local. A limit argument as above then shows that we may assume $T$ to be noetherian.
This case  follows from Lemma \ref{redlocal} below applied to the connected components of $X_T$.
\end{proof}

\begin{lem}
\label{redlocal}
Let $\pi:Y\to T$ be a projective morphism with $T$ local noetherian, $Y$ non-empty and connected, and fibers of dimension $\leq 1$. Then the morphism $(\rk,\det):\K_{0}(Y)\to\Z\times\Pic(Y)$ is bijective.
\end{lem}

\begin{proof}
The surjectivity of $(\rk,\det)$ is obvious and we prove its injectivity. 

As explained in \S\ref{defK0}, $\K_0(Y)$ is generated by classes of vector bundles on~$Y$. Lemma \ref{vblocalcurve} below and induction on the rank of vector bundles show that it is even generated by classes of line bundles on $Y$.
Let $\sL$ and $\sM$ be two line bundles on~$Y$. Applying Lemma \ref{vblocalcurve} twice
with the same $l\gg0$ and with the same very ample line bundle $\sO_Y(1)$ on $Y$ yields exact sequences $0\to \sO_{Y}(-l)\to\sL\oplus\sM\to\sF_1\to 0$ and  $0\to \sO_{Y}(-l)\to\sO_{Y}\oplus(\sL\otimes\sM)\to\sF_2\to 0$. Since the line bundles $\sF_1$ and~$\sF_2$ have the same determinant, they are isomorphic, and we deduce the identity
$[\sL]+[\sM]=[\sO_{Y}]+[\sL\otimes\sM]\in \K_0(Y)$. This identity and the fact that $\K_0(Y)$ is generated by line bundles implies that $x=[\det(x)]+(\rk(x)-1)[\sO_{Y}]$ for all $x\in \K_0(Y)$. This shows at once the required injectivity.
\end{proof}

\begin{lem}
\label{vblocalcurve}
In the setting of Lemma \ref{redlocal}, there exists a $\pi$-ample line bundle $\sO_Y(1)$ on~$Y$ with the following property. For all vector bundles $\sE$ of rank~$\geq 2$ on $Y$ and all $l\gg 0$, there exists a short exact sequence of vector bundles on $Y$ of the form
\begin{equation}
\label{sesvb}
0\to\sO_{Y}(-l)\to\sE\to\sF\to 0\rlap{.}
\end{equation}
\end{lem}

\begin{proof}
Let $\sO_Y(1)$ be a $\pi$-ample line bundle on $Y$, let $t$ be the closed point of $T$, and let $A\subseteq Y_t$ be a finite subset meeting all the irreducible components of $Y_t$. If $m\gg 0$, then $H^0(Y_t, \mathcal{O}_{Y_t}(m))\to H^0(A, \mathcal{O}_A(m))$ is surjective. As a consequence, after replacing $\sO_Y(1)$ with $\sO_Y(m)$, we may assume the existence of a section $\alpha\in H^0(Y_t, \mathcal{O}_{Y_t}(1))$ that vanishes at only finitely many points.

The same argument yields, for some $m\geq 0$, a section $\beta\in H^0(Y_t, \mathcal{E}_{Y_t}(m))$ that vanishes at only finitely many points. Let $B\subseteq Y_t$ be a finite subset meeting every irreducible
component of~$Y_t$ and such that $\alpha\beta$ does not vanish at any point of~$B$.
Since $H^0(Y_t, \mathcal{E}_{Y_t}(n))\to H^0(B, \mathcal{E}_B(n))$ is surjective for $n\gg 0$, we can choose $n\geq 0$ and a section $\gamma\in H^0(Y_t, \mathcal{E}_{Y_t}(n))$ such that $\alpha\beta$ and $\gamma$  are linearly dependent at only finitely many points of $Y_t$. Let $C\subseteq Y_t$ be this finite set of points. 

Choose any $l\geq \max(m,n)$ such that $H^0(Y_t, \mathcal{O}_{Y_t}(l-n))\to H^0(C, \mathcal{O}_C(l-n))$
and  $H^0(Y,\sE(l)) \to  H^0(Y_t, \mathcal{E}_{Y_t}(l))$
are surjective.  (Such~$l$ exist by Serre vanishing \cite[Th\'eor\`eme~2.2.1]{EGA31}; this is where
we use the noetherianity of~$T$.)
 Then there exists $\delta\in H^0(Y_t, \mathcal{O}_{Y_t}(l-n))$ such that $\tau:=\alpha^{l-m}\beta+\gamma\delta\in H^0(Y_t, \mathcal{E}_{Y_t}(l))$ vanishes nowhere. Lift $\tau$ to a section $\sigma\in H^0(Y,\sE(l))$. Since $\pi$ is proper and $T$ is local, $\sigma$ does not vanish on $Y$, thus giving rise to a short exact sequence of the form~(\ref{sesvb}).
\end{proof}

\begin{rem}
If the residue field of $T$ is infinite, the proof of Lemma \ref{vblocalcurve} can be simplified as one can then apply \cite[Corollary~3.6]{KleGrass} on $Y_t$ to construct $\tau$.
\end{rem}

\subsubsection{Projective bundles and blow-ups}
\label{prblsheaf}

If $X$ is a proper variety over $k$ and if $\sE$ is a vector bundle of rank $r$ on $X$, the formula (\ref{proj}) induces an isomorphism of functors 
\begin{equation}
\label{projfunct}
\K_{0,X/k}^r\isoto \K_{0,\bP_X\sE/k}.
\end{equation}
In view of the isomorphism $\rk:\K_{0,\Spec(k)/k,\Zar}\isoto\Z$ given by Proposition~\ref{K0curve}, we deduce that $\K_{0,\bP^{r-1}_k/k,\Zar}\simeq\Z^r$ and that $([\sO_{\bP^{r-1}_k}(-j)])_{0\leq j\leq r-1}$ forms a basis of the $\Z$\nobreakdash-module $\K_{0,\bP^{r-1}_k/k,\Zar}(k)$. The family $([\sO_{\bP^{d}_k}])_{0\leq d\leq r-1}$ is another basis.
For $r\geq 2$, identifying the morphism $(\rk,\det):\K_{0,\bP^{r-1}_k/k,\Zar}\to\Z_{\bP^{r-1}_k/k}\times\Pic_{\bP^{r-1}_k/k,\Zar}=\Z^2$ yields an isomorphism $\SK_{0,\bP^{r-1}_k/k,\Zar}\simeq \Z^{r-2}$ and shows that $([\sO_{\bP^{d}_k}])_{0\leq d\leq r-3}$ is a basis of the $\Z$\nobreakdash-module $\SK_{0,\bP^{r-1}_k/k,\Zar}(k)$.

Let $X$ be a proper variety over $k$ and $i:Y\to X$ be a regular closed immersion of pure codimension $c\geq 1$. Define $p:X'\to X$ to be the blow-up of $X$ along $Y$ with exceptional divisor $Y'$, and $p':Y'\to Y$ and $i':Y'\to X'$ to be the natural morphisms. Then the formula (\ref{blowup}) induces an isomorphism of functors
\begin{equation}
\label{blowupfunct}
\K_{0,X/k}\times \K_{0,Y/k}^{c-1}\isoto \K_{0,X'/k},
\end{equation}
in view of the functorialities described in \S\ref{functfunct}.

\subsection{The functor \texorpdfstring{$\CH^2_{X/k,\fppf}$}{CH²\_\{X/k,fppf\}}} 
\label{functCH2}

We introduce, for a smooth proper geometrically connected threefold $X$ over $k$ with geometrically trivial Chow group of zero-cycles, the functor $\CH^2_{X/k,\fppf}$. It will play for codimension $2$ cycles the same role as the Picard functor $\Pic_{X/k,\fppf}$ does for codimension $1$ cycles.

\subsubsection{The class of a point}
\label{classpointpar}
We first exhibit a canonical class $\nu_X\in \K_{0,X/k,\fppf}(k)$.

\begin{prop}
\label{classofpoint}
Let $X$ be a smooth proper geometrically connected variety over $k$ whose degree map $\deg:\CH_0(X_{\ok})\to\Z$ is an isomorphism.
 Choose $\tau\in\{\Zar,\et,\fppf\}$. Assume that $k=\ok$ if $\tau=\Zar$ and that $k=k_{\p}$ if $\tau=\et$. There exists a unique $\nu_X\in \K_{0,X/k,\tau}(k)$ such that for all finite extensions $k'$ of $k$ and all coherent sheaves $\sF$ on $X_{k'}$ whose support has dimension $0$, one has 
\begin{equation}
\label{eqnu}
[\sF]=h^0(X_{k'},\sF)\cdot \nu_X\in \K_{0,X/k,\tau}(k').
\end{equation}
\end{prop}

\begin{proof}
Using \cite[2.2/13]{BLR}, choose a finite Galois extension $l$ of $k$ 
and a point $x\in X(l)$.
Let $n$ be the dimension of $X$.
For all field extensions $l'$ of~$l$, the definition of the flat pull-back of a cycle \cite[\S 1.5, \S 1.7]{Fulton} and the formula \cite[X,~(1.1.3)]{SGA6} show the commutativity of the natural diagram
\begin{equation}
\label{commzerocycles}
\begin{aligned}
\xymatrix
@R=0.4cm
{
\CH_0(X_l)\ar^{\varphi^n}[r]\ar[d]^{} & F_0\K_0(X_l)\ar[d]^{}  \\
\CH_0(X_{l'})\ar^{\varphi^{n}}[r]& F_0\K_0(X_{l'})\rlap{,}
}
\end{aligned}
\end{equation}
where the morphisms $\varphi^n$ are defined in \S\ref{coherent}.

Assume first that $\tau=\fppf$.

 Let $\{x_1,\dots,x_m\}$ be the $\Gal(l/k)$-orbit of $x$.
Since $\deg:\CH_0(X_{\ok})\isoto\Z$ is an isomorphism, there exists a finite extension $l'$ of $l$ such that the $x_i$ all have the same class in $\CH_0(X_{l'})$. Since $\Spec(l')\to \Spec(l)$ is an fppf covering, we deduce from~(\ref{commzerocycles}) that $\varphi^n([x])\in \K_{0,X/k,\tau}(l)$ is $\Gal(l/k)$\nobreakdash-invariant, hence descends to a class $\nu_X\in \K_{0,X/k,\tau}(k)$ since $\Spec(l)\to\Spec(k)$ is an \'etale covering.
Applying~(\ref{eqnu}) with $k'=l$ and $\sF=\sO_x$ shows that this is the only possible choice for $\nu_X$ and proves the uniqueness assertion of Proposition \ref{classofpoint}.

Let us now show that $\nu_X$ satisfies (\ref{eqnu}). Let $k'$ be a finite extension of $k$ and let $\sF$ be a coherent sheaf on $X_{k'}$ whose support has dimension $0$. Let $l'$ be a finite extension of $k$ containing both $k'$ and $l$, with the property that all the points in the support of $\sF_{l'}$ have residue field $l'$, and are rationally equivalent to $x$. (Such an $l'$ exists since $\deg:\CH_0(X_{\ok})\isoto\Z$.) The formula \cite[X,~(1.1.3)]{SGA6} and the commutativity of (\ref{commzerocycles}) show that $[\sF_{l'}]=h^0(X_{l'},\sF_{l'})\cdot \nu_X\in \K_{0,X/k,\tau}(l')$. Since $\Spec(l') \to \Spec(k')$ is an fppf covering, identity (\ref{eqnu}) follows.

If $\tau=\et$ (resp.\ $\tau=\Zar$), all the fppf (resp.\ fppf or \'etale) coverings that
appear above are \'etale (resp.\ Zariski) coverings, proving the proposition in these cases.
\end{proof}

\subsubsection{Codimension $2$ cycles on a threefold}
\label{CH2def}

Let us fix in \S\ref{CH2def} a smooth proper geometrically connected threefold $X$ over~$k$ whose degree map $\deg:\CH_0(X_{\Omega})\to\Z$ is an isomorphism for all algebraically closed field extensions $k\subseteq\Omega$ (an assumption that is satisfied if $X$ is $\ok$-rational).

Choose $\tau\in\{\Zar,\et,\fppf\}$, and assume that $k=\ok$ if $\tau=\Zar$ and that $k=k_{\p}$ if $\tau=\et$, so that $\Spec(l)\to\Spec(k)$ is a $\tau$\nobreakdash-covering for any finite extension~$l$ of~$k$.
Let $\nu_X\in \K_{0,X/k,\tau}(k)$ be the class defined in Proposition~\ref{classofpoint}. 
If $x\in X(l)$ for some finite extension $l$ of $k$, then $(\rk,\det)([\sO_x])=(0,\sO_X)$ (see \S\ref{coherent}).
In view of (\ref{eqnu}), one therefore has $\nu_X\in \SK_{0,X/k,\tau}(k)\subseteq \SK_{0,X/k,\tau}(l)$.

We still denote by~$\nu_X$ the morphism of $\tau$\nobreakdash-sheaves
$\nu_X:\Z\to \SK_{0,X/k,\tau}$ such that $\nu_X(1)=\nu_X$.
We view $\nu_X$ as a morphism of presheaves of abelian groups.

\begin{Def} 
\label{defch2xk}
We let $\CH^2_{X/k,\tau}:(\Sch/k)^{\op}\to(\Ab)$ be the (presheaf) cokernel of~$\nu_X:\Z\to \SK_{0,X/k,\tau}$. When $\CH^2_{X/k,\fppf}$ is representable, we let $\bCH^2_{X/k}$ be the group scheme over~$k$ that represents it.
\end{Def}

\begin{rem}
Let $k'/k$ be a field extension. There is an obvious identification $\K_{0,X_{k'}/k'}(T)=\K_{0,X/k}(T)$ for all $T\in(\Sch/k')$. We thus obtain natural isomorphisms $\K_{0,X_{k'}/k',\tau}(T)=\K_{0,X/k,\tau}(T)$, $\SK_{0,X_{k'}/k',\tau}(T)=\SK_{0,X/k,\tau}(T)$ and $\CH^2_{X_{k'}/k',\tau}(T)=\CH^2_{X/k,\tau}(T)$ for all $T\in(\Sch/k')$ and $\tau\in\{\Zar,\et,\fppf\}$.
In particular, if $\CH^2_{X/k,\fppf}$ is representable, so is $\CH^2_{X_{k'}/k',\fppf}$, and $\bCH^2_{X_{k'}/k'}=(\bCH^2_{X/k})_{k'}$.
\end{rem}

 The following proposition justifies these definitions.

\begin{prop}
\label{propK0CH2}
Associating with the class $[Z]\in \CH^2(X_{\ok})$ of a codimension $2$ integral closed subvariety $Z\subset X_{\ok}$ the class $[\sO_Z]\in \K_0(X_{\ok})$ of its structure sheaf induces a $\Gamma_k$-equivariant isomorphism
\begin{equation}
\label{okpoints}
\CH^2(X_{\ok})\isoto \CH^2_{X_{\ok}/\ok,\Zar}(\ok)=\CH^2_{X/k,\Zar}(\ok)\rlap{.}
\end{equation}
\end{prop}

\begin{proof}
In view of (\ref{eqnu}), one has a natural isomorphism
\begin{equation}
\label{okpointsint}
\SK_0(X_{\ok})/F_0\K_0(X_{\ok})\isoto\CH^2_{X_{\ok}/\ok,\Zar}(\ok)\rlap{.}
\end{equation}
Precomposing (\ref{okpointsint}) with  $\varphi^2:\CH^2(X_{\ok})\isoto\Gr^2_F\K_0(X_{\ok})=\SK_0(X_{\ok})/F_0\K_0(X_{\ok})$ (see \S\ref{coherent}) yields the isomorphism (\ref{okpoints}). It is $\Gamma_k$\nobreakdash-equivariant by construction.
\end{proof}

The next lemma will be used in the proof of Theorem \ref{threp} (iv).
For the definition of $\alpha|_{X_t} \in \CH^2(X_t)$ in its statement,
see \cite[Example~5.2.1]{Fulton}.

\begin{lem}
\label{functforMurre}
Let $T$ be a smooth connected variety over $\ok$.
\begin{enumerate}[(i)]
\item For all $\alpha\in \CH^2(X_T)$, there exists a class $\beta\in \SK_0(X_T)$ with the property that for all $t\in T(\ok)$, the image of $\alpha|_{X_t}$ by (\ref{okpoints}) is induced by $\beta|_{X_t}$.
\item For all $\beta\in \SK_0(X_T)$, there exists a class $\alpha\in\CH^2(X_T)$ with the property that for all $t\in T(\ok)$, the image of $\alpha|_{X_t}$ by (\ref{okpoints}) is induced by $\beta|_{X_t}$.
\end{enumerate}
\end{lem}

\begin{proof}[Proof of~(i)]
We can assume that $\alpha$ is the class of an integral subvariety $Z\subset X_T$ of codimension $2$. Define $\beta:=[\sO_Z]\in \SK_0(X_T)$. 
By the Riemann--Roch theorem without denominators, one has $\alpha=-c_2(\beta)$ (see \S\ref{coherent}). For $t\in T(\ok)$, one has $\alpha|_{X_t}=-c_2(\beta|_{X_t})$. Applying the Riemann--Roch theorem without denominators again shows that the image of $\alpha|_{X_t}$ by (\ref{okpoints}) is the class induced by $\beta|_{X_t}$.
\let\qed\relax\end{proof}
\begin{proof}[Proof of~(ii)]
Define $\alpha=-c_2(\beta)\in\CH^2(X_T)$ and argue as in (i).
\end{proof}

\section{Geometrically rational threefolds}
\label{sec3}

In Section \ref{sec3}, we prove the representability of the functor $\CH^2_{X/k,\fppf}$ defined in~\S\ref{CH2def} if $X$ is a smooth projective $\ok$-rational threefold, and study the group scheme $\bCH^2_{X/k}$ that represents it.

\subsection{Main statement}

Our goal is the following theorem.

\begin{thm}
\label{threp}
Let $X$ be a smooth projective
 $\ok$-rational threefold over $k$. Then:
\begin{enumerate}[(i)]
\item $\CH^2_{X/k,\fppf}$ is represented by a smooth group scheme $\bCH^2_{X/k}$ over~$k$.
\item $(\bCH^2_{X/k})^0$ is an abelian variety over~$k$.
\item $\CH^2_{X_{k_{\p}}/k_{\p},\fppf}=\CH^2_{X_{k_{\p}}/k_{\p},\et}$ and $\CH^2_{X_{\ok}/\ok,\fppf}=\CH^2_{X_{\ok}/\ok,\Zar}$.
\item The $\Gamma_k$-equivariant isomorphism 
\begin{equation}
\label{okpointsbis}
\psi^2_X:\CH^2(X_{\ok})\isoto \bCH^2_{X/k}(\ok)
\end{equation}
 obtained by combining (iii) and (\ref{okpoints}) restricts to a $\Gamma_k$-equivariant bijective regular homomorphism (in the sense of~\S\ref{parMurre})
\begin{equation}
\label{okopoints}
\psi^2_X:\CH^2(X_{\ok})_{\alg}\isoto (\bCH^2_{X/k})^0(\ok).
\end{equation}
\item The \'etale group scheme $\bCH^2_{X/k}/(\bCH^2_{X/k})^0$ over $k$ is associated with the $\Gamma_k$\nobreakdash-module $\NS^2(X_{\ok})$, which as a $\Z$\nobreakdash-module is free of finite rank.
\item The isomorphism (\ref{okopoints}) induces an isomorphism $\Ab^2(X_{k_\p})\isoto(\bCH^2_{X_{k_\p}/k_\p})^0$ (where $\Ab^2(X_{k_\p})$ denotes Murre's intermediate
Jacobian, introduced in~\S\ref{parMurre}).
\end{enumerate}
We endow $\bCH^2_{X/k}$ with the principal polarization (in the sense of \S\ref{pp}) induced by the principal polarization $\theta_X$ of $\Ab^2(X_{k_\p})$ (see \S\ref{parMurre}).
\begin{enumerate}[(i)]
\item[(vii)] If $X$ is $k$-rational, there exists a smooth projective curve $B$ over $k$ such that $\bCH^2_{X/k}$ is a principally polarized direct factor of $\bPic_{B/k}$ (in the sense of \S\ref{pp}).
\end{enumerate}
\end{thm}

The proof of Theorem \ref{threp} is given in \S\ref{parproof}. Theorem \ref{threp} is complemented
in~\S\ref{parblowup}
 by a computation of $\bCH^2_{X/k}$ for varieties constructed as blow\nobreakdash-ups, and
 in~\S\ref{parobs}
 by an analysis of the obstructions to $k$-rationality arising from Theorem \ref{threp} (vii).

\begin{rems}
\label{remvariantech2xk}
(i) Let $X$ be a smooth projective variety over $k$. As recalled in~\S\ref{parMurre},  Achter, Casalaina-Martin and Vial have endowed $\Ab^2(X_{\ok})$ with a natural $k_\p$\nobreakdash-structure.
If~$X$ is moreover a $\ok$-rational threefold, Theorem \ref{threp} (vi) further endows $\Ab^2(X_{\ok})$ with a natural $k$-structure $(\bCH^2_{X/k})^0$. Trying to descend $\Ab^2(X_{\ok})$ to $k$ under more general hypotheses gives an incentive to define $\CH^2_{X/k,\fppf}$ and to prove its representability in a greater generality.

(ii) For instance it would be nice to define and study a functor $\CH^2_{X/k,\fppf}$ for all smooth proper varieties $X$ over~$k$ such that $\CH_0(X)_{\Q}$ is supported in dimension~$1$ in the sense of Bloch
and Srinivas (see \cite[Definition 2.1]{CGBW}).
If a good enough theory of motivic cohomology $H^*_{\sM}$
over a field of characteristic $p>0$ were available without having to invert~$p$ in
the coefficients,
a~natural choice would be the fppf sheafification of the functor $T\mapsto H^4_{\sM}(X_T,\Z(2))$.
One could also consider the fppf sheafification of the functor $T \mapsto H^2(X_T,\sK_{2})$,
where~$\sK_{2}$ denotes Quillen's $K$\nobreakdash-theory sheaf \cite{blochformula},
or of the functor $T \mapsto A^2(X_T)$, where~$A^2$
denotes Fulton's cohomological Chow group \cite[Definition~17.3]{Fulton}.

(iii) Even with our definition of $\CH^2_{X/k,\fppf}$, it would be interesting to show the representability of $\CH^2_{X/k,\fppf}$ under the hypothesis, weaker than $\ok$\nobreakdash-rationality, that $X$ is a smooth, proper and geometrically connected threefold over $k$ such that $\deg:\CH_0(X_{\Omega})\isoto\Z$ is an isomorphism for all algebraically closed field extensions $k\subseteq\Omega$. We note, however, that it is the proof of representability that we give here, and which is specific to $\bar{k}$-rational threefolds, that yields
the crucial Theorem~\ref{threp}~(vii) and thus provides obstructions to $k$-rationality.

(iv) We still denote by $\theta_X$ the principal polarization of $\bCH^2_{X/k}$ induced by that of $\Ab^2(X_{k_{\p}})$, as in Theorem \ref{threp} (vi). For the sake of completeness, we
extract a characterization of $\theta_X$ from the definition of the isomorphism~(\ref{okopoints}) and from \cite[Property 2.4, Corollary 2.8]{CGBW}. Let $\ell$ be a prime number invertible in~$k$.
Consider the diagram
\begin{align*}
\xymatrix{
(\bCH^2_{X/k})^0(\ok)\{\ell\}
&
\ar[l]_(.47){\psi^2_X}^(.47)\sim
\CH^2(X_{\ok})_{\alg}\{\ell\}
\ar[r]^(.42){\lambda^2} &
H^3(X_{\ok},\Z_{\ell}(2))\otimes\Q_{\ell}/\Z_{\ell}\rlap{,}
}
\end{align*}
where $\lambda^2$ is Bloch's Abel--Jacobi map (see \cite[\S 2]{Bloch}, \cite[(2.3)]{CGBW}) and $\psi^2_X$ is the map (\ref{okopoints}).
 Then $c_1(\theta_X)\in H^2((\bCH^2_{X/k})^0_{\ok},\Z_{\ell}(1))=\big(\bigwedge^2H^1((\bCH^2_{X/k})^0_{\ok},\Z_{\ell})\big)(1)$ corresponds, via the identification
$$H^1((\bCH^2_{X/k})^0_{\ok},\Z_{\ell})^\vee\xrightarrow{T_{\ell}(\lambda^2\circ(\psi^2_X)^{-1})}H^3(X_{\ok},\Z_{\ell}(2))/(\torsion)$$
(in which ${}^\vee$ stands for $\Hom(-,\Zl)$),
to the opposite of the cup product pairing $\bigwedge^2\hspace{-.2em}H^3(X_{\ok},\Z_{\ell}(2))\to
 H^6(X_{\ok},\Z_{\ell}(4))\isoto
\Z_{\ell}(1)$.
\end{rems}

\subsection{Proof of Theorem \ref{threp}}
\label{parproof}

\subsubsection{Resolution of indeterminacies}
\label{resindet}

Our main tool is a resolution of indeterminacies result that is due to Abhyankar \cite{Abhyankar} if $k$ is perfect, and to Cossart and Piltant~\cite{CPI} in general. 

\begin{prop}
\label{resolindet}
Let $X$ and $Y$ be smooth projective threefolds over $k$. Let $f:Y\dashrightarrow X$ be a birational map.
Then there exists a diagram
\begin{equation}
\label{resindetdiag}
X\xleftarrow{h} X'=Y_{N+1}\to \dots\to Y_{j+1}\xrightarrow{p_j} Y_j\to\dots\to Y_1=Y
\end{equation}
of regular projective varieties over $k$ such that $f=h\circ p_{N}^{-1}\circ\dots\circ p_{1}^{-1}$, where $p_j$ is the blow-up of an integral regular closed subscheme $Z_j\subset Y_j$ of codimension $c_j$ and where $h$ is projective and birational.
\end{prop}

\begin{proof}
This follows from \cite{CPI}, as explained in \cite[Proposition 2.11]{CGBW}. The standing assumption that $k$ is perfect in \cite{CGBW} is irrelevant if one really uses \cite[Proposition~4.2]{CPI} (or \cite[Theorem 5.9]{CJS}) instead of \cite[(9.1.4)]{Abhyankar} in the proof of \cite[Proposition 2.11]{CGBW}.
\end{proof}

\begin{rem}
If $k$ is not perfect, the subschemes $Z_j\subset Y_j$ may not be smooth over~$k$.
\end{rem}

\subsubsection{Representability if $X$ is $k$-rational}
\label{repkrat}

In \S\ref{repkrat}, we fix $\tau\in\{\Zar,\et,\fppf\}$ and assume that $k=\ok$ if $\tau=\Zar$ and that $k=k_{\p}$ if $\tau=\et$.
We also let $X$ be a smooth projective $k$-rational threefold. 
By Proposition \ref{resolindet}, there exists a diagram~(\ref{resindetdiag}) with $Y=\bP^3_k$. Remark~\ref{isoloc} (ii), Corollary~\ref{corblup} and the isomorphism~(\ref{blowupfunct}) then give canonical decompositions
\begin{align}
\Z_{X'/k}&\isofrom\Z_{\bP^3_k/k}\rlap{,}\\
\Pic_{X'/k}&\isofrom\Pic_{\bP^3_k/k}\times\prod_{c_j\geq 2} \Z_{Z_j/k}\rlap{,}\\
\label{eq:decomposition of k0}
\K_{0,X'/k}&\isofrom \K_{0,\bP^3_k/k}\times \prod_{c_j=2} \K_{0,Z_j/k}\times \prod_{c_j=3} (\K_{0,Z_j/k})^2\rlap{.}
\end{align}
Identifying $(\rk,\det):\K_{0,X'/k}\to \Z_{X'/k}\times\Pic_{X'/k}$ in terms of these decompositions and using the isomorphisms $(\rk,\det):\K_{0,Z_j/k,\tau}\isoto \Z_{Z_j/k}\times\Pic_{Z_j/k,\tau}$ given by Proposition \ref{K0curve} yields an isomorphism
\begin{equation}
\label{SK0blup}
 \SK_{0,X'/k,\tau}\isofrom
\SK_{0,\bP^3_k/k,\tau}\times \prod_{c_j=2} \Pic_{Z_j/k,\tau}\times \prod_{c_j=3} \Z_{Z_j/k}\rlap{.}
\end{equation}

The morphisms $p_j$ are lci by \cite[\S 1.2]{Thomason}, hence so is the structural morphism $X'\to\Spec(k)$ by \cite[VIII, Proposition 1.5]{SGA6}. Any closed embedding $X'\hookrightarrow\bP^N_X$ of the $X$-scheme $X'$ is a regular immersion by \cite[VIII, Proposition 1.2]{SGA6}, which shows that $h$ is lci, hence perfect \cite[VIII, Proposition 1.7]{SGA6}. The functors $h^*:\K_{0,X/k}\to \K_{0,X'/k}$ and $h_*:\K_{0,X'/k}\to \K_{0,X/k}$ satisfy $h_*\circ h^*=\Id$ by Proposition~\ref{pushfwdSK0} (ii).
Since they restrict to $h^*:\SK_{0,X/k}\to \SK_{0,X'/k}$ and to $h_*:\SK_{0,X'/k}\to \SK_{0,X/k}$ (see Proposition~\ref{pushfwdSK0} (i)), we deduce a natural decomposition
\begin{equation}
\label{directfactor}
\SK_{0,X/k}\times\Ker\mkern-1mu\big(h_*:\SK_{0,X'/k}\to \SK_{0,X/k}\big)\isoto \SK_{0,X'/k}\rlap{.}
\end{equation}

The three summands of the right-hand side of (\ref{SK0blup}) are represented by group schemes locally of finite type over $k$, respectively by \S\ref{prblsheaf}, by \S\ref{Pic} and by Proposition \ref{locconsrep}. It follows that $\SK_{0,X'/k,\tau}$ is represented by a group scheme locally of finite type over $k$. So is $\SK_{0,X/k,\tau}$, by (\ref{directfactor}) and Lemma~\ref{repfact}.

Let $x'\in X'(\ok)$ be a general point, and let $x\in X(\ok)$ and $y\in \bP^3(\ok)$ be its images by $h$ and by $p_1\circ\dots\circ p_N$. Then $h^*[\sO_x]=[\sO_{x'}]=p_N^*\circ\dots\circ p_1^*[\sO_y]\in \SK_{0,X'/k,\tau}(\ok)$.
Consequently, $h^*\circ\nu_X:\Z\to \SK_{0,X'/k,\tau}$ and $p_N^*\circ\dots\circ p_1^*\circ\nu_{\bP^3_k}:\Z\to \SK_{0,X'/k,\tau}$ both send $1\in\Z(\ok)$ to $[\sO_{x'}]\in \SK_{0,X'/k,\tau}(\ok)$.
For $n\neq 0$, the class $\nu_{\bP^3_k}(n)$ does not belong to the identity component of $\SK_{0,\bP^3_k/k,\tau}$, by \S\ref{prblsheaf}. We deduce from the above and from (\ref{SK0blup}) that $n[\sO_{x'}]\notin (\SK_{0,X'/k,\tau})^0(\ok)$, hence that 
$\nu_{X}(n)\notin (\SK_{0,X/k,\tau})^0(\ok)$.
Lemma \ref{repquotZ} then shows that $\CH^2_{X/k,\tau}$ is represented by a group scheme locally of finite type over $k$.

Applying the above with $\tau=\fppf$ and combining (\ref{SK0blup}), (\ref{directfactor}) and the equality
$h^*\circ\nu_X=p_N^*\circ\dots\circ p_1^*\circ\nu_{\bP^3_k}$ yields isomorphisms
\begin{equation}
\label{CH2sum}
\bCH^2_{X/k}\times\mkern4muG \isoto \bCH^2_{X'/k} \isofrom
\bCH^2_{\bP^3_k/k}\times \prod_{c_j=2} \bPic_{Z_j/k}\times \prod_{c_j=3} \Z_{Z_j/k}
\end{equation}
of smooth group schemes over $k$, where $G$ denotes the group scheme representing the $\fppf$ sheafification of $\Ker\mkern-1mu\big(h_*:\SK_{0,X'/k}\to \SK_{0,X/k}\big)$ (whose representability follows from Lemma~\ref{repfact}) and where $\bCH^2_{X'/k}$ denotes the group scheme representing the presheaf cokernel $\Coker\big(h^*\circ\nu_X :\Z\to \SK_{0,X'/k,\fppf}\big)$ (whose representability follows from Lemma~\ref{repquotZ}).

\subsubsection{Representability if $X$ is $\ok$-rational}
\label{repokrat}
In \S\ref{repokrat}, we prove Theorem \ref{threp} (i)--(iii) for a smooth projective $\ok$-rational threefold $X$ over $k$. 

Choose $\tau\in\{\Zar,\et,\fppf\}$ and assume that $k=\ok$ if $\tau=\Zar$ and that $k=k_{\p}$ if $\tau=\et$.
Let $l$ be a finite extension of $k$ such that $X$ is $l$-rational.
 Then $\SK_{0,X_l/l,\tau}$ is represented by a group scheme locally of finite type over~$l$, by the arguments of~\S\ref{repkrat} applied to the $l$-variety $X_l$.
By Lemma~\ref{lem:locrep=>rep}, it follows that  $\SK_{0,X/k,\tau}$
is represented by a group scheme locally of finite type over~$k$.
As explained in~\S\ref{repkrat}, the morphism $\nu_X:\Z\to \SK_{0,X/k,\tau}$ defined in \S\ref{CH2def} has the property that $\nu_{X}(n)\notin (\SK_{0,X/k,\tau})^0(\ok)$ for all $n\neq 0$. It now follows from Lemma~\ref{repquotZ} that $\CH^2_{X/k,\tau}$ is represented by a group scheme locally of finite type over~$k$.

\begin{proof}[Proof of Theorem \ref{threp} (i)--(iii)]
Applying the above argument to the $k$\nobreakdash-variety~$X$
with $\tau=\fppf$, to the $k_{\p}$\nobreakdash-variety $X_{k_{\p}}$ with $\tau=\et$,
and to the $\ok$\nobreakdash-variety $X_{\ok}$ with $\tau=\Zar$
shows that the three functors
$\CH^2_{X/k,\fppf}$,
$\CH^2_{X_{k_{\p}}/k_{\p},\et}$ and
$\CH^2_{X_{\ok}/\ok,\Zar}$
are represented by group schemes locally of finite type over~$k$,
over~$k_{\p}$ and over~$\ok$, respectively.
In particular, the latter two are sheaves for the fppf topology,
which proves Theorem~\ref{threp}~(iii).
In addition, the arguments of \S\ref{repkrat} applied to the $\ok$\nobreakdash-variety $X_{\ok}$ show that $(\bCH^2_{X_{\ok}/\ok})^0$ is a direct factor of a product of Jacobians of smooth projective curves over $\ok$ (see (\ref{CH2sum})), hence is an abelian variety; in particular,
it is smooth.
As $((\bCH^2_{X/k})^0)_{\ok}=(\bCH^2_{X_{\ok}/\ok})^0$,
Theorem~\ref{threp}~(i)--(ii) follows.
\end{proof}

\subsubsection{Relation with Murre's work}

We now prove Theorem \hbox{\ref{threp}~(iv)--(vi)}.

\begin{proof}[Proof of Theorem \ref{threp} (iv)]
That $\psi^2_X(\CH^2(X_{\ok})_{\alg})\subseteq ((\bCH^2_{X/k})^0(\ok))$, and that the resulting morphism 
$\psi^2_X:\CH^2(X_{\ok})_{\alg}\to (\bCH^2_{X/k})^0(\ok)$ is a regular homomorphism follow at once from Lemma \ref{functforMurre} (i). It remains to show that (\ref{okopoints}) is surjective.

Since $\bCH^2_{X_{\ok}/\ok}$ represents $\CH^2_{X_{\ok}/\ok,\Zar}$ by  Theorem \ref{threp}~(iii), we can choose a connected Zariski neighbourhood $T$ of the identity in $(\bCH^2_{X_{\ok}/\ok})^0$ and a class $\beta\in \SK_{0}(X_T)$ inducing the natural inclusion $T\to \bCH^2_{X_{\ok}/\ok}$. Lemma \ref{functforMurre} (ii) then implies that $T(\ok)$ is contained in the image of (\ref{okopoints}). Since $(\bCH^2_{X/k})^0(\ok)$ is generated by $T(\ok)$ as a group, we have proved the surjectivity of (\ref{okopoints}).
\end{proof}

\begin{proof}[Proof of Theorem \ref{threp} (v)] 
That the \'etale $k$-group scheme $(\bCH^2_{X/k})/(\bCH^2_{X/k})^0$ corresponds to the $\Gamma_k$-module $\NS^2(X_{\ok})$ follows at once from (\ref{okpointsbis}) and (\ref{okopoints}).

Applying the discussion of \S\ref{repkrat}, and more precisely identity (\ref{CH2sum}), to the $\ok$\nobreakdash-variety $X_{\ok}$,
shows, in view of the isomorphism $\bCH^2_{\bP^3_{\ok}/\ok}\simeq \Z$ (see \S\ref{prblsheaf}), that $\NS^2(X_{\ok})$ is a free $\Z$\nobreakdash-module of finite rank, being a direct factor of such a module.
\end{proof}

\begin{proof}[Proof of Theorem \ref{threp} (vi)] 
 The regular homomorphism (\ref{okopoints}) induces a morphism $\iota_{X_{\ok}}:\Ab^2(X_{\ok})\to(\bCH^2_{X_{\ok}/\ok})^0$.
Since (\ref{okopoints}) is $\Gamma_k$-equivariant, and in view of the definition of $\Ab^2(X_{k_{\p}})$ recalled in \S\ref{parMurre}, this morphism descends by Galois descent to a morphism $\iota_{X_{k_{\p}}}:\Ab^2(X_{k_{\p}})\to(\bCH^2_{X_{k_{\p}}/k_{\p}})^0$ of abelian varieties over $k_{\p}$. To prove that $\iota_{X_{k_{\p}}}$ is an isomorphism, it suffices to prove that $\iota_{X_{\ok}}$ is an isomorphism.
From now on, we may thus assume that $k=\ok$.

By Proposition \ref{resolindet}, there exists a diagram (\ref{resindetdiag}). Since $k=\ok$, all the varieties~$Z_j$ and $Y_j$ that appear in it are smooth over $k$. Consider the diagram
\begin{equation}
\label{Murreblup}
\begin{aligned}
\xymatrix
@R=0.4cm
@C=1.1cm
{
\mkern20mu\Ab^2(\bP^3_k)\times\prod_{c_j=2}\bPic^0_{Z_j/k}\ar^(.65){\sim}[r]\ar@<-.71em>[d]^(.57){(\iota_{\bP^3_k},\Id)} &\Ab^2(X')\ar[d]^{\iota_{X'}}  \\
(\bCH^2_{\bP^3_k/k})^0\times\prod_{c_j=2}\bPic^0_{Z_j/k}\ar^(.65){\sim}[r]&(\bCH^2_{X'/k})^0\rlap{,}
}
\end{aligned}
\end{equation}
where the lower horizontal isomorphism is induced by (\ref{SK0blup}) and the upper horizontal isomorphism is the one provided by \cite[Lemma 2.10]{CGBW}. Since $\CH^2(\bP^3_k)_{\alg}=0$, one has $\Ab^2(\bP^3_k)=(\bCH^2_{\bP_k^3/k})^0=0$ and the left vertical arrow of (\ref{Murreblup}) is an isomorphism. We claim that (\ref{Murreblup}) commutes. Since $k=\ok$, it suffices to verify that it commutes at the level of $k$-points, which follows from unwinding the definitions and making use of Lemma~\ref{lemCH2SK0}. A glance at (\ref{Murreblup}) now shows that $\iota_{X'}$ is an isomorphism.

Now, consider the diagram
\begin{equation}
\label{Murremodif}
\begin{aligned}
\xymatrix
@R=0.4cm
@C=1.1cm
{
\Ab^2(X)\ar^{h^+}[r]\ar^{\iota_X}[d] &\Ab^2(X')\ar[r]^{h_+}\ar[d]_(.45){\wr}^{\iota_{X'}}&\Ab^2(X)\ar[d]^{\iota_{X}}   \\
(\bCH^2_{X/k})^0\ar^{h^*}[r]&(\bCH^2_{X'/k})^0\ar^{h_*}[r]&(\bCH^2_{X/k})^0\rlap{,}
}
\end{aligned}
\end{equation}
whose lower horizontal arrows are induced by (\ref{directfactor}) and hence satisfy $h_*\circ h^*=\Id$, and whose upper horizontal arrows are given by the functoriality of Murre's intermediate Jacobians (see \cite[\S 2.2.1]{CGBW}) and satisfy $h_+\circ h^+=\Id$ as a consequence of the identity $h_*\circ h^*=\Id:\CH^2(X)\to\CH^2(X)$ stemming from the projection formula \cite[Proposition~8.3(c)]{Fulton}. To show that (\ref{Murremodif}) commutes, it suffices to check that it commutes at the level of $k$-points, since $k=\ok$. This follows from Lemma~\ref{lemCH2SK0} for the left-hand square, and from the fact that the morphisms $\varphi^c$ considered in~\S\ref{coherent} are compatible with proper push-forwards \hbox{\cite[Example 15.1.5]{Fulton}} for the right-hand square. A diagram chase in (\ref{Murremodif}) shows that $\iota_X$ is an isomorphism since $\iota_{X'}$ is one, which concludes the proof.
\end{proof}

\subsubsection{Further analysis of k-rational varieties}
\label{biratbeh}

We resume the discussion of \S\ref{repkrat} with $\tau=\fppf$, and keep the notation introduced there.
Since $\bCH^2_{\bP^3_k/k}\simeq \Z$ by \S\ref{prblsheaf}, identity (\ref{CH2sum}) reads:
\begin{equation}
\label{CH2sumbis}
\bCH^2_{X/k}\times\mkern4muG \isoto \bCH^2_{X'/k} \isofrom
\Z\times \prod_{c_j=2} \bPic_{Z_j/k}\times \prod_{c_j=3} \Z_{Z_j/k}\rlap{.}
\end{equation}
The identity component of the right-hand side is isomorphic to $\prod_{c_j=2}\bPic^0_{Z_j/k}$, hence
it carries a natural principal polarization (see \S\ref{Piccurves}).
Via~\eqref{CH2sumbis}, we thus obtain
a principal polarization on~$\bCH^2_{X'/k}$
in the sense of~\S\ref{pp}.

\begin{prop}
\label{propppdirfact}
The isomorphism (\ref{CH2sumbis}) realizes $\bCH^2_{X/k}$ and $G$ as principally polarized direct factors, in the sense of \S\ref{pp}, of $\bCH^2_{X'/k}$, and the induced polarization on $\bCH^2_{X/k}$ coincides with the one defined in Theorem \ref{threp} (vi).
\end{prop}

\begin{proof}
We fix once and for all a prime number~$\ell$ invertible in~$k$
and start with a few recollections about (Borel--Moore) $\ell$\nobreakdash-adic \'etale homology.
If~$V$ is a variety over~$\ok$, 
the $i$\nobreakdash-th \'etale homology group of~$V$ with coefficients in~$\Ql(j)$
is defined by
 $H_i(V,\Ql(j))=H^{-i}(V,R\varepsilon^!\Ql(j))$,
where $\varepsilon:V\to \Spec(\ok)$ denotes the
structural morphism (see \cite[Définition~2]{laumon}).
This group is covariantly functorial with respect to proper morphisms
(\emph{loc.\ cit.}, \textsection4)
and comes with a cap product operation
$H^s(V,\Ql(t)) \times H_i(V,\Ql(j)) \to H_{i-s}(V,\Ql(j+t))$
(\emph{loc.\ cit.}, p.\,VIII-09),
and with a cycle class map $\cl:\CH_i(V)\to H_{2i}(V,\Ql(-i))$
(\emph{loc.\ cit.}, \textsection6),
for any $i$, $j$, $s$, $t$.
If~$V$ is of pure dimension~$n$, we denote by~$[V]$ the fundamental cycle of~$V$
(see \cite[\textsection1.5]{Fulton}),
so that $\cl([V]) \in H_{2n}(V,\Ql(-n))$.
For all~$i$, $j$,
the map
\begin{align}
\kappa_V^{\et}:H^{2n-i}(V,\Ql(j+n)) \to H_{i}(V,\Ql(j))
\end{align}
defined by $\kappa_V^{\et}(x)= x \cap \cl([V])$
is an isomorphism
if in addition~$V$ is smooth (see \cite[Prop.~3.2]{laumon}).
Thus, given a proper morphism $f:V'\to V$ from a variety~$V'$ of pure dimension~$n'$
to a smooth variety~$V$ of pure dimension~$n$ over~$\ok$, one can define a push-forward in \'etale cohomology
\begin{align}
\label{eq:pushforward in cohomology}
f_*:H^s(V',\Ql(t)) \to H^{s+2n-2n'}(V,\Ql(t+n-n'))
\end{align}
as the composition $(\kappa_V^{\et})^{-1} \circ f_* \circ \kappa_{V'}^{\et}$.
When~$V$ is a proper variety over~$\ok$ of pure odd dimension~$n$, we will speak of
the ``cup product pairing on $H^n(V,\Ql((n+1)/2))$'' to refer to the pairing
$H^n(V,\Ql((n+1)/2)) \times H^n(V,\Ql((n+1)/2)) \to \Ql(1)$
induced by
the cup product
and by the push-forward map $H^{2n}(V,\Ql(n+1)) \to \Ql(1)$ along the structural morphism $V\to \Spec(\ok)$
(see~\eqref{eq:pushforward in cohomology}).

\begin{lem}
\label{lem:lambda1lambda2}
Let $f:D \to V$ be a morphism between projective varieties
over~$\ok$, where~$D$ has pure dimension~$d$
and~$V$ has pure dimension~$d+1$,
for some integer~$d$.
Suppose that~$V$ is smooth.
Let
$\lambda^1:\Pic(D)\{\ell\} \isoto H^1(D,\Ql/\Zl(1))$
and $\lambda^2:\CH^2(V)\{\ell\} \to H^3(V,\Ql/\Zl(2))$
respectively
denote the Kummer isomorphism
and
Bloch's $\ell$\nobreakdash-adic
Abel--Jacobi map.
Then the diagram
\begin{align}
\label{diag:lemmelambda1lambda2}
\begin{aligned}
\xymatrix@R=3ex@C=5em{
V_\ell(F^2\K_0(V)) \ar[r]^{-V_\ell(c_2)} &
V_\ell(\CH^2(V))
\ar[r]^{V_\ell(\lambda^2)} &  H^3(V,\Ql(2))\\
V_\ell(\K_0(D)) \ar[r]^{V_\ell(\det)}
\ar[u]_(.45){f_*} &
V_\ell(\Pic(D)))
 \ar[r]^{V_\ell(\lambda^1)} & H^1(D,\Ql(1)) \ar[u]_(.45){f_*}
}
\end{aligned}
\end{align}
commutes,
where the right-hand side vertical arrow is the map~\eqref{eq:pushforward in cohomology}
and
the left-hand side vertical arrow is induced by the composition of the canonical map $\K_0(D) \to \GG_0(D)$,
which sends the rank~$0$ subgroup of~$\K_0(D)$
to $F_{d-1}\GG_0(D)$
(see \cite[X, Corollaire~1.3.3]{SGA6}),
with $f_*:F_{d-1}\GG_0(D) \to F_{d-1}\GG_0(V)$.
\end{lem}

We stress that~$D$ is not assumed to be reduced in Lemma~\ref{lem:lambda1lambda2}.

\begin{proof}
Let $\widetilde{\K_0}(D)=\Ker(\rk:\K_0(D)\to\Z)$
and $\Gr^1_\gamma \K_0(D) = \widetilde{\K_0}(D) / \SK_0(D)$
(a piece of notation justified by the fact that
$\SK_0(D)$ and $\widetilde{\K_0}(D)$ form the beginning of
the $\gamma$\nobreakdash-filtration on~$\K_0(D)$).
As the canonical map
$\K_0(D) \to \GG_0(D)$
sends~$\widetilde{\K_0}(D)$ to $F_{d-1}\GG_0(D)$
and~$\SK_0(D)$ to $F_{d-2}\GG_0(D)$
(see \cite[X, Corollaire~1.3.3]{SGA6}),
there is an induced map $f_*:V_\ell(\Gr^1_\gamma \K_0(D)) \to V_\ell(\Gr^2_F\K_0(V))$
and it suffices to prove the commutativity of the diagram
\begin{align}
\label{diag:lemmelambda1lambda2bis}
\begin{aligned}
\xymatrix@R=3ex@C=5em{
V_\ell(\Gr^2_F \K_0(V)) \ar[r]^{-V_\ell(c_2)}_\sim &
V_\ell(\CH^2(V))
\ar[r]^{V_\ell(\lambda^2)} &  H^3(V,\Ql(2))\\
V_\ell(\Gr^1_\gamma \K_0(D)) \ar[r]^{V_\ell(\det)}_\sim
\ar[u]_(.45){f_*} &
V_\ell(\Pic(D)))\ar@{.>}[u]
 \ar[r]^{V_\ell(\lambda^1)} & H^1(D,\Ql(1))\rlap{,} \ar[u]_(.45){f_*}
}
\end{aligned}
\end{align}
without the dotted arrow.
We note that the leftmost horizontal arrows of~\eqref{diag:lemmelambda1lambda2bis} are isomorphisms;
 their inverses are induced by the map
$\varphi^2:\CH^2(V) \to \Gr^2_F \K_0(V)$ of~\eqref{CHK}
and by the map $\Pic(D)\to \Gr^1_\gamma \K_0(D)$
which sends the class of a Cartier divisor~$Z$ on~$D$
to the class of~$[\sO_D(Z)]-[\sO_D] \in \widetilde{\K_0}(D)$.

Let us complete this diagram with a dotted arrow induced
by the composition of the canonical map $\Pic(D) \to \CH_{d-1}(D)$ (see \cite[\textsection2.1]{Fulton})
with the push-forward
$f_*:\CH_{d-1}(D)\to \CH_{d-1}(V)$.

When~$D$ is smooth, the right half of~\eqref{diag:lemmelambda1lambda2bis}
commutes by \cite[Proposition~3.3, Proposition~3.6]{Bloch}
and
the left half by the description of the inverses of the horizontal arrows.
Thus~\eqref{diag:lemmelambda1lambda2bis} commutes in this case.

In general, let us
choose a family $(D_j)_{j \in J}$
of smooth projective varieties of pure dimension~$d$,
and for each~$j\in J$, a morphism $\nu_j:D_j\to D$ and an element $n_j \in \Z_\ell$,
such that the equality
of $d$\nobreakdash-cycles with coefficients in~$\Z_\ell$ 
\begin{align}
\label{eq:cycles decomposition D}
[D]=\sum_{j\in J} n_j \nu_{j*}[D_j]
\end{align}
holds.
When $\dim(D)\leq 2$ (which will be the case when we apply the lemma), one can choose
the~$D_j$ to be desingularisations of the irreducible components of~$D^\red$
and the~$n_j$ to be the multiplicities, in~$D$, of these irreducible components.
In arbitrary dimension, such~$D_j$, $\nu_j$ and~$n_j$ exist by the Gabber--de~Jong alteration
theorem \cite[Theorem~2.1]{illusietemkin} applied to the irreducible components of~$D^\red$.

For $j\in J$, let $f_j=f \circ \nu_j:D_j\to V$.  As~$D_j$ is smooth,
we have already seen that
the outer square of~\eqref{diag:lemmelambda1lambda2bis} with~$D$ and~$f$ replaced with~$D_j$ and~$f_j$
commutes.
In order
to show that the outer square of~\eqref{diag:lemmelambda1lambda2bis} itself commutes,
it therefore suffices,
by the contravariant functoriality
of the lower row of~\eqref{diag:lemmelambda1lambda2bis},
to check the equality
$f_* = \sum_{j\in J} n_j\mkern1muf_{j*} \circ \nu_j^*$
 of maps $V_\ell(\Gr^1_\gamma \K_0(D)) \to V_\ell(\Gr^2_F \K_0(V))$
and the same equality of maps $H^1(D,\Ql(1)) \to H^3(V,\Ql(2))$.
Let us set
 $\Gr^F_i\GG_0(D) = F_i\GG_0(D)/F_{i-1}\GG_0(D)$
and denote by
$\kappa_D:\Gr^1_\gamma \K_0(D) \to \Gr^F_{d-1}\GG_0(D)$
 the map induced
by the canonical map $\K_0(D)\to \GG_0(D)$.
Coming back to the definition of~$f_*$ in the two contexts,
we now see that it is enough to check the equalities
\begin{align}
\label{eq:desired equality kappaD}
\kappa_D = \sum_{j\in J} n_j\mkern1mu \nu_{j*} \circ \kappa_{D_j} \circ \nu_j^*:V_\ell(\Gr^1_\gamma \K_0(D)) \to
V_\ell(\Gr^F_{d-1}\GG_0(D))
\end{align}
and
\begin{align}
\label{eq:desired equality kappaDet}
\kappa_D^{\et} = \sum_{j\in J} n_j\mkern1mu \nu_{j*} \circ \kappa_{D_j}^{\et} \circ \nu_j^*:H^1(D,\Ql(1))\to H_{2d-1}(D,\Ql(1-d))
\rlap{.}
\end{align}
The
 $\K_0(D)$\nobreakdash-module structure of~$\GG_0(D)$
induces for any~$i$ a cap product operation
$\Gr_1^\gamma \K_0(D) \times \Gr_i^F\GG_0(D)\to \Gr_{i-1}^F \GG_0(D)$ 
(see \cite[X, Corollaire~1.3.3]{SGA6}).
Letting $[\sO_D]$ denote the class of~$\sO_D$ in~$\Gr_d^F\GG_0(D)$,
we have $\kappa_D(x)= x \cap [\sO_D]$
for any $x \in \Gr^1_\gamma \K_0(D)$;
moreover~\eqref{eq:cycles decomposition D}
implies
the equality $[\sO_D]=\sum_{j\in J} n_j\mkern1mu \nu_{j*}[\sO_{D_j}]$ in $\Gr^F_d\GG_0(D) \otimes_\Z\Zl$
(see \cite[Example~15.1.5]{Fulton}).
In view of the projection formula
\cite[IV, (2.11.1.2)]{SGA6}, we deduce~\eqref{eq:desired equality kappaD}.
Similarly,
the definition of~$\kappa_D^{\et}$,
the equality obtained by applying~$\cl$ to~\eqref{eq:cycles decomposition D}
and the projection formula \cite[Proposition~4.2]{laumon}
together imply~\eqref{eq:desired equality kappaDet}.
\end{proof}

Let us finally start the proof of Proposition~\ref{propppdirfact}.
As~$X$ is smooth over~$k$, the morphism~$h$ gives rise to a push-forward map
$h_*:H^3(X'_{\ok},\Ql(2)) \to H^3(X_{\ok},\Ql(2))$
(see~\eqref{eq:pushforward in cohomology}),
satisfying $h_* \circ h^* = \Id$ on $H^3(X_{\ok},\Ql(2))$,
since
$h_*(h^*x \cap \cl([X'_{\ok}])) = x \cap \cl(h_*[X'_{\ok}])= x \cap \cl([X_{\ok}])$
for $x \in H^3(X_{\ok},\Ql(2))$
(see \cite[Proposition~4.2]{laumon}).
Let $K=\Ker\mkern-1mu\big(h_*:H^3(X'_{\ok},\Ql(2)) \to H^3(X_{\ok},\Ql(2))\big)$.
We obtain a decomposition
\begin{align}
\label{eq:decomposition of h3xp}
H^3(X_{\ok},\Ql(2)) \oplus K \isoto H^3(X'_{\ok},\Ql(2))\rlap{.}
\end{align}
A second decomposition of the right-hand side can be obtained
using
the formula for the \'etale cohomology of the blow-up of a regularly immersed closed subscheme
\cite[Proposition~2.2.2.1]{RiouGysin} applied to the blow-ups appearing in
diagram~(\ref{resindetdiag}).
This yields a canonical isomorphism
\begin{align}
\label{eq:isoh3h1zj}
\bigoplus_{c_j=2} H^1((Z_j)_{\ok},\Ql(1)) \isoto H^3(X'_{\ok},\Ql(2))
\end{align}
even though both~$X'_{\ok}$ and~$(Z_j)_{\ok}$ may fail to be regular.

Let $Z'_j$ denote the normalization of~$(Z_j)_{\ok}^{\red}$
and $\nu_j:Z'_j\to (Z_j)_{\ok}$ the natural morphism.
The normality of~$Z_j$ implies that~$(Z_j)_{\ok}$ is geometrically unibranch
and hence that~$\nu_j$ is universally bijective
(see \cite[Proposition~6.15.6, Proposition~6.15.5]{EGA42}).
We deduce that
 $\nu_j^*:H^1((Z_{j})_{\ok},\Q_\ell(1)) \to H^1(Z'_{j},\Q_\ell(1))$
 is an isomorphism
(see \cite[VIII, Corollaire~1.2]{SGA42}).

Letting $L=V_{\ell}(A(G^0_{\ok}))$,
we now consider the diagram of isomorphisms
\begin{align}
\label{touslesisos}
\begin{aligned}
\xymatrix@R=3ex@C=1.4em{
*!<-1.9ex,0ex>\entrybox{H^3(X_{\ok},\Ql(2))\oplus K}
\ar[r]^(.57)\sim & 
H^3(X'_{\ok},\Ql(2)) &
\displaystyle\smash[b]{\bigoplus_{c_j=2}} H^1((Z_{j})_{\ok},\Q_\ell(1))
\ar[d]^(.45){\wr}_(.45){\nu_j^*}\ar[l]_(.54){\sim} \\
*!<-1.8ex,0ex>\entrybox{V_\ell((\bCH^2_{X/k})^0_{\ok})\oplus L} \ar@<1.75ex>[d]^(.45)\wr
 &  & 
\displaystyle\smash[b]{\bigoplus_{c_j=2}} H^1(Z'_{j},\Q_\ell(1))\ar[d]^(.45){\wr}\\
*!<0ex,0ex>\entrybox{V_\ell(A((\bCH^2_{X/k})^0_{\ok}))\oplus L}
\ar[r]^(.52)\sim &
V_\ell(A((\bCH^2_{X'/k})^0_{\ok})) &
\displaystyle\smash[b]{\bigoplus_{c_j=2}} V_\ell(A((\bPic^0_{Z_j/k})_{\ok}))
\ar[l]_(.52){\sim}\rlap{,}
}
\end{aligned}
\end{align}
\vskip 9pt\noindent{}%
whose upper horizontal arrows are~(\ref{eq:decomposition of h3xp}) and~(\ref{eq:isoh3h1zj}),
whose lower horizontal arrows stem
from (\ref{CH2sumbis}),
whose left vertical isomorphism results from Theorem~\ref{threp}~(ii)
and whose lower right vertical isomorphism
is the composition of the canonical isomorphism
$V_\ell(\Pic(Z'_j)) \isoto V_\ell(A((\bPic^0_{Z_j/k})_{\ok}))$
coming from~\textsection\ref{Pic}
and~\eqref{Chevanorm}
with the Kummer isomorphism
$H^1(Z'_{j},\Q_\ell(1)) \isoto V_\ell(\Pic(Z'_j))$.

Let $m_j$ denote the multiplicity of $(Z_j)_{k_{\p}}$, i.e.\ the length of its
generic local ring, and~$\theta_j$ the canonical principal polarization
of $\bPic^0_{Z_j/k}$ (see \S\ref{Piccurves}).

\begin{lem}
\label{lem:lemma which implies propppdirfact}
\begin{enumerate}[(i)]
\item Diagram (\ref{touslesisos}) transports the opposite of the cup product pairing on
$H^3(X'_{\ok},\Ql(2))$ to the pairing on $\bigoplus_{c_j=2} V_\ell(A((\bPic^0_{Z_j/k})_{\ok}))$ defined as the orthogonal sum of the $\Ql(1)$\nobreakdash-valued Weil pairings associated with the polarizations $m_j\theta_j$.
\item Diagram (\ref{touslesisos}) transports $H^3(X_{\ok},\Ql(2))$ to 
$V_\ell((\bCH^2_{X/k})^0_{\ok})$ and $K$ to $L$.
\item The isomorphism
$H^3(X_{\ok},\Ql(2))=V_\ell((\bCH^2_{X/k})^0_{\ok})$
resulting from~(ii)
coincides with the one induced by Bloch's $\ell$\nobreakdash-adic Abel--Jacobi map
and by the identification between
$V_\ell((\bCH^2_{X/k})^0_{\ok})$  and $V_\ell(\CH^2(X_{\ok}))$
that stems from Theorem~\ref{threp}~(iv), (v).
\end{enumerate}
\end{lem}

\begin{proof}
Let us consider the commutative diagram
\begin{align}
\label{diag:h1h1picpicch2ch2}
\begin{aligned}
\xymatrix@R=3ex@C=2em{
\displaystyle\smash[b]{\bigoplus_{c_j=2}}H^1((Z_j)_{\ok},\Ql(1)) \ar[d]_(.43)\wr \ar[rr]^{\nu_j^*} &&
\displaystyle\smash[b]{\bigoplus_{c_j=2}}H^1(Z'_j,\Ql(1)) \ar[d]^(.43)\wr  \\
\displaystyle\smash[b]{\bigoplus_{c_j=2}}V_\ell(\Pic((Z_j)_{\ok})) \ar@{=}[d] \ar[rr]^{\nu_j^*} &&
\displaystyle\smash[b]{\bigoplus_{c_j=2}}V_\ell(\Pic(Z'_j)) \ar@{=}[d] \\
\displaystyle\smash[b]{\bigoplus_{c_j=2}}V_\ell(\bPic_{Z_j/k}) \ar@{=}[d] &
*!<4ex,0ex>\entrybox{\displaystyle\smash[b]{\bigoplus_{c_j=2}}V_\ell(\bPic^0_{Z_j/k})} \ar@<-4ex>@{=}[d] \ar[r] \ar[l]_(.572)\sim &
\displaystyle\smash[b]{\bigoplus_{c_j=2}}V_\ell(A((\bPic^0_{Z_j/k})_{\ok})) \ar@{=}[d] \\
V_\ell(\bCH^2_{X'/k}) &
*!<3ex,0ex>\entrybox{V_\ell((\bCH^2_{X'/k})^0)} \ar[l]_(.572)\sim \ar[r] &
V_\ell(A((\bCH^2_{X'/k})^0_{\ok}))\rlap{,}
}
\end{aligned}
\end{align}
in which the unlabelled horizontal arrows are the obvious ones
(the bottom leftward
arrow being an isomorphism in view of~\eqref{CH2sumbis}),
the top vertical arrows are the Kummer isomorphisms,
the middle vertical isomorphisms come from~\textsection\ref{Pic}
and~\eqref{Chevanorm}, and the bottom vertical isomorphisms
are induced by~\eqref{CH2sumbis}.

Since the top horizontal arrow of this diagram is an isomorphism,
all of the maps appearing in~\eqref{diag:h1h1picpicch2ch2} have to be isomorphisms.

We note that as a consequence of the projection formula \cite[Proposition~4.2]{laumon}
and of the equality of cycles $[(Z_j)_{\ok}]=m_j \nu_{j*}[Z'_j]$,
the top horizontal isomorphism of~\eqref{diag:h1h1picpicch2ch2} transports
the cup product pairing on $H^1((Z_j)_{\ok},\Ql(1))$
to the cup product pairing on $H^1(Z'_j,\Ql(1))$ multiplied by~$m_j$.

As on the other hand~\eqref{diag:h1h1picpicch2ch2}
transports the Weil pairing
on $V_\ell(A((\bPic^0_{Z_j/k})_{\ok}))=V_\ell(\bPic^0_{Z'_j/{\ok}})$
(see~\eqref{Chevanorm})
to the cup product pairing on $H^1(Z'_j,\Ql(1))$,
we see that
Lemma~\ref{lem:lemma which implies propppdirfact}~(i)
amounts to the assertion that~\eqref{eq:isoh3h1zj} transports
the orthogonal sum of the cup product pairings on $H^1((Z_j)_{\ok},\Ql(1))$
to the opposite of the cup product pairing on $H^3(X'_{\ok},\Ql(2))$.
When the~$Z_j$ are smooth, this is shown in \cite[(2.9)]{CGBW};
the same argument applies in our setting.

Thus, it only remains to prove Lemma~\ref{lem:lemma which implies propppdirfact}~(ii) and~(iii).
For this, it suffices to check the commutativity of the squares
\begin{align}
\label{square:hupperstar}
\begin{aligned}
\xymatrix@R=3ex{
V_\ell(A((\bCH^2_{X'/k})^0_{\ok})) \ar[r]_(.54)\sim^(.54){\gamma'} &  H^3(X'_{\ok},\Ql(2)) \\
\ar[u]^{h^*}  V_\ell(A((\bCH^2_{X/k})^0_{\ok}))\ar[r]_(.54)\sim^(.54){\gamma} & H^3(X_{\ok},\Ql(2)) \ar[u]_{h^*}
}
\end{aligned}
\end{align}
and
\begin{align}
\begin{aligned}
\label{square:hlowerstar}
\xymatrix@R=3ex{
\ar[d]_{h_*}
V_\ell(A((\bCH^2_{X'/k})^0_{\ok})) \ar[r]_(.54)\sim^(.54){\gamma'} &   H^3(X'_{\ok},\Ql(2))
 \ar[d]^{h_*} \\
V_\ell(A((\bCH^2_{X/k})^0_{\ok}))\ar[r]_(.54)\sim^(.54)\gamma & H^3(X_{\ok},\Ql(2))\rlap{,}
}
\end{aligned}
\end{align}
where~$\gamma$ is the isomorphism constructed from Bloch's $\ell$\nobreakdash-adic Abel--Jacobi map
for the smooth variety~$X$ (see the statement of Lemma~\ref{lem:lemma which implies propppdirfact}~(iii))
and~$\gamma'$ is the isomorphism extracted
from~\eqref{eq:isoh3h1zj} and~\eqref{diag:h1h1picpicch2ch2},
and where the vertical arrows are those appearing
in the upper and lower rows of~\eqref{touslesisos}.

The square~\eqref{square:hlowerstar} fits into a larger diagram
\begin{align}
\label{diag:toobig}
\begin{aligned}
\xymatrix@R=3ex@C=.6em{
\displaystyle\smash[b]{\bigoplus_{c_j=2}} V_\ell(\K_0((Z_j)_{\ok})) \ar[d]^{\alpha} \ar[rr]^\det_\sim &&
\displaystyle\smash[b]{\bigoplus_{c_j=2}} V_\ell(\Pic((Z_j)_{\ok})) \ar@{=}[r] \ar@{=}[d] &
\displaystyle\smash[b]{\bigoplus_{c_j=2}} H^1((Z_j)_{\ok},\Ql(1)) \ar@{=}[d] \\
V_\ell(\SK_0(X'_{\ok})) \ar[d]^{h_*} \ar[r] &
*!<.5ex,0ex>\entrybox{V_\ell(\bCH^2_{X'/k})} \ar@<-.5ex>[d]^{h_*} \ar[r]^(.39){\beta'} & V_\ell(A((\bCH^2_{X'/k})^0_{\ok})) \ar[d]^{h_*}
\ar[r]_(.54)\sim^(.54){\gamma'} &   H^3(X'_{\ok},\Ql(2)) \ar[d]^{h_*} \\
V_\ell(\SK_0(X_{\ok})) \ar[r] &
*!<.5ex,0ex>\entrybox{V_\ell(\bCH^2_{X/k})} \ar[r]^(.39){\beta\phantom{'}} & V_\ell(A((\bCH^2_{X/k})^0_{\ok}))
\ar[r]_(.54)\sim^(.54){\gamma} & H^3(X_{\ok},\Ql(2))\rlap{,}
}
\end{aligned}
\end{align}
in which the map~$\alpha$
is induced by~\eqref{eq:decomposition of k0}
(see also~\eqref{SK0blup}),
the map~$\beta'$ comes from the bottom row of~\eqref{diag:h1h1picpicch2ch2},
the map~$\beta$ is constructed in the same way as~$\beta$'
(legitimate thanks to Theorem~\ref{threp}~(v))
and the isomorphisms of the square of the top right corner
all come from~\eqref{eq:isoh3h1zj} and~\eqref{diag:h1h1picpicch2ch2}.

In order for the square in the bottom right corner to commute,
it suffices that the outer
square of the diagram commute, since the other inner squares clearly commute.
That is, fixing~$j$ such that~$c_j=2$ and letting~$\alpha_j$
and~$\alpha^\et_j$ respectively denote the $j$\nobreakdash-th component of~$\alpha$
and of~\eqref{eq:isoh3h1zj}, we need only prove that the square
\begin{align}
\label{square:halphajhalphajet}
\begin{aligned}
\xymatrix@R=3ex{
V_\ell(\K_0((Z_j)_{\ok})) \ar[d]^{h_* \circ\mkern1mu \alpha_j} \ar[r]^(.47)\sim &
H^1((Z_j)_{\ok},\Ql(1)) \ar[d]^{h_* \circ\mkern1mu \alpha_j^\et} \\
V_\ell(\SK_0(X_{\ok})) \ar[r] &
H^3(X_{\ok},\Ql(2))\rlap{,}
}
\end{aligned}
\end{align}
whose horizontal arrows are extracted from~\eqref{diag:toobig}, commutes.

Set $D_j=Z_j \times_{Y_j} X'$
and $E_j=Z_j\times_{Y_j} Y_{j+1}$, where  the $Y_j$ are the varieties appearing in diagram~(\ref{resindetdiag}).
Let $q_j:D_j \to Z_j$ and $p'_j:E_j \to Z_j$ denote
the projections.  Let $\iota_j:Z_j \hookrightarrow Y_j$, $\iota'_j:E_j \hookrightarrow Y_{j+1}$
and $\delta_j:D_j\hookrightarrow X'$ be the inclusions, so that~$\iota_j$, $\iota'_j$,~$\delta_j$ are regular
closed immersions of codimensions~$c_j$,~$1$,~$1$, respectively.
Recall that
$\alpha^\et_j=(p_{j+1} \circ \dots \circ p_N)^* \circ \iota'_{j*} \circ p_j'^*$,
where $\iota'_{j*}$ denotes the map given by
cup product with the class of the Cartier divisor~$(E_j)_{\ok}$
in $H^2_{(E_j)_{\ok}}((Y_{j+1})_{\ok},\Ql(1))$
(see \cite[\textsection2.1]{RiouGysin})
composed with the map
$H^3_{(E_j)_{\ok}}((Y_{j+1})_{\ok},\Ql(2))\to H^3((Y_{j+1})_{\ok},\Ql(2))$
forgetting the support.
As~$E_j$ pulls back, as a Cartier divisor, to~$D_j$, we deduce that
\begin{align}
\label{eq:alphajdeltaqj}
\alpha_j^\et= \delta_{j*}\circ q_j^*\rlap{,}
\end{align}
where~$\delta_{j*}$ is again defined as in \emph{loc.\ cit.}, \textsection2.1.
Similarly, recall that $\alpha_j$ is
given by $x \mapsto (p_{j+1}\circ\dots\circ p_N)^*( \iota'_{j*\mkern1mu}  p_j'^* x \otimes [\sO_{Y_{j+1}}(E_j)])=
((p_{j+1}\circ\dots\circ p_N)^* \iota'_{j*\mkern1mu}  p_j'^* x) \otimes [\sO_{X'}(D_j)])$
(see~\eqref{blowup}).
Noting that
the morphisms $p_{j+1}\circ\dots\circ p_N$ and~$\iota'_j$ are Tor-independent
(indeed one has $\mathrm{Tor}^A_i(A/fA,B)=0$ for any $i>0$, any commutative ring~$A$, any $A$\nobreakdash-algebra~$B$
and any $f \in A$ such that neither~$f$ nor its image in~$B$ is a zero divisor),
the base change theorem \cite[Theorem 3.10.3]{Lipman}
allows us to rewrite this as
\begin{align}
\label{eq:alphajetdeltaqj}
\alpha_j(x)= (\delta_{j*}\circ q_j^*)(x) \otimes [\sO_{X'}(D_j)]
\end{align}
for any
$x\in V_\ell(\K_0((Z_j)_{\ok}))$.

Let $\alpha_j':V_\ell(\K_0((Z_j)_{\ok})) \to V_\ell(\SK_0(X'_{\ok}))$ be given
by $\alpha'_j(x)=(\delta_{j*}\circ q_j^*)(x)$.
In view of~\eqref{eq:alphajdeltaqj}
and of the contravariant functoriality of the first row of~\eqref{diag:toobig},
we deduce from Lemma~\ref{lem:lambda1lambda2} applied to $h\circ \delta_j:D_j\to X$
that the square obtained by
replacing, in~\eqref{square:halphajhalphajet}, the left-hand side vertical arrow
 $h_* \circ \alpha_j$
with $h_* \circ \alpha_j'$ commutes.
On the other hand, it follows from~\eqref{eq:alphajetdeltaqj}
that the map $h_*\circ \alpha_j-h_*\circ \alpha_j'$ takes its values
in $V_\ell(F^3\K_0(X_{\ok}))$,
since $\rk([\sO_{X'}(D_j)]-[\sO_{X'}])=0$
(see \cite[X, Corollaire~1.3.3]{SGA6}).
Now the lower horizontal map of~\eqref{square:halphajhalphajet}
vanishes on $V_\ell(F^3\K_0(X_{\ok}))$ since it factors through~$c_2$;
we conclude that the square~\eqref{square:halphajhalphajet} itself commutes,
and therefore so does~\eqref{square:hlowerstar}.

Let us turn to~\eqref{square:hupperstar}.
We introduce a desingularisation $\pi:X'' \to X'_{\ok}$
of~$X'_{\ok}$
(which exists by Cossart and Piltant \cite{CPII})
and consider the square
\begin{align}
\label{square:hupperstarbis}
\begin{aligned}
\xymatrix@R=3ex{
\ar[d]_(.45){\pi^*}
V_\ell(A((\bCH^2_{X'/{k}})^0_{\ok})) \ar[r]_(.54)\sim^(.54){\gamma'} & H^3(X'_{\ok},\Ql(2))
\ar[d]^(.45){\pi^*}\\
V_\ell(A((\bCH^2_{X''/{\ok}})^0))  \ar[r]_(.54)\sim^(.54){\gamma''} &  H^3(X'',\Ql(2))\rlap{,}
}
\end{aligned}
\end{align}
where~$\gamma''$ is constructed from Bloch's $\ell$\nobreakdash-adic Abel--Jacobi map
for the smooth variety~$X''$ in the same way as~$\gamma$ for~$X$.
One verifies the commutativity of the square~\eqref{square:hupperstarbis}
by proceeding exactly as we did with~\eqref{square:hlowerstar},
that is, by reducing to
Lemma~\ref{lem:lambda1lambda2}
using the diagram
obtained by replacing, in~\eqref{diag:toobig},
all occurrences of~$X$ with~$X''$ and all occurrences
of~$h_*$ with~$\pi^*$,
and using the equalities obtained by replacing, in~\eqref{eq:alphajdeltaqj} and~\eqref{eq:alphajetdeltaqj}
and in their proofs,
$\alpha_j$ and $\alpha_j^\et$ with $\pi^*\circ\alpha_j$ and $\pi^*\circ\alpha_j^\et$,
and
$D_j$, $q_j$, $\delta_j$ with $D_j''$, $q_j''$, $\delta_j''$,
where
 $D_j''=Z_j \times_{Y_j} X''$
and where $q_j'':D_j'' \to Z_j$
and $\delta_j'':D_j''\hookrightarrow X''$ denote the projections.

All the $\ok$-varieties in sight and all the $\ok$-morphisms between them are defined over a common subfield of~$\ok$ that is finitely generated over the prime field. Their $\ell$-adic cohomology groups are thus endowed with a weight filtration (see \cite[\textsection2]{jannsenweights}).
As the~$Z'_j$ are smooth and projective,
the groups $H^1(Z'_j,\Ql(1))$
are pure of weight~$-1$.
In view of the isomorphisms~\eqref{touslesisos},
we deduce that the group $H^3(X'_{\ok},\Ql(2))$ is pure of weight~$-1$ as well.
On the other hand,
the kernel of the right-hand side vertical map of~\eqref{square:hupperstarbis} has weights~$<-1$,
as follows from cohomological descent and Deligne's theorem on the Weil conjectures
(\emph{op.\ cit.}, \textsection9; proper smooth hypercoverings
of~$X'_{\ok}$ that start with~$\pi$ exist by \cite[Theorem~4.1]{dJ}).
Hence the right-hand side vertical map of~\eqref{square:hupperstarbis} is injective.

This injectivity, the commutativity of~\eqref{square:hupperstarbis}
and the commutativity of the square
\begin{align}
\begin{aligned}
\xymatrix@R=3ex{
V_\ell(A((\bCH^2_{X''/\ok})^0)) \ar[r]_(.54)\sim^(.54){\gamma''} &  H^3(X'',\Ql(2)) \\
\ar[u]^{\pi^* \circ h^*}  V_\ell(A((\bCH^2_{X/k})^0_{\ok}))\ar[r]_(.54)\sim^(.54){\gamma} & H^3(X_{\ok},\Ql(2)) \ar[u]_{\pi^*\circ h^*}
}
\end{aligned}
\end{align}
(which holds by functoriality of Bloch's Abel--Jacobi map \cite[Proposition~3.5]{Bloch})
together imply that~\eqref{square:hupperstar} commutes.
This concludes the proof of Lemma~\ref{lem:lemma which implies propppdirfact}.
\end{proof}

We resume the proof of Proposition \ref{propppdirfact}.  Consider the diagram
\begin{equation}
\label{isova}
 (\bCH^2_{X/k})^0_{\ok}\times A(G^0_{\ok})
\isoto A((\bCH^2_{X'/k})^0_{\ok})\isofrom
\prod_{c_j=2} A((\bPic^0_{Z_j/k})_{\ok})
\end{equation}
of isomorphisms of abelian varieties stemming from (\ref{CH2sumbis}) and whose $\ell$-adic Tate modules appear on the bottom line of (\ref{touslesisos}).
The product of the polarizations $m_j\theta_j$ on the right-hand side of (\ref{isova}) induces a polarization $\lambda$ on the left-hand side $(\bCH^2_{X/k})^0_{\ok}\times A(G^0_{\ok})$ of (\ref{isova}). Let us view the Weil pairing of $\lambda$ as a $\Q_{\ell}(1)$\nobreakdash-valued pairing on $H^3(X'_{\ok},\Ql(2))$
thanks to (\ref{touslesisos}). By
Lemma~\ref{lem:lemma which implies propppdirfact}~(i), it is equal to the opposite of the cup product pairing on $H^3(X'_{\ok},\Ql(2))$. Since,
by the projection formula \cite[Proposition~4.2]{laumon},
the decomposition~(\ref{eq:decomposition of h3xp}) is orthogonal with respect to the cup product,
it follows from Lemma \ref{lem:lemma which implies propppdirfact} (ii) that $\lambda$ is a product polarization on $(\bCH^2_{X/k})^0_{\ok}\times A(G^0_{\ok})$.
Since the restriction of the cup product pairing on $H^3(X'_{\ok},\Ql(2))$ to $H^3(X_{\ok},\Ql(2))$ coincides with the cup product pairing on $H^3(X_{\ok},\Ql(2))$, it follows from Lemma \ref{lem:lemma which implies propppdirfact} (iii) that the restriction of $\lambda$ to $(\bCH^2_{X/k})^0_{\ok}$ is the canonical principal polarization defined in Theorem \ref{threp} (vi).

By a theorem of Debarre \cite[Corollary 2]{Debarre}, polarized abelian varieties can be written in a unique way as a product of indecomposable polarized abelian varieties. As the $(A((\bPic^0_{Z_j/k})_{\ok}),m_j\theta_j)$ are indecomposable or trivial (since so are the $(A((\bPic^0_{Z_j/k})_{\ok}),\theta_j)$), we deduce the existence of a partition $\{j|\mkern2muc_j=2\}=J\sqcup J'$ such that (\ref{isova}) induces isomorphisms $\prod_{j\in J}A((\bPic^0_{Z_j/k})_{\ok})\isoto (\bCH^2_{X/k})^0_{\ok}$ and $\prod_{j\in J'}A((\bPic^0_{Z_j/k})_{\ok})\isoto A(G^0_{\ok})$. Since $\lambda$ restricts to a principal polarization on $(\bCH^2_{X/k})^0_{\ok}$, we see that $m_j=1$ for all the $j\in J$ such that $A((\bPic^0_{Z_j/k})_{\ok})$ is non-zero.

Thus, the product of the polarizations $\theta_j$ on the right-hand side of (\ref{isova}) induces 
on $(\bCH^2_{X/k})^0_{\ok}\times A(G^0_{\ok})$
a polarization  which is at the same time a principal polarization
and the product of two polarizations,
and which is therefore the product of two principal polarizations;
moreover, the first of these coincides with
the canonical principal polarization of Theorem~\ref{threp}~(vi).
Proposition~\ref{propppdirfact} is proved.
\end{proof}

Now that Proposition \ref{propppdirfact} is proved, we let $J_1$ (resp.\ $J_2$, resp.\ $J_3$) be the set of indices $j$ such that $c_j=2$ and the curve $Z_j$ is smooth over $k$ (resp.\ such that $c_j=2$ and $Z_j$ is not smooth over $k$, resp.\ such that $c_j=3$), and we proceed to show that the $(Z_j)_{j\in J_2}$ do not contribute to $(\bCH^2_{X/k})^0$.

\begin{lem}
\label{inducedmapzero}
The map $$\prod_{j \in J_2} \bPic^0_{Z_j/k} \to
(\bCH^2_{X/k})^0$$ induced by~(\ref{CH2sumbis}) vanishes.
\end{lem}

\begin{proof}
We fix $j\in J_2$.
Let $a:\bPic^0_{Z_j/k}\to(\bCH^2_{X/k})^0$
and $b:\bPic^0_{Z_j/k}\to G^0$ denote the maps
induced by~\eqref{CH2sumbis}.
Let us assume that $a\neq 0$ and derive a contradiction.
When $a\neq 0$, we claim that
$\bPic^0_{Z_j/k} = \Ker(a)\times\Ker(b)$,
that~$\Ker(a)$ is affine and that~$\Ker(b)$
is a non-trivial abelian variety;
Corollary~\ref{corsplit} then provides the desired contradiction.
It thus suffices to prove the claim.
To this end, it is enough to check
that
$(\bPic^0_{Z_j/k})_{k_{\p}} = \Ker(a_{k_{\p}})\times\Ker(b_{k_{\p}})$,
that~$\Ker(a_{k_{\p}})$ is affine and that~$\Ker(b_{k_{\p}})$
is a non-trivial abelian variety, as these three properties descend to~$k$.

By functoriality, the map~$b$ induces maps $A(b_{k_{\p}}):A((\bPic^0_{Z_j/k})_{k_{\p}}) \to A(G^0_{k_{\p}})$
and
 $L(b_{k_{\p}}):L((\bPic^0_{Z_j/k})_{k_{\p}}) \to L(G^0_{k_{\p}})$.
As $(\bCH^2_{X/k})^0_{k_{\p}}$ is an abelian variety
(see Theorem~\ref{threp}~(ii)),
we have $L(a_{k_{\p}})=0$
and the map $A(a_{k_{\p}})$
can be viewed as a map
 $A(a_{k_{\p}}):A((\bPic^0_{Z_j/k})_{k_{\p}}) \to (\bCH^2_{X/k})^0_{k_{\p}}$
through which $a_{k_{\p}}$ factors,
so that our assumption that $a_{k_{\p}}\neq 0$ implies that
$A(a_{k_{\p}})\neq 0$.
On the other hand, the map $L(b_{k_{\p}})$
is a closed immersion
since~$a\times b$
is one
and $L(a_{k_{\p}})=0$.

Proposition~\ref{propppdirfact}
allows us to view
$A((\bPic^0_{Z_j/k})_{k_{\p}})$ as a principally polarized direct factor of the
product of
principally polarized abelian varieties
$(\bCH^2_{X/k})^0_{k_{\p}} \times A(G^0_{k_{\p}})$
over~$k_{\p}$,
through $A(a_{k_{\p}}) \times A(b_{k_{\p}})$.
As the decomposition of a principally polarized abelian variety
into its indecomposable factors is unique, as
$A((\bPic^0_{Z_j/k})_{k_{\p}})$ is itself indecomposable (see~\S\ref{Piccurves}),
and as $A(a_{k_{\p}})\neq 0$, necessarily
$A(b_{k_{\p}})=0$ and $A(a_{k_\p})$ is a closed
immersion,
so that $\Ker(a_{k_{\p}})=L((\bPic^0_{Z_j/k})_{k_{\p}})$
(see~(\ref{Chevalley})).

All in all, the
exact sequences~(\ref{Chevalley}) fit into a commutative diagram
\begin{equation}
\label{diagChevs}
\begin{aligned}
\xymatrix@C=1.5em@R=3ex{
0\ar[r]& \Ker(a_{k_{\p}})\ar[r]
\ar@{^{(}->}+<0pt,-12pt>;[d]
 &(\bPic^0_{Z_j/k})_{k_{\p}}\ar[r]\ar[d]^{b_{k_{\p}}} & A((\bPic^0_{Z_j/k})_{k_{\p}})\ar[d]^0  \ar[r]&0\\
0\ar[r]&L(G^0_{k_{\p}})\ar[r]& G^0_{k_{\p}} \ar[r]& A(G^0_{k_{\p}})\ar[r]&0\rlap{.}
}
\end{aligned}
\end{equation}
An isomorphism $\Ker(b_{k_{\p}})\isoto  A((\bPic^0_{Z_j/k})_{k_{\p}})$
and then all of the desired statements now result from this diagram,
in view of the remark that the snake homomorphism is trivial since
it goes from an abelian variety
to an affine group.
\end{proof}

\begin{rem}
At the end of the proof of Proposition~\ref{propppdirfact} (penultimate
paragraph), we had seen, as a by-product of~\cite{Debarre},
that for any $j\in \{1,\dots,N\}$ such that $c_j=2$,
if the map $A((\bPic^0_{Z_j/k})_{\ok}) \to (\bCH^2_{X/k})^0_{\ok}$
induced by~(\ref{CH2sumbis}) does not vanish,
then the curve~$Z_j$ is geometrically reduced.  
Lemma~\ref{inducedmapzero} reproves this fact by other means, in the course of showing that~$Z_j$ is even
smooth.
Instead of reproving it, it would have been possible to use the geometric reducedness of~$Z_j$ to simplify
the proof of Lemma~\ref{inducedmapzero},
inasmuch as Corollary~\ref{corsplit} could have been applied over~$k_{\p}$.
\end{rem}

\begin{proof}[Proof of Theorem \ref{threp} (vii)]
In view of Lemma \ref{inducedmapzero}, we may consider the quotient $H$ of $G$ by its subgroup scheme $\prod_{j\in J_2} \bPic^0_{Z_j/k}$. Thanks to the exact sequences $0\to \bPic^0_{Z_j/k}\to \bPic_{Z_j/k}\to \Z_{Z_j/k}\to 0$ given by Proposition~\ref{NS} for $j\in J_2$,
we deduce from (\ref{CH2sumbis}) an isomorphism
\begin{equation}
\label{CH2sumter}
 \bCH^2_{X/k}\times\mkern4muH
\isofrom
\Z\times \prod_{j\in J_1} \bPic_{Z_j/k}\times \prod_{j\in J_2\cup J_3}\Z_{Z_j/k}\rlap{.}
\end{equation}
Let $B$ be the disjoint union of $\bP^1_k$, of the curves $\bP^1_{\pi_0(Z_j/k)}$ for all $j\in J_2\cup J_3$, and of the curves $Z_j$ for all $j\in J_1$. It is a smooth projective curve over $k$, and it has the property that 
$\bPic_{B/k}\simeq\Z\times \prod_{j\in J_1} \bPic_{Z_j/k}\times \prod_{j\in J_2\cup J_3}\Z_{Z_j/k}$ since $\bPic_{\bP^1_{\pi_0(Z_j/k)}/k}\simeq \Res_{\pi_0(Z_j/k)/k}(\Z)\simeq \Z_{Z_j/k}$ for $j\in J_2\cup J_3$ by Proposition \ref{locconsrep}.
The isomorphism $\bPic_{B/k}\isoto \bCH^2_{X/k}\times\mkern4muH$
deduced from (\ref{CH2sumter}) realizes $\bCH^2_{X/k}$ as a principally polarized direct factor of $\bPic_{B/k}$ by Proposition~\ref{propppdirfact}, as desired.
\end{proof}

The proof of Theorem \ref{threp} is now complete.

\subsection{Blow-ups}
\label{parblowup}

The next proposition, which relies on arguments already used in the proof of Theorem \ref{threp}, allows one to compute $\bCH^2_{X/k}$ in concrete situations.

\begin{prop}
\label{calculblowup}
Let $X$ be a smooth projective $\ok$-rational threefold over $k$, let $i:Y\to X$ be the inclusion of a smooth closed subvariety of pure codimension $c$, and let $p:X'\to X$ be its blow-up. Then the formula (\ref{blowupfunct}) induces an isomorphism
\begin{equation}
\label{decblowup}
\begin{alignedat}{4}
\bCH^2_{X/k}\times\bPic_{Y/k}\isoto\bCH^2_{X'/k}\hspace{2em}&\textrm{if $c=2$}\\
\textrm{(resp.}\hspace{1em} \bCH^2_{X/k}\times\mkern2.5mu \Z_{Y/k}\isoto\bCH^2_{X'/k}\hspace{2em}&\textrm{if $c=3$),}
\end{alignedat}
\end{equation}
respecting the principal polarizations furnished
by Theorem~\ref{threp} (applied to~$X$ and to~$X'$)
and by~\textsection\ref{Piccurves} (applied to~$Y$).
\end{prop}

\begin{proof}
The isomorphism~(\ref{blowupfunct}), Corollary~\ref{corblup} and Remark~\ref{isoloc} (ii) yield canonical isomorphisms $\K_{0,X/k}\times \K_{0,Y/k}^{c-1}\isoto \K_{0,X'/k}$, $\bPic_{X/k}\times\mkern2.5mu\Z_{Y/k}\isoto\bPic_{X'/k}$ and $\Z_{X/k}\isoto\Z_{X'/k}$.
Identifying $(\rk,\det):\K_{0,X'/k}\to \Z_{X'/k}\times\Pic_{X'/k}$ in terms of these decompositions and using Proposition \ref{K0curve}, we obtain an isomorphism
\begin{equation}
\label{decompoblowup}
\begin{alignedat}{4}
\SK_{0,X/k,\fppf}\times\Pic_{Y/k,\fppf}\isoto \SK_{0,X'/k,\fppf}\hspace{2em}&\textrm{if $c=2$}\\
\textrm{(resp.}\hspace{2.5em} \SK_{0,X/k,\fppf}\times\Z_{Y/k}\isoto \SK_{0,X'/k,\fppf}\hspace{2em}&\textrm{if $c=3$).}
\end{alignedat}
\end{equation}
If $l/k$ is a finite extension and $x\in(X\setminus Y)(l)$, one has $p^*[\sO_{x}]=[\sO_{p^{-1}(x)}]\in \K_0(X'_{l})$. It follows that $p^*\nu_X(1)=\nu_{X'}(1)\in \SK_{0,X'/k,\fppf}(k)\subseteq \SK_{0,X'/k,\fppf}(l)$. We thus deduce from (\ref{decompoblowup}) the required isomorphism (\ref{decblowup}) of $k$-group schemes.

If $c=2$, considering the commutative diagram
\begin{equation}
\label{diagcomparCH2Ab}
\begin{aligned}
\xymatrix
@R=0.4cm
@C=1.1cm
{
\mkern20mu\Ab^2(X_{k_\p})\times\bPic^0_{Y_{k_\p}/k_\p}\ar^(.60){\sim}[r]\ar@<.6em>[d]^(.5){\wr} &\Ab^2(X'_{k_\p})\ar[d]^{\wr}  \\
(\bCH^2_{X_{k_\p}/k_{\p}})^0\times\bPic^0_{Y_{k_\p}/k_\p}\ar^(.60){\sim}[r]&(\bCH^2_{X'_{k_\p}/k_{\p}})^0\rlap{}
}
\end{aligned}
\end{equation}
whose vertical arrows stem from Theorem \ref{threp} (vi), whose lower horizontal arrow is the above constructed isomorphism and whose upper horizontal arrow is that of \cite[Lemma 2.10]{CGBW} concludes the proof, as the latter arrow respects the principal polarizations by \cite[Lemma 2.10]{CGBW}. If $c=3$, one can argue in the same way, using a diagram similar to (\ref{diagcomparCH2Ab}) in which $\bPic^0_{Y_{k_\p}/k_\p}$ does not appear.
\end{proof}

\subsection{Obstructions to \texorpdfstring{$k$-rationality}{k-rationality}}
\label{parobs}

The most general obstruction to the $k$\nobreakdash-ratio\-nality of a smooth projective $\ok$-rational threefold obtained in this article is Theorem~\ref{threp}~(vii). In this subsection, we spell out concrete consequences of this theorem.

 We recall that a $\Gamma_k$-module $M$ is a \textit{permutation} $\Gamma_k$-module if it is free of finite rank as a $\Z$-module and admits a $\Z$-basis that is permuted by the action of $\Gamma_k$, and that it is \textit{stably of permutation} if there exists a $\Gamma_k$\nobreakdash-equivariant isomorphism $M\oplus N_1\simeq N_2$ for some permutation $\Gamma_k$-modules $N_1$ and $N_2$.

If $X$ is a smooth projective $\ok$-rational threefold, we associate with any class $\alpha\in \NS^2(X_{\ok})^{\Gamma_k}=(\bCH^2_{X/k}/(\bCH^2_{X/k})^0)(k)$ (see Theorem \ref{threp} (v)) its inverse image $(\bCH^2_{X/k})^{\alpha}$ in $\bCH^2_{X/k}$. It is an fppf torsor under $(\bCH^2_{X/k})^{0}$, hence an \'etale torsor under $(\bCH^2_{X/k})^{0}$ by \cite[III, Corollary 4.7, Remark 4.8 (a)]{Milne}. We let $[(\bCH^2_{X/k})^{\alpha}]\in H^1(k, (\bCH^2_{X/k})^{0})$ be its Galois cohomology class.

\begin{thm}
\label{listobstructions}
Let $X$ be a smooth projective $k$-rational threefold over $k$. Then:
\begin{enumerate}[(i)]
\item The $\Gamma_k$-module $\NS^2(X_{\ok})$ is a direct factor of a permutation $\Gamma_k$-module.
\item There exists an isomorphism
$(\bCH^2_{X/k})^0\simeq\bPic^0_{C/k}$ of principally polarized abelian varieties over $k$ for some smooth projective curve $C$ over $k$.
\item 
For all smooth projective geometrically connected curves $D$ of genus~$\geq2$ over~$k$, all morphisms 
$\psi:(\bCH^2_{X/k})^0\to\bPic^0_{D/k}$ identifying $\bPic^0_{D/k}$ with a principally polarized direct factor of $(\bCH^2_{X/k})^0$ and all $\alpha\in \NS^2(X_{\ok})^{\Gamma_k}$, there exists $d\in\Z$ such that $\psi_*[(\bCH^2_{X/k})^\alpha]=[\bPic^d_{D/k}]\in H^1(k, \bPic^0_{D/k})$. 
\item 
For all elliptic curves $E$ over~$k$ and all morphisms 
$\psi:(\bCH^2_{X/k})^0\to E$ identifying $E$ with a principally polarized direct factor of $(\bCH^2_{X/k})^0$, there exists a class $\eta\in H^1(k,E)$ such that for all $\alpha\in \NS^2(X_{\ok})^{\Gamma_k}$, there exists $d\in\Z$ with $\psi_*[(\bCH^2_{X/k})^\alpha]=d\eta\in H^1(k, E)$. 
\end{enumerate}
\end{thm}

\begin{proof}
Theorem \ref{threp} (vii) shows the existence of a smooth projective curve $B$ over~$k$ such that $\bCH^2_{X/k}$ is a principally polarized direct factor of $\bPic_{B/k}$. We denote by $q:\bPic_{B/k}\to \bCH^2_{X/k}$ the projection onto this direct factor and by $r:\bCH^2_{X/k}\to \bPic_{B/k}$ the inclusion of this direct factor.

Passing to the groups of connected components shows that $\NS^2(X_{\ok})$ is a direct factor of $\Z_{B/k}(\ok)$, which is a permutation $\Gamma_k$\nobreakdash-module, thus proving~(i).

Passing to the identity components shows that $(\bCH^2_{X/k})^0$ is a principally polarized direct factor of $\bPic^0_{B/k}$. In view of the uniqueness of the decomposition of a principally polarized abelian variety as a product of indecomposable ones, and in view of the description of the indecomposable factors of $\bPic^0_{B/k}$ (see \cite[\S 2.1]{CGBW}), there exists a union $C$ of connected components of $B$ such that $(\bCH^2_{X/k})^0\simeq\bPic^0_{C/k}$ as principally polarized abelian varieties over $k$, thus proving (ii).

Let us now fix~$D$, $\psi$ and~$\alpha$ as in~(iii).
The composition $p:=\psi\circ q^0:\bPic^0_{B/k}\to \bPic^0_{D/k}$ realizes $\bPic^0_{D/k}$ as a principally polarized direct factor of $\bPic^0_{B/k}$.
All indecomposable principally polarized direct factors of $\bPic^0_{B/k}$ are of the form $\bPic^0_{B'/k}$ for some connected component $B'$ of $B$ (as the connected components of $B$ are in bijection with those of $B_{k_{\p}}$, this statement reduces to the case of a perfect field, for which see \cite[\S 2.1]{CGBW}). Let $B'$ be the connected component of $B$ corresponding to $\bPic^0_{D/k}$. Since $D$ is geometrically connected of genus $\geq 2$, we see that $\bPic^0_{D/k}$, hence also $\bPic^0_{B'/k}$, is geometrically indecomposable of dimension~$\geq 2$, and it follows that $B'$ is geometrically connected of genus $\geq 2$. 
By the precise form of the Torelli theorem \cite[Th\'eor\`emes~1~et~2]{Serre}, after possibly replacing~$q$ with $-q$ (and $p$ with $-p$) if $D$ is not hyperelliptic, we can identify $D$ and $B'$ in such a way that
 $p$ is the pull-back by the inclusion \mbox{$i:D\simeq B'\hookrightarrow B$}.
Since $q: \bPic_{B/k}\to \bCH^2_{X/k}$ realizes $\bCH^2_{X/k}$ as a direct factor of $\bPic_{B/k}$, one can find $\beta\in\NS^1(B_{\ok})^{\Gamma_k}$ with $q(\beta)=\alpha\in \NS^2(X_{\ok})^{\Gamma_k}$. Letting 
$d:=i^*\beta\in \NS^1(D_{\ok})\simeq\Z$,
we obtain 
$\psi_*[(\bCH^2_{X/k})^\alpha]=p_*[\bPic^{\beta}_{B/k}]=[\bPic^d_{D/k}]\in H^1(k, \bPic^0_{D/k})$, which proves~(iii).

Fix $E$ and $\psi$ as in (iv). Arguing as above shows that $p:=\psi\circ q^0:\bPic^0_{B/k}\to E$ identifies $E$ with $\bPic^0_{B'/k}$ for some connected component $i:B'\hookrightarrow B$ of $B$ which is geometrically connected of genus $1$. The genus $1$ curve $B'$ has a natural structure of $\bPic^0_{B'/k}$-torsor, and we set $\eta:=[B']\in H^1(k,\bPic^0_{B'/k})=H^1(k,E)$ be its class. 
Let $\alpha\in \NS^2(X_{\ok})^{\Gamma_k}$.
Setting $s:=i^*\circ r: \bCH^2_{X/k}\to\bPic_{B'/k}$ and $d:=s_*\alpha\in \NS^1(B'_{\ok})\simeq \Z$, we get
$\psi_*[(\bCH^2_{X/k})^\alpha]=s_*[(\bCH^2_{X/k})^\alpha]=[\bPic^d_{B'/k}]=d\eta\in H^1(k, \bPic^0_{B'/k})=H^1(k,E)$, which completes the proof of (iv).
\end{proof}

\begin{rems}
\label{remobs}
(i) If $X$ is a smooth projective $k$-rational variety and if $k$ has characteristic $0$, one can apply the weak factorization theorem \cite[Theorem~0.3.1]{AKMW}
to show that the $\Gamma_k$-module $\NS^2(X_{\ok})$ is stably of permutation. This statement is stronger than Theorem \ref{listobstructions} (i).  We do not know if it holds if $k$ has characteristic~$p>0$ and $X$ has dimension $\geq 3$. We do not know either whether Theorem~\ref{listobstructions}~(i) holds for smooth projective $k$-rational varieties of dimension $\geq 4$.

(ii) The geometric N\'eron--Severi group $\NS^1(X_{\ok})$ of a smooth projective $k$\nobreakdash-rational variety $X$ is stably of permutation (see  \cite[Theorem 2.2]{Maninperfect} if $X$ is a surface and  \cite[Proposition 2.A.1]{descente2} in general).
Theorem~\ref{listobstructions}~(i) and Remark \ref{remobs} (i) may be viewed as analogues of this classical statement for codimension $2$ cycles.

(iii) To obtain a variant of Theorem \ref{listobstructions} (iii) in the case where $D$ is connected but not geometrically connected, one can apply Theorem \ref{listobstructions} (iii) to the $l$-rational variety $X_l$ over $l$, where $l$ is the algebraic extension $l:=H^0(D,\sO_D)$ of $k$.  The same remark applies to Theorem~\ref{listobstructions}~(iv).
\end{rems}

\section{Smooth complete intersections of two quadrics}
\label{sec4}

In this last section, we apply the above results to $k$-varieties $X$ that are three-dimensional smooth complete intersections of two quadrics, computing the variety $\bCH^2_{X/k}$ (in Theorem \ref{thintjac}) and providing a necessary and sufficient criterion for their $k$\nobreakdash-rationality (in Theorem \ref{thrat}). A more classical necessary and sufficient criterion for their (separable) $k$-unirationality (Theorem \ref{thunirat}) allows us to give examples, for any algebraically closed field~$\kappa$, of such varieties over $\kappa(\!(t)\!)$ that are separably $\kappa(\!(t)\!)$\nobreakdash-unirational but not $\kappa(\!(t)\!)$-rational (Theorem \ref{thmC((t))}).

Many of the geometric results that we need along the way are available in the literature only in characteristic $0$ or in characteristic~$\neq 2$ \cite{Reid,Donagi,Wang}, and we extend them to arbitrary characteristic.

\subsection{Lines in a complete intersection of two quadrics}
\label{subseclines}

Let $X\subset\bP^5_k$ be a three-dimensional smooth complete intersection of two quadrics over $k$. We let $F$ be the Hilbert scheme of lines in $X$ (also called the Fano variety of lines of $X$).

\begin{lem}
\label{lemdroites}
The following assertions hold:
\begin{enumerate}[(i)]
\item The variety $X$ does not contain any plane.
\item The normal bundle $N_{L/X}$ of a line $L\subset X$ is isomorphic either to $\sO_L^{\oplus 2}$ or to $\sO_L(1)\oplus\sO_L(-1)$.
\item The variety $F$ is a non-empty geometrically connected smooth projective surface with trivial canonical bundle. Its tangent space at a $k$\nobreakdash-point corresponding to a line $L\subset X$ is naturally isomorphic to $H^0(X,N_{L/X})$.
\end{enumerate}
\end{lem}

\begin{proof}[Proof of~(i)]
By the Lefschetz hyperplane theorem \cite[XII, Corollary 3.7]{SGA2}, $\Pic(X)$ is generated by $\sO_X(1)$. If $X$ contained a plane $P$, we would have $\sO_X(P)\simeq\sO_X(l)$ for some $l\in\Z$, hence an equality of intersection numbers
$$1=\sO_P(1)\cdot\sO_P(1)=\sO_X(P)\cdot\sO_X(1)\cdot\sO_X(1)=4l\rlap{,}$$
which is a contradiction.
\let\qed\relax\end{proof}
\begin{proof}[Proof of~(ii)]
Since $L$ and $X$ are complete intersections in $\bP^5_k$, the normal exact sequence  
$0\to N_{L/X}\to N_{L/\bP^5_k}\to N_{X/\bP^5_k}|_L\to 0$
of the inclusions $L\subset X\subset \bP^5_k$ reads:
\begin{equation}
\label{normal}
0\to N_{L/X}\to \sO_L(1)^{\oplus 4}\to \sO_L(2)^{\oplus 2}\to 0.
\end{equation}
It follows that $N_{L/X}$ is a rank $2$ vector bundle of degree $0$ on $L$, hence is of the form $\sO_L(l)\oplus\sO_L(-l)$ for some $l\geq 0$ (see \cite{HazeM}). Since it admits an injective morphism to $\sO_L(1)^{\oplus 4}$, one has $l\in\{0,1\}$.
\let\qed\relax\end{proof}
\begin{proof}[Proof of~(iii)]
The computation of the tangent space is \cite[Chapter I, Theorem~2.8.1]{Kollarbook}. To prove that $F$ is smooth, geometrically connected and non-empty, we may work over $\ok$. For all $L\in F(\ok)$, one has 
$h^1(L,N_{L/X_{\ok}})=0$ and $h^0(L,N_{L/X_{\ok}})=2$ by (ii). The variety $F_{\ok}$ is thus smooth of dimension $2$ at $L$ by \cite[Chapter~I, Theorem~2.8.3]{Kollarbook}. That $F_{\ok}$ is non-empty and connected follows from \cite[Th\'eor\`eme~2.1~b)~c)]{DeMa}. To compute the canonical bundle of $F$, we let $G$ be the Grassmannian of lines in $\bP^5_{\ok}$ and $0\to\sS\to\sO_G^{\oplus 6}\to\sQ\to 0$ be the short exact sequence of tautological bundles on $G$, where $\sS$ and $\sQ$ respectively have rank $4$ and rank $2$. The variety of lines $F$ is defined in $G$ by the vanishing of a section of $(\Sym^2\sQ)^{\oplus 2}$. Since $F$ is smooth of the expected dimension, the normal short exact sequence reads $0\to T_F\to T_G|_F\to N_{F/G}\to 0$, where $N_{F/G}\simeq (\Sym^2\sQ)^{\oplus 2}|_F$ and $T_G|_F\simeq(\sS^{\vee}\otimes\sQ)|_F$. We deduce that
$$K_F\simeq(\det((\Sym^2\sQ)^{\oplus 2})\otimes\det(\sS\otimes\sQ^{\vee}))|_F\simeq\det(\sQ)|_F^{\otimes 6}\otimes\det(\sQ)|_F^{\otimes -6}\simeq\sO_F\rlap{.}\eqno\qed$$
\let\qed\relax
\end{proof}

We will later show that $F_{\ok}$ is actually an abelian surface (see Theorem \ref{thintjac}).

Let $Z\subset X\times F$ be the universal line in $X$. If $\Lambda\subset X$ is a line, the second projection $Z\cap( \Lambda\times(F\setminus \{\Lambda\}))\to F\setminus \{\Lambda\}$ is a closed immersion by \cite[Proposition~8.11.5]{EGA43}.
Its image is a subscheme $W(\Lambda)\subseteq F\setminus \{\Lambda\}$ parametrizing the lines of $X$ that are distinct from $\Lambda$ and that intersect $\Lambda$. 
If $L\in W(\Lambda)(k)$ and $x=L\cap\Lambda$, then $T_L W(\Lambda)\subseteq T_LF=H^0(X,N_{L/X})$ is the subset of those $\sigma \in H^0(X,N_{L/X})$ such that $\sigma_x\in\langle T_xL,T_x\Lambda\rangle/T_xL$
(as can be seen using \cite[Remark~4.5.4~(ii)]{sernesi}). 

\begin{lem}
\label{lemdeuxdroites}
If $L$ and $\Lambda$ are two distinct lines in $X$ that intersect, the inclusion of tangent spaces $T_LW(\Lambda)\subseteq T_LF$ is strict.
\end{lem}

\begin{proof}
 In view of Lemma \ref{lemdroites} (ii), we may distinguish two cases according to the isomorphism class of the normal bundle $N_{L/X}$. Suppose first that $N_{L/X}\simeq\sO_L^{\oplus 2}$. Consider the point $x=L\cap \Lambda$.
Since $N_{L/X}$ is globally generated, we can choose a section $\sigma\in H^0(L,N_{L/X})$ such that $\sigma_x\notin\langle T_xL, T_x\Lambda\rangle/T_xL$. Then, the tangent vector of $F$ at $L$ associated to $\sigma$ by Lemma \ref{lemdroites} (iii) is not tangent to $W(\Lambda)$.

Assume now that $N_{L/X}\simeq\sO_L(1)\oplus\sO_L(-1)$. Choose homogeneous coordinates $X_0,\dots,X_5$ on $\bP^5_k$ such that $L=\{X_0=\dots=X_3=0\}$ and use $X_0,\dots,X_3$ to identify $N_{L/\bP^5_k}$ with $\sO_L(1)^{\oplus 4}$. After a coordinate change, we may assume that the composition
$\sO_L(1)\to\sO_L(1)\oplus\sO_L(-1)\to\sO_L(1)^{\oplus 4}$ of the inclusion of the first factor and of the first arrow of (\ref{normal}) is the inclusion of the first factor. In such coordinates, $\{X_0+\varepsilon X_4=X_1=X_2=X_3=0\}$ and $\{X_0+\varepsilon X_5=X_1=X_2=X_3=0\}$ are two $k[\varepsilon]/(\varepsilon^2)$\nobreakdash-points  of~$F$. 
Assuming for contradiction that $T_LW(\Lambda)=T_LF$ and using
the characterization recalled above of $T_LW(\Lambda)$ viewed as a subspace of $H^0(X,N_{L/X})$,
we see that $\Lambda\subset P:=\{X_1=X_2=X_3=0\}$.
Since $X$ contains $L$, the monomials $X_4^2$, $X_4X_5$ and $X_5^2$ do not appear in the equations of $X$. Since $X$ also contains the above two infinitesimal deformations of $L$, neither do the monomials $X_0X_4$ and $X_0X_5$. It follows that the intersection of $X$ with the plane $P$ is equal either to $P$ or to the double line $\{X_0^2=0\}\subset P$. Since $\Lambda\subset X\cap P$ but $\Lambda\neq L$, we deduce that $P\subset X$, which contradicts Lemma \ref{lemdroites} (i).
\end{proof}

\subsection{Projecting from a line}
\label{subsecproj}

We keep the notation of \S\ref{subseclines}.

Assume that $X$ contains a line $\Lambda\subset X$, which we fix. We denote by $\mu:X'\to X$ and $\widetilde{\mu}:(\bP^5_k)'\to\bP^5_k$ the blow-ups of $\Lambda$ in $X$ and in $\bP^5_k$, and by $\nu:X'\to \bP^3_k$ and $\widetilde{\nu}:(\bP^5_k)'\to\bP^3_k$ the morphisms obtained by projecting away from~$\Lambda$.

\begin{prop}
\label{propprojectline}
There exists a smooth projective geometrically connected curve $\Delta\subset \bP^3_k$ of genus $2$ such that $\nu$ can be identified with the blow-up of $\Delta$ in $\bP^3_k$.
\end{prop}

\begin{proof}
Write $\Lambda=\{X_0=\dots=X_3=0\}$ in appropriate homogeneous coordinates $X_0,\dots,X_5$ of $\bP^5_k$. The morphism $\widetilde{\nu}:(\bP^5_k)'\to\bP^3_k$ realizes $(\bP^5_k)'$ as the projectivization (in Grothendieck's sense) of the vector bundle $\mathcal{E}=\sO_{\bP^3_k}\oplus \sO_{\bP^3_k}\oplus\sO_{\bP^3_k}(1)$ on $\bP^3_k$. In a natural way, we use homogeneous coordinates $X_0,\dots,X_3$ on $\bP^3_k$, and we let $X_4, X_5$ (resp.\ $X_6$) denote the global sections of $\mathcal{E}$ (resp. of $\mathcal{E}(-1)$) on $\bP^3_k$ corresponding to the direct sum decomposition of $\mathcal{E}$.
If $X$ has equations $\{L_1X_4+L_2X_5+Q=0\}$ and $\{L'_1X_4+L'_2X_5+Q'=0\}$ in $\bP^5_k$, with $L_1,L_2,L_1',L_2'$ (resp.\ $Q,Q'$) linear (resp. quadratic) in $X_0,\dots,X_3$, then $X'$ has equations 
$\{L_1X_4+L_2X_5+QX_6=0\}$ and $\{L'_1X_4+L'_2X_5+Q'X_6=0\}$ in~$(\bP^5_k)'$.

The fibers of $\nu: X'\to\bP^3_k$ are defined by two linear equations in a $2$-dimensional projective space, hence are isomorphic to $\bP^0$, to $\bP^1$ or to~$\bP^2$. 
The determinantal subscheme $\Delta\subset\bP^3_k$ defined by the vanishing of the maximal minors of the matrix
\begin{equation}
\label{matrix}
\begin{pmatrix}
L_1& L_2 &Q \\
L_1'& L_2' &Q'
\end{pmatrix}
\end{equation}
endows the subset of $\bP^3_k$ over which the fibers of $\nu$ are positive-dimensional with a schematic structure. We claim that $\nu|_{\nu^{-1}(\Delta)}:\nu^{-1}(\Delta)\to\Delta$ is a flat family of lines in the projective bundle $\widetilde{\nu}|_{\widetilde{\nu}^{-1}(\Delta)}:\widetilde{\nu}^{-1}(\Delta)\to\Delta$. To see it, we work on an affine open subset of $\Delta$ with coordinate ring $R$. One has to show that the cokernel $M$ of the linear map $R^2\to R^3$ given by the transpose of the matrix (\ref{matrix}) is free of rank~$2$. This follows from \cite[Proposition~20.8]{Eisenbud} since $\Fitt_1(M)=0$ by the definition of~$\Delta$ and $\Fitt_2(M)=R$ by Lemma \ref{lemdroites} (i).

Let us show that $\Delta$ is smooth of the expected dimension (equal to $1$).
To this end, we fix $x\in \Delta(\ok)$
such that $T_x\Delta_{\ok}$ has dimension $\geq 2$ and derive a contradiction.

We first assume that $\nu^{-1}(x)$ is the line $\{X_6=0\}\subset\widetilde{\nu}^{-1}(x)$.
Then the linear forms $L_1,L_2,L_1'$ and $L_2'$ vanish at $x$.
As $T_x\Delta_{\ok}$ has dimension $\geq 2$, the differentials at~$x$ of the cubic forms $L_1Q'-L_1'Q$ and $L_2Q'-L_2'Q$ are linearly dependent.
After replacing $(L_1,L_1')$ and $(L_2,L_2')$ with suitable $\ok$-linear combinations of  $(L_1,L_1')$ and $(L_2,L_2')$ (which is possible by a change of coordinates), we may therefore assume that the cubic $\{L_1Q'-L_1'Q=0\}$ is singular at $x$.
Since $Q$ and~$Q'$ do not both vanish at~$x$ (otherwise $X_{\ok}$ would contain the plane $\widetilde{\mu}(\widetilde{\nu}^{-1}(x))$, contradicting Lemma~\ref{lemdroites}~(i) over $\ok$), and since $L_1$ and $L_1'$ vanish at $x$, we deduce that $L_1$ and $L_1'$ are linearly dependent. Consequently, an appropriate $\ok$-linear combination of the degree~$2$ equations defining $X_{\ok}$ is of the form $L_2''X_5+Q''$, where $L_2''$ and $Q''$ are respectively linear and quadratic in $X_0,\dots,X_3$. Since $V:=\{L_2''X_5+Q''=0\}$ is singular at $p:=[0:0:0:0:1:0]$ and $X_{\ok}\subset V$ is a Cartier divisor containing $p$, we see that $X_{\ok}$ is also singular at $p$, which is absurd. 

Thus the line $\nu^{-1}(x)$ is not equal to $\{X_6=0\}\subset\widetilde{\nu}^{-1}(x)$.
In other words,
its image by~$\mu$ is a line $L\subset X$ distinct from $\Lambda$. We note that the open subset of the relative Hilbert scheme of lines of $\widetilde{\nu}:(\bP^5_k)'\to \bP^3_k$ consisting of those lines that are not defined by the equation $\{X_6=0\}$ in a fiber of $\widetilde{\nu}$ is naturally isomorphic to the scheme parametrizing the lines in $\bP^5_k$ that are distinct from $\Lambda$ but intersect~$\Lambda$. Since moreover $\nu|_{\nu^{-1}(\Delta)}:\nu^{-1}(\Delta)\to\Delta$ is a family of lines over $\Delta$, two independent tangent vectors of $\Delta$ at $x$ give rise to two independent tangent vectors of $W(\Lambda)$ at~$L$. As $F$ is smooth of dimension~$2$ at $L$ by Lemma~\ref{lemdroites}~(iii), this contradicts Lemma \ref{lemdeuxdroites}
and finishes the proof that~$\Delta$ is a smooth curve.

It then follows from \cite[Theorem A2.60, Example A2.67]{EisenbudSyz} that $\sO_\Delta$ is resolved by the Eagon--Northcott complex
$$0\to\sO_{\bP^3_k}(-4)^{\oplus 2}\to\sO_{\bP^3_k}(-2)\oplus\sO_{\bP^3_k}(-3)^{\oplus 2}\to\sO_{\bP^3_k} \to\sO_\Delta\to 0\rlap{,}$$
where the first arrow is given by the matrix (\ref{matrix}) and the second one by its maximal minors. Using it allows one to compute that $h^0(\Delta,\sO_\Delta)=1$ and $h^1(\Delta,\sO_\Delta)=2$, hence that the smooth projective curve $\Delta$ is geometrically connected of genus~$2$.

To conclude, we denote by $\lambda:Y\to \bP^3_k$ the blow-up of $\Delta$ in $\bP^3_k$. Since $\nu|_{\nu^{-1}(\Delta)}:\nu^{-1}(\Delta)\to\Delta$ is a flat family of lines, hence a smooth morphism, the subscheme $\nu^{-1}(\Delta)\subset X'$ is a smooth divisor. Since $X'$ is smooth, $\nu^{-1}(\Delta)$ is a Cartier divisor in $X'$, and the universal property of a blow-up yields a morphism $\phi:X'\to Y$ such that $\lambda\circ\phi=\nu$. Both $\lambda$ and $\nu$ are birational (the latter because the generic fiber of $\nu$ is a $0$-dimensional projective space), hence so is $\phi$. We deduce that $\phi$ is an isomorphism, as is any birational
morphism between smooth projective varieties with the same Picard number.
\end{proof}

\begin{rem}
\label{remokline}
By Lemma \ref{lemdroites} (iii), a smooth complete intersection of two quadrics in $\bP^5_k$  contains a line over $\ok$, hence is $\ok$-rational by Proposition \ref{propprojectline}.
\end{rem}

\subsection{The intermediate Jacobian}

In Theorem \ref{thintjac}, we compute $\bCH^2_{X/k}$ for threefolds $X$ that are smooth complete intersections of two quadrics. Note that such varieties are $\ok$-rational by Remark \ref{remokline}, so that Theorem \ref{threp} applies to them.

Theorem \ref{thintjac} (iv) will be used in a crucial manner in the proof of Theorem \ref{thrat}.
In characteristic $\neq 2$, it goes back to the work of Wang \cite{Wang}.

In the statement of Theorem~\ref{thintjac}~(ii), we denote by $\Alb_{V/k}^0$
(resp.\ $\Alb_{V/k}^1$) the Albanese variety (resp.\ torsor)
of a smooth proper geometrically connected variety~$V$ over~$k$.
This is the abelian variety over~$k$ (resp.\ torsor under $\Alb_{V/k}^0$)
which underlies the
solution of the universal problem of morphisms from~$V$ to torsors under abelian varieties over~$k$.
We recall that $\Alb_{V/k}^0$ is canonically dual to
the abelian variety $(\bPic^0_{V/k})^\red$ and that the formation
of $\Alb_{V/k}^0$ is compatible with arbitrary
extensions of scalars
(see \cite[Exp.~236, Théorème~3.3~(iii)]{FGA}, in which the geometric fibers of $X\to S$ should be assumed to
be connected).
We also recall that $\Alb_{V/k}^0=\bPic_{V/k}^0$ and $\Alb_{V/k}^1=\bPic_{V/k}^1$
if in addition $V$ is a curve.

By a \emph{conic} on~$X$ we mean a $1$\nobreakdash-dimensional closed subscheme of~$X$ which, when viewed as a subscheme
of~$\bP^5_k$, is
the intersection of a quadric and a plane.

\begin{thm}
\label{thintjac}
Let $X\subset\bP^5_k$ be a smooth complete intersection of two quadrics, let~$F$ be its variety of lines, and let $Z\subset X\times F$ be the universal line.
Denote by $\psi^2_X:\CH^2(X_{\ok})\isoto \bCH^2_{X/k}(\ok)$
the isomorphism of Theorem~\ref{threp}~(iv).
Then:
\begin{enumerate}[(i)]
\item The degree map $\deg:\CH^2(X_{\ok})\to\Z$ induces, via $\psi^2_X$, a short exact sequence $0\to(\bCH^2_{X/k})^0\to\bCH^2_{X/k}\xrightarrow{\delta}\Z\to 0$ of $k$-group schemes.
\item The class $[\sO_Z]\in \K_0(X_F)$ induces isomorphisms $F\isoto(\bCH^2_{X/k})^1:=\delta^{-1}(1)$ and $\Alb^0_{F/k}\isoto(\bCH^2_{X/k})^0$.
\item Let $\Xi\subset \CH^2(X_{\ok})$ be the subset of classes represented
by a conic on~$X_{\ok}$.
There exists a unique reduced closed subscheme $D \subset (\bCH^2_{X/k})^2:=\delta^{-1}(2)$ such that $\psi^2_X(\Xi)=D(\ok)$. The scheme~$D$ is a smooth projective geometrically connected curve of genus~$2$
over~$k$.
\item
Via the identifications $\bPic^0_{D/k}=\Alb^0_{D/k}$
and
$\bPic^1_{D/k}=\Alb^1_{D/k}$,
the inclusion $D \subset (\bCH^2_{X/k})^2$
induces
 isomorphisms of principally polarized abelian varieties
$\bPic^0_{D/k} \isoto (\bCH^2_{X/k})^0$
and of torsors
$\bPic^1_{D/k} \isoto (\bCH^2_{X/k})^2$.
\end{enumerate}
\end{thm}

\begin{proof}
The sheaf $\sO_Z$ induces a class $[\sO_Z]\in F_3\GG_0(X_F)=F^2\K_0(X_F)=\SK_0(X_F)$ by \S\ref{coherent} (as $F$ is smooth by Lemma \ref{lemdroites} (iii)). It therefore induces a morphism $F\to \bCH^2_{X/k}$. This morphism factors through $(\bCH^2_{X/k})^1$ because it sends a point $x\in F(\ok)$ to the class in $\bCH^2_{X/k}(\ok)=\CH^2(X_{\ok})$ of the line in $X_{\ok}$ associated with~$x$ (as
$Z$ is flat over $F$), which has degree $1$.

A morphism $a:F \to (\bCH^2_{X/k})^1$ having been constructed, we are now free to extend the scalars from~$k$
to any finite Galois extension of~$k$: indeed, the existence, unicity and smoothness of~$D$ can be tested
over such an extension, and all other conclusions of the theorem can even be tested over~$\ok$.
As~$F$ is smooth, we may therefore, and will, assume that $F(k)\neq\varnothing$
(see \cite[2.2/13]{BLR}).

Let us fix a line $\Lambda \subset X$ defined over~$k$.
Proposition~\ref{propprojectline} yields a diagram $X\xleftarrow{\mu}X'\xrightarrow{\nu}\bP^3_{k}$, where $\mu$ is the blow-up of~$\Lambda$ and $\nu$ is the blow-up of a smooth projective geometrically connected curve $\Delta\subset\bP^3_{k}$ of genus~$2$. Our knowledge of the Chow groups of a blow-up \cite[Proposition~6.7~(e)]{Fulton} shows the existence of a $\Gamma_k$-equivariant short exact sequence $0\to \CH^2(X_{\ok})_{\alg}\to\CH^2(X_{\ok})\xrightarrow{\deg}\Z\to 0$. Assertion (i) now follows from 
(\ref{okpointsbis}) and~(\ref{okopoints}).

Let $f:\Alb^0_{F/k}\to\Alb^0_{(\bCH^2_{X/k})^1/k}=(\bCH^2_{X/k})^0$
denote the morphism between Albanese varieties induced by $a:F\to(\bCH^2_{X/k})^1$.
The morphism $b:\Delta\to F$ associating with $x\in \Delta$ the line $\mu(\nu^{-1}(x))$ induces a morphism $g:\bPic_{\Delta/k}^0\to \Alb^0_{F/k}$ between Albanese varieties.

The composition $f\circ g:\bPic_{\Delta/k}^0\to(\bCH^2_{X/k})^0$
is an isomorphism of principally polarized abelian varieties.  Indeed,
it coincides at the level of $\ok$-points with the principally polarized isomorphism 
$\bPic_{\Delta/k}^0\isoto(\bCH^2_{X/k})^0$
obtained by applying Proposition \ref{calculblowup} to $\mu$ and $\nu$, hence is equal to it.
It follows that the kernel of $g$ is trivial and that $\Alb^0_{F/k}$ has dimension $\geq 2$, hence that the first Betti number of $F_{\ok}$ is $\geq 4$. Lemma \ref{lemdroites} (iii) and the classification of surfaces with Kodaira dimension $0$ (see \cite[Table p.25 and Theorem 6]{BoMu}) now show that $F_{\ok}$ is an abelian surface. We deduce that $g$ is an isomorphism as its kernel is trivial and as $\bPic_{\Delta/k}^0$ and $\Alb^0_{F/k}$ are both abelian surfaces. It follows that $f$ is an isomorphism.

The variety~$F$ is isomorphic to an abelian variety since so is~$F_{\ok}$ and since $F(k)\neq\varnothing$.  Thus~$a$ and~$f$ can be identified;
hence~$a$ is an
isomorphism as well, and~(ii) is proved.

Let $c:\Delta \to (\bCH^2_{X/k})^2$ be defined by $c(x)=a(b(x))+a(\Lambda)$,
where $\Lambda$ denotes the rational point of~$F$ corresponding to~$\Lambda$.
Passing to Albanese torsors yields
a morphism $\bPic^1_{\Delta/k} \to  (\bCH^2_{X/k})^2$
which is an isomorphism since the underlying morphism of abelian varieties
is the (principally polarized) isomorphism $f\circ g$.
It follows, in particular, that~$c$ is a closed immersion.

Let us set $D=c(\Delta)$
and check
that $D(\ok)=\psi^2_X(\Xi)$, where $\Xi \subseteq \CH^2(X_{\ok})$
is the subset appearing in~(iii).  The theorem will then be proved.

We take up the notation $\widetilde{\mu}$, $\widetilde{\nu}$ of~\textsection\ref{subsecproj}
and the notation introduced in the proof of Proposition~\ref{propprojectline}.
The subvariety $\widetilde{\mu}^{-1}(X) \subset (\bP^5_k)'$ is given by the system of equations
$\{(L_1X_4+L_2X_5+QX_6)X_6=0\}$ and $\{(L'_1X_4+L'_2X_5+Q'X_6)X_6=0\}$.
If $X_0,\dots,X_3$ are the homogeneous coordinates of $x \in \Delta(\ok)$,
this system of equations defines
 a conic
in the plane $\widetilde{\nu}^{-1}(x)$
since the matrix~\eqref{matrix} has rank~$\leq 1$.
Thus, for any $x \in \Delta(\ok)$,
we have exhibited a (singular) conic on~$X_{\ok}$ whose class in $\CH^2(X_{\ok})$
is visibly equal to $[\mu(\nu^{-1}(x))]+[\Lambda_{\ok}]$, that is, to $(\psi^2_X)^{-1}(c(x))$.
Hence $D(\ok) \subseteq \psi^2_X(\Xi)$.

Let us now fix a conic~$C$ on~$X_{\ok}$ and prove that $\psi^2_X([C]) \in D(\ok)$.
There exist a (unique) plane $P \subset \bP^5_{\ok}$ such that $C=X_{\ok}\cap P$
and a (unique) quadric $Y \subset \bP^5_{\ok}$
containing $X_{\ok} \cup P$.
As~$X$ is smooth, the singular locus of~$Y$ is
 disjoint from~$X_{\ok}$ and has dimension~$\leq 0$.
Lemma~\ref{lem:quadriqueschow} below provides a plane $P' \subset Y$
containing~$\Lambda_{\ok}$ such that $[P]=[P'] \in \CH_2(Y)$.
As $X_{\ok}$ is an effective Cartier divisor on~$Y$ and as~$X_{\ok}$
contains
neither~$P$ nor~$P'$
(see Lemma~\ref{lemdroites}~(i)), it follows that
$[C]=[X_{\ok}\cap P]=[X_{\ok}\cap P'] \in \CH^2(X_{\ok})$
\cite[Proposition~2.6(a)]{Fulton}.
Now $P'=\widetilde{\mu}(\widetilde{\nu}^{-1}(x))$ for a unique $x \in \bP^3(\ok)$;
as $X_{\ok} \cap P'$ is a conic in~$P'$, a glance at the equations
of $\widetilde{\mu}^{-1}(X) \cap \widetilde{\nu}^{-1}(x)$ shows that~\eqref{matrix} has
rank~$\leq 1$, hence $x \in \Delta(\ok)$ and $[X_{\ok}\cap P']=[\mu(\nu^{-1}(x))]+[\Lambda_{\ok}]$,
so that $\psi^2_X([C])=c(x)\in D(\ok)$, as desired.
\end{proof}

\begin{lem}
\label{lem:quadriqueschow}
Let $Y \subset \bP^5_{\ok}$ be a quadric whose singular locus has dimension~$\leq 0$.
Let $P \subset Y$ be a plane.
Let $\Lambda \subset Y$ be a line along which~$Y$ is smooth.
There exists a (unique) plane $P' \subset Y$ rationally equivalent to~$P$ on~$Y$
such that $\Lambda \subset P'$.
\end{lem}

\begin{proof}
Let $X_0,\dots,X_5$ denote the homogeneous coordinates of~$\bP^5_{\ok}$.
After a linear change of coordinates, we may assume that~$\Lambda$
is the line $X_0=X_2=X_4=X_5=0$ and that~$Y$ is defined
by the equation $X_0X_1+X_2X_3+X_4X_5=0$ (if~$Y$ is smooth)
or
by the equation $X_0X_1+X_2X_3+X_4^2=0$ (otherwise).
Indeed, letting $H \subset \bP^5_{\ok}$ be a hyperplane
containing~$\Lambda$
and avoiding the singular locus of~$Y$, if~$Y$ is singular (so that~$Y$ is in this case a cone
over the quadric $Y \cap H$, which is smooth),
and setting $H = \bP^5_{\ok}$ otherwise,
one can use~$\Lambda$ to split off two hyperbolic planes from
a quadratic form defining the smooth quadric $Y \cap H$
\cite[Proposition~7.13, Lemma~7.12]{Kniga}; the remaining regular
quadratic form of dimension~$1$ or~$2$ has the desired shape
since the ground field is algebraically closed.

If~$Y$ is smooth,
the intersection of~$Y$ with the linear subspace $\{X_0=X_2=0\}$
is the union of two planes containing~$\Lambda$.
According to \emph{op.\ cit.}, Proposition~68.2,
one of them is rationally equivalent to~$P$ on~$Y$.

If~$Y$ is singular, let~$P'$ be the plane that contains~$\Lambda$ and the singular point of~$Y$.
The smooth quadric $Y \cap H$ cannot contain a plane
(\emph{op.\ cit.}, Lemma~8.10), therefore $P \cap H$ is a line.
The two lines $P \cap H$ and~$\Lambda$ are rationally equivalent on~$Y \cap H$
(\emph{op.\ cit.}, Proposition~68.2); hence~$P$ and~$P'$,
being the cones over $P \cap H$ and over~$\Lambda$
inside~$Y$,
are rationally equivalent on~$Y$ \cite[Example~2.6.2]{Fulton}.
\end{proof}

\subsection{Rationality}

We can now prove a necessary and sufficient criterion for the $k$-rationality of three-dimensional smooth complete intersections of two quadrics.

\begin{thm}
\label{thrat}
Let $X\subset\bP^5_k$ be a smooth complete intersection of two quadrics. Then $X$ is $k$-rational if and only if it contains a line defined over $k$.
\end{thm}

\begin{proof}
If $X$ contains a line $\Lambda$, then projecting from $\Lambda$ induces a birational map $X\dashrightarrow\bP^3_k$ (see the more precise Proposition \ref{propprojectline}).

Assume, conversely, that $X$ is $k$-rational. Let $D$ be as in Theorem~\ref{thintjac}~(iii).
Let $\psi:(\bCH^2_{X/k})^0\isoto\bPic^0_{D/k}$ be the inverse of the isomorphism
of Theorem~\ref{thintjac}~(iv).
 Using Theorem \ref{thintjac} (ii), we identify the variety of lines $F$ of $X$ with the torsor $(\bCH^2_{X/k})^1$ under $(\bCH^2_{X/k})^0$. Theorem~\ref{listobstructions}~(iii) shows the existence of $d\in\Z$ such that $\psi_*[F]=[\bPic^d_{D/k}]\in H^1(k,\bPic^0_{D/k})$ and Theorem~\ref{thintjac}~(iv) yields the identity $[\bPic^1_{D/k}]=2\hspace{.15em}\psi_*[F]\in H^1(k,\bPic^0_{D/k})$. Combining these two equalities shows that
\begin{equation}
\label{d1-d}
\psi_*[F]=[\bPic^d_{D/k}]=[\bPic^{1-d}_{D/k}]\in H^1(k,\bPic^0_{D/k}).
\end{equation}
Noting that $K_D\in\bPic^{2}_{D/k}(k)$ since $D$ has genus $2$, we see that the $\bPic^0_{D/k}$\nobreakdash-torsor $\bPic^{2}_{D/k}$ is trivial. As one of $d$ and $1-d$ is even, it follows from (\ref{d1-d}) that $\psi_*[F]=0\in H^1(k,\bPic^0_{D/k})$. Consequently, $F(k)\neq\varnothing$ and $X$ contains a line defined over $k$.
\end{proof}

\subsection{Unirationality}
\label{parunirat}

We now study the  $k$-unirationality and the separable $k$\nobreakdash-uni\-rationality of smooth complete intersections of two quadrics. 

\begin{thm}
\label{thunirat}
Fix $N\geq 4$, and let $X\subset\bP^N_k$ be a smooth complete intersection of two quadrics. 
The following assertions are equivalent:
\begin{enumerate}[(i)]
\item The variety $X$ is separably $k$-unirational.
\item The variety $X$ is $k$-unirational.
\item One has $X(k)\neq \varnothing$.
\end{enumerate}
\end{thm}

 If $N=4$, Theorem~\ref{thunirat}~(ii)$\Leftrightarrow$(iii) is due to Manin \cite[Theorems 29.4 and~30.1]{Manin} and Knecht \cite[Theorem~2.1]{Knecht}.
Over infinite perfect fields of characteristic not~$2$, Theorem~\ref{thunirat} can be found in \cite[Remark 3.28.3]{ctssd1}.
 The counterpart of Theorem~\ref{thunirat}~(ii)$\Leftrightarrow$(iii) for cubic hypersurfaces is due to Koll\'ar \cite[Theorem~1.1]{Kollarcubic}, and we verify, in Theorem \ref{thuniratcub} below, that the equivalence with~(i)
holds for cubic hypersurfaces as well.

We start with a few lemmas. The first one is an analogue for separable $k$\nobreakdash-uni\-rationality of a statement proved by Koll\'ar \cite[Lemma 2.3]{Kollarcubic} for $k$-unirationality.

\begin{lem}
\label{lemkol}
Let $X$ be an integral variety over $k$. The following are equivalent:
\begin{enumerate}[(i)]
\item The variety $X$ is separably $k$-unirational.
\item The variety $X_{k(t)}$ is separably $k(t)$-unirational.
\item There exists a separable dominant rational map $\bA^m_k\dashrightarrow X$ for some $m\geq 0$.
\end{enumerate}
\end{lem}

\begin{proof}
 The implications (i)$\Rightarrow$(ii)$\Rightarrow$(iii) are immediate.
 To prove (iii)$\Rightarrow$(i), we adapt \cite[Lemma 2.3]{Kollarcubic}. We argue by induction on~$m$. Let $\phi:\bA^m_k\dashrightarrow X$ be a  separable dominant rational map with $m>\dim(X)$, and let $U\subset \bA^m_k$ be a dense open subset such that $\phi|_U$ is a smooth morphism. If $k$ is infinite, set $k'=k$. If $k$ is finite, choose a prime number $\ell$ invertible in~$k$ and a $\Zl$\nobreakdash-extension $k'$ of $k$. As $k'$ is infinite, one can choose a point $u=(u_1,\dots,u_m)\in U(k')$. After renumbering the $u_i$, we may assume that $k(u_m)\subseteq k(u_1)$, so that there exists $P_1 \in k[X_1]$ with $u_m=P_1(u_1)$. The ideal $I(u) \subset k[X_1,\dots,X_m]$
of $u$ viewed as a closed point in~$\bA^m_k$ is then generated by $X_m-P_1(X_1)$ and by elements of $k[X_1,\dots,X_{m-1}]$. It is thus generated by polynomials of the form $X_m-P(X_1,\dots,X_{m-1})$. Since~$u$ is \'etale over~$k$, the zero locus $Z\subset\bA^m_k$ of one of these polynomials intersects the fiber of $\phi|_U$ through $u$ transversally at $u$. Then $\phi|_Z:Z\dashrightarrow X$ is smooth at $u$, hence it is dominant and separable. As $Z\simeq\bA^{m-1}_k$, the induction hypothesis applies.
\end{proof}

\begin{lem}
\label{odp}
Let~$X$ be either a smooth complete intersection of two quadrics in~$\bP^4_{\ok}$ or a smooth
cubic surface in~$\bP^3_{\ok}$.  In the cubic case, assume that the characteristic of~$k$ is not~$2$.
Then there exists a hyperplane section of~$X$ that contains an ordinary double point.
\end{lem}

\begin{proof}
Set $N=4$ if~$X$ is an intersection of two quadrics in~$\bP^4_{\ok}$ and $N=3$ if~$X$ is a cubic surface
in~$\bP^3_{\ok}$.
By \cite[Theorem 24.4]{Manin}, the variety $X$ is isomorphic to a blow-up of~$\bP^2_{\ok}$ in $9-N$ points $p_1,\dots,p_{9-N}\in\bP^2(\ok)$, no three of which lie on the same line.
As the embedding $X\subset\bP^N_{\ok}$ is given by the anticanonical bundle of~$X$, the hyperplane sections of $X$ correspond bijectively to the cubic curves in~$\bP^2_{\ok}$ through the $p_i$. When $N=4$,
the hyperplane section associated with the union of the lines $(p_1p_2)$
and $(p_3p_4)$ and of a general line through $p_5$ has the required
property.
When $N=5$ and the characteristic of~$k$ is not~$2$,
the hyperplane section associated with the union of the unique conic through $p_1,\dots,p_5$
and of a general line through $p_6$ has the required
property.
(If~$k$ had characteristic~$2$, it could happen that a general line through~$p_6$ is tangent to this conic.)
\end{proof}

\begin{lem}
\label{lemdual}
Fix $N\geq 4$, let $X\subset\bP^N_{\ok}$ be a smooth complete intersection of two quadrics, and let $x\in X(\ok)$. Then, if $H\subset \bP^N_{\ok}$ is a general hyperplane containing~$x$, the variety $X\cap H$ is smooth of dimension $N-3$.
\end{lem}

\begin{proof}
Suppose for contradiction that the conclusion of the lemma does not hold. It follows that for every hyperplane $H\subset\bP^N_{\ok}$ containing~$x$, there exists a quadric $Q\subset \bP^N_{\ok}$ in the pencil defining $X$ such that $Q\cap H$ is not smooth of dimension $N-2$ along some point of~$X$. Consequently, the projective variety $W\subset \bP (H^0(\bP^N_{\ok},\sO(2)))\times\bP (H^0(\bP^N_{\ok},\sO(1)))$ parametrizing such pairs $(Q,H)$ is at least $(N-1)$\nobreakdash-dimensional.

Let $Q_0\subset\bP^N_{\ok}$ be any singular quadric in the pencil defining
$X$.  (Such a $Q_0$ exists since the locus of quadrics with a unique
singular point has codimension~$1$ in the projective space of quadrics
in~$\bP^N_{\ok}$, so that its closure contains an ample divisor.)
Note that $x\in X\subset Q_0$. As $X$ is smooth, the singular locus of $Q_0$ is zero-dimensional, hence $Q_0$ is a cone over a smooth quadric $Q'_0$, and $X$ does not contain the vertex of this cone. The projective dual $(Q_0)^{\vee}$ of $Q_0$ can be naturally identified with the projective dual of $Q_0'$. By Lemma \ref{quaddual} below, the variety $(Q_0)^{\vee}$ has dimension $N-2$ and is either a smooth quadric or a linear space whose dual is a line that meets~$Q_0$ in its vertex and nowhere else. In both cases, the subset of $(Q_0)^{\vee}$ consisting of the hyperplanes that contain $x$ is a non-trivial hyperplane section of~$(Q_0)^{\vee}$, hence is $(N-3)$-dimensional. It follows that the variety parametrizing the hyperplanes~$H\subset\bP^N_{\ok}$ with $(Q_0,H)\in W$ is at most $(N-3)$-dimensional.

Let $\pi: W\to\bP^1_{\ok}\subset \bP (H^0(\bP^N_{\ok},\sO(2)))$ be the projection of $W$ to the pencil of quadrics defining $X$.  The fibers of $\pi$ over singular quadrics have dimension $\leq~N-3$, and there exist such fibers. As $\dim(W)\geq N-1$, some fiber of $\pi$ over a smooth quadric must have dimension $\geq N-1$. In other words, 
 there exists a smooth quadric $Q\subset \bP^N_{\ok}$ in the pencil defining $X$ such that $(Q,H)\in W$ for all hyperplanes $H\subset\bP^N_{\ok}$ containing~$x$. Consequently, the projective dual of $Q$ is the hyperplane dual to the point $x$. This contradicts Lemma \ref{quaddual} below since $x\in X\subset Q$.
\end{proof}

\begin{lem}
\label{quaddual}
Fix $N\geq 1$ and let $Q\subset \bP^N_{\ok}$ be a smooth quadric. Then the projective dual of $Q$ is a smooth quadric if $k$ has characteristic $\neq 2$ or if $N$ is odd; otherwise it is a hyperplane whose dual is a point not contained in $Q$.
\end{lem}

\begin{proof}
In appropriate homogeneous coordinates, the quadric $Q$ is defined by the equation $\{X_0X_1+\dots+X_{N-1}X_N=0\}$ if $N$ is odd, and by the equation $\{X_0^2+X_1X_2+\dots+X_{N-1}X_N=0\}$ if $N$ is even (see \cite[XII, Proposition~1.2]{SGA72}). The lemma follows by a direct computation.
\end{proof}

We finally reach the goal of this subsection.

\begin{proof}[Proof of Theorem \ref{thunirat}]
That (i) implies (ii) is clear. Since projective $k$-unirational varieties always have $k$-points, (ii) implies (iii). It remains to show that (iii) implies~(i). In view of Lemma \ref{lemkol}, we may replace $k$ with $k(t)$ to prove this implication. We thus assume from now on that $k$ is infinite.

We argue by induction on $N\geq 4$. Suppose first that $N=4$. Since $k$ is infinite, \cite[Theorems 29.4 and~30.1~(i)]{Manin} shows that $X$ is $k$-unirational. (The hypothesis that $k$ is perfect is not used in the proof of \cite[Theorem~30.1~(i)]{Manin} for degree~$4$ del Pezzo surfaces. One only needs to notice that the variety of lines in a degree~$4$ del Pezzo surface over $k$ is \'etale over $k$.) In particular, again because $k$ is infinite, we see that $X(k)$ is Zariski dense in $X$. 
Choose $x\in X(k)$ general. Let $\nu:\wX\to X$ be the blow-up of $X$ at $x$, with exceptional divisor $E\subset \wX$.
By \cite[Theorems~24.4 and 24.5]{Manin} applied to $X_{\ok}$ and $\wX_{\ok}$, the linear system $|{-}K_{\wX}|$ yields an embedding of~$\wX$ in $\bP^3_k$ as a cubic surface, the image of $E$ being a line. Let $\pi:\wX\to\bP^1_{k}$ be the conic bundle obtained by projecting $\wX$ away from the line $E$.  Choose $y\in\bP^1(k)$ general. The scheme $\pi^{-1}(y)\cap E$ can be identified with the exceptional divisor of the blow-up at $x$ of a hyperplane section of $X$ that contains~$x$ and is singular at~$x$. In view of Lemma \ref{odp}, since $x$ and $y$ have been chosen general, the singularity at $x$ of this hyperplane section is an ordinary double point
(see \cite[XVII, (4.1)]{SGA72}).
 We deduce that $\pi^{-1}(y)\cap E$ is smooth over $k$, hence that the morphism $\pi|_E:E\to\bP^1_k$ is separable. The base change $Y:=\wX\times_{\bP^1_k} E\to E$ of $\pi$ by $\pi|_E$ is a conic bundle over~$E$ with a section, hence is birational to $\bP^1_k\times E=\bP^1_k\times\bP^1_k$. The projection $Y\to\wX$ is dominant and separable because so is $\pi|_E$. The variety~$\wX$ is thus separably $k$-unirational. This concludes the proof since $X$ is birational to $\wX$.

 If $N\geq 5$, choose $x\in X(k)$, and consider the space $B$ of hyperplanes in $\bP^N_k$ that contain~$x$. Let
$Z = \{(w,b) \in X\times B\mkern1mu;\mkern1mu w \in b\}$
and let $p:Z\to B$ and $q:Z\to X$ be the natural projections.
The generic fiber of $p$ is a smooth complete intersection of two quadrics in $\bP^{N-1}_{k(B)}$ by Lemma \ref{lemdual}, and has a $k(B)$-point induced by $x$, hence is separably $k(B)$-unirational by the induction hypothesis. Since $B$ is a projective space, it follows that $Z$ is separably $k$-unirational. As the generic fiber of the dominant map $q$ is a projective space, hence is smooth, Lemma \ref{lemkol} shows that $X$ is also separably $k$-unirational. 
\end{proof}

The strategy of the proof of Theorem \ref{thunirat} can also be applied to smooth cubic hypersurfaces, as we now briefly explain.  Theorem~\ref{thuniratcub} generalises a theorem of Kollár, who proved
the equivalence between (ii) and (iii) in \cite[Theorem~1.1]{Kollarcubic}.

\begin{thm}
\label{thuniratcub}
Fix $N\geq 3$, and let $X\subset\bP^N_k$ be a smooth cubic hypersurface. 
The following assertions are equivalent:
\begin{enumerate}[(i)]
\item The variety $X$ is separably $k$-unirational.
\item The variety $X$ is $k$-unirational.
\item One has $X(k)\neq \varnothing$.
\end{enumerate}
\end{thm}

The proof we give below differs from the one that appears
in the published version of this article (which was invalid when~$k$ has characteristic~$2$). It takes into account the erratum published as~\cite{erratum}.

\begin{proof}
The implication (i)$\Rightarrow$(ii) is trivial, and the equivalence (ii)$\Leftrightarrow$(iii) is \cite[Theorem~1.1]{Kollarcubic}. We prove (ii)$\Rightarrow$(i). By Lemma~\ref{lemkol}, we may assume that~$k$ is infinite. By (ii), one can thus choose $(x, y)\in (X \times X)(k)$ general. We argue by induction on $N\geq 3$. 
The induction step is identical to the one in the proof of Theorem~\ref{thunirat}, replacing Lemma~\ref{lemdual} with Bertini's theorem, which can be applied as $x$ is general.

We now suppose that $N=3$. Let $C_x$ (resp.\ $C_y$) be the hyperplane section of~$X$ that is singular at $x$ (resp.\ at $y$). By \cite[First unirationality construction of~\S 2, Proposition 3.1, Lemma 3.2]{Kollarcubic}, 
the curves $C_x$ and $C_y$ are $k$\nobreakdash-rational, and the third intersection point map $\phi:C_x\times C_y\dashrightarrow X$ is dominant.

Let us prove that~$\phi$ is separable. We may assume that~$k$ is algebraically closed.
Let $(x',y') \in (C_x \times C_y)(k)$
be general.
Assume for contradiction that $\phi$ is not \'etale at~$(x',y')$.
  This exactly means that the lines~$T_{y'}C_y$ and~$T_{x'}C_x$ meet at some point $z$.
  As $(x,y)$ is general, one has $y\notin T_xX$, so the planes $T_xX$ and~$T_yX$ intersect along a line $D$ not containing $y$.  In addition, as $x'$ is general, one has~$x'\notin D$.  It follows that $z=D\cap T_{x'}C_x$ is independent of $y'$. The point $z$ is therefore characterized by the property that it lies on all the tangents to $C_y$ at smooth points.  By symmetry, it also lies on all the tangents to $C_x$ at smooth points, and hence it is also independent of the choice of~$x$ and of~$y$.  

We have now shown that the point $z$ belongs to the tangent hyperplane to $X$ at a general point. The projective dual $\check{X} \subset \check{\bP}^3_k$ of $X$ is therefore contained in the projective dual of~$z$,
which is a hyperplane in~$\check{\bP}^3_k$.

It follows that $X$ is not ordinary in the sense of \mbox{\cite[I-6]{Kleiman}}, and we deduce from \cite[Theorem~17~(iii)$\Rightarrow$(i)]{Kleiman} that no hyperplane section of $X$ admits a nondegenerate double point.  In particular, no hyperplane section of $X$  consists of a line and a conic intersecting transversally.  It therefore follows from \cite[Theorem 1.1 (iii)$\Rightarrow$(i)]{Homma} that $k$ has characteristic~$2$ and that~$X$ is projectively equivalent to the Fermat cubic surface $F \subset \bP^3_k$,
defined by $a^3+b^3+c^3+d^3=0$.  As the tangent plane to~$F$ at $[a:b:c:d]\in F(k)$ has coordinates $[a^2:b^2:c^2:d^2]$ in~$\check{\bP}^3_k$,  the projective dual $\check{F}$ of~$F$ is the Fermat cubic surface in~$\check{\bP}^3_k$.  In particular, 
it is not contained in a hyperplane of~$\check{\bP}^3_k$, and hence neither is~$\check{X}$. This is the required contradiction.
\end{proof}

\subsection{Examples over fields of Laurent series} 

Here are examples of smooth complete intersections of two quadrics in $\bP^5_{\kappa(\!(t)\!)}$ which are not $\kappa(\!(t)\!)$\nobreakdash-rational.
The equations we use in characteristic~$2$ are borrowed from Bhosle \cite{Bhosle}.

\begin{thm}
\label{thmC((t))}
Let $\kappa$ be an algebraically closed field. 
If the characteristic of~$\kappa$ is $\neq 2$, let $a_0,\dots,a_5\in\kappa$ be pairwise distinct elements, and consider the smooth projective variety $X\subset\bP^5_{\kappa(\!(t)\!)}$ with equations
\begin{align*}
\left\{\mkern5mu
\begin{aligned}
 t X_0^2 +tX_1^2+X_2^2+\dots +X_5^2 &= 0\\
ta_0X_0^2 +ta_1 X_1^2+a_2X_2^2+ \dots + a_5 X_5^2 &= 0.
\end{aligned}
\right.
\end{align*}
If $\kappa$ has characteristic $2$, let $a,b,c\in\kappa$ be pairwise distinct elements and consider the smooth projective variety $X\subset\bP^5_{\kappa(\!(t)\!)}$ with equations
\begin{align*}
\left\{\mkern5mu
\begin{aligned}
t X_0X_1 +X_2X_3+X_4X_5& = 0\\
t(X_0^2+aX_0X_1+X_1^2)+(X_2^2+bX_2X_3+X_3^2)&+(X_4^2+cX_4X_5+X_5^2) = 0.
\end{aligned}
\right.
\end{align*}
Then $X$ is separably $\kappa(\!(t)\!)$-unirational, $\kappa(\!(t^{\frac{1}{2}})\!)$-rational, but not $\kappa(\!(t)\!)$-rational.
\end{thm}

\begin{proof}
In view of Theorem~\ref{thrat} and Theorem~\ref{thunirat},
the conclusion of the theorem
is equivalent to the assertion
 that~$X$ contains a point over~$\kappa(\!(t)\!)$,
a line over $\kappa(\!(t^{\frac{1}{2}})\!)$, but no line over $\kappa(\!(t)\!)$,
and this is what we shall now prove.

Let $Y\subset\bP^5_{\kappa}$ be the subvariety with equations $\{\sum_i X_i^2=\sum_i a_iX_i^2=0\}$ if $\kappa$ has characteristic $\neq 2$, with equations 
\begin{align*}
\left\{\mkern5mu
\begin{aligned}
X_0X_1 +X_2X_3+X_4X_5& = 0\\
(X_0^2+aX_0X_1+X_1^2)+(X_2^2+bX_2X_3+X_3^2)&+(X_4^2+cX_4X_5+X_5^2) = 0
\end{aligned}
\right.
\end{align*}
if $\kappa$ has characteristic $2$.
Let $\kappa(\!(t)\!)\subset \kappa(\!(u)\!)$ be the quadratic extension with $u^2=t$ and $\phi:X_{\kappa(\!(u)\!)}
\myxrightarrow{\sim} Y_{\kappa(\!(u)\!)}$ the isomorphism $(X_0,\dots,X_5)\mapsto (uX_0,uX_1,X_2,\dots,X_5)$.

As $\kappa$ is algebraically closed,
the variety $X$ has a $\kappa(\!(t)\!)$-point in the subspace $\{X_0=X_1=0\}$.
The variety~$Y$
contains a line (by Lemma~\ref{lemdroites} (iii)), so that
$X_{\kappa(\!(u)\!)}$ contains a line as well.
Let us now assume that~$X$ itself contains a line $L\subset X$, and derive a contradiction.

We denote by $\tau:\bmu_2\times\bP^5_{\kappa}\to\bP^5_{\kappa}$ the action of the $\kappa$-group scheme $\bmu_2$ on~$\bP^5_{\kappa}$, via its non-trivial character on $X_0$ and $X_1$ and trivially on $X_2$, $X_3$, $X_4$ and~$X_5$. The subvariety $Y\subset\bP^5_{\kappa}$ is $\tau$-invariant. Let $\sigma:\bmu_2\times \Spec(\kappa(\!(u)\!))\to \Spec(\kappa(\!(u)\!))$ be the $\bmu_2$-action endowing $\Spec(\kappa(\!(u)\!))$ with its natural structure of $\bmu_2$-torsor over  $\Spec(\kappa(\!(t)\!))$. It extends to an action of $\bmu_2$ on $\Spec(\kappa[[u]])$ for which the closed point is an invariant subscheme. 

Let us regard the line $L':=\phi(L_{\kappa(\!(u)\!)})\subset Y_{\kappa(\!(u)\!)}$ as a $\kappa(\!(u)\!)$-point of the variety of lines~$F$ of $Y$. In view of the equation defining $\phi$ and since $L$ is defined over~$\kappa(\!(t)\!)$, the morphisms $\bmu_2 \times \Spec(\kappa(\!(u)\!))\to F$ given by
the orbits of $L'$ with respect to the actions of~$\bmu_2$ on $F_{\kappa(\!(u)\!)}$ induced by $\tau$ and by  $\sigma$ coincide. It follows that the specialization $L''\subset Y$ of $L'$ with respect to the $u$-adic valuation of~$\kappa(\!(u)\!)$ is $\tau$\nobreakdash-invariant. Since $Y$ does not meet the projectivization of the subspace of $\kappa^6$ where $\bmu_2$ acts via its non-trivial character, the line $L''$ must be contained in the projectivization $\{X_0=X_1=0\}$ of the subspace  where $\bmu_2$ acts trivially. But the intersection of $Y$ with $\{X_0=X_1=0\}$ is an elliptic curve, which contains no line. This is the required contradiction, and the proposition is proved.
\end{proof}

\begin{rems}
\label{remsLuroth}
(i)
We do not know whether the variety~$X$
appearing in Proposition~\ref{thmC((t))} is stably rational over~$\kappa(\!(t)\!)$,
even when $\kappa=\bC$.

(ii)
If $\kappa$ has characteristic $2$, the variety $X$ considered in Proposition~\ref{thmC((t))} is not $\kappa(\!(t)\!)$\nobreakdash-rational, but it becomes rational over the perfect closure $\kappa(\!(t)\!)_{\p}$ of~$\kappa(\!(t)\!)$.  It follows that one cannot prove Proposition~\ref{thmC((t))} by applying Theorem~\ref{thrat} over~$\kappa(\!(t)\!)_{\p}$. A theory of intermediate Jacobians over imperfect fields is therefore crucial for our proof of Theorem~\ref{thmC((t))}. 

(iii)
The variety $X$ appearing in Proposition \ref{thmC((t))} when $\kappa$ has characteristic $2$ is the first example of a smooth projective variety over a field $k$ which is $k_{\p}$-rational, which has a $k$-point, but which is not $k$-rational. There are no such examples in dimension $\leq 2$
(see Proposition~\ref{surfacesimparfaites} below).

(iv)
Over fields~$k$ of characteristic $p>2$,
smooth complete intersections of two quadrics $X \subset \bP^5_k$ cannot satisfy
the conditions of Remark~\ref{remsLuroth}~(iii).
Indeed, if~$X$ is $k_{\p}$\nobreakdash-rational,
then the period of the variety of lines on~$X$,
viewed as a torsor under the intermediate Jacobian of~$X$,
must be a power of~$p$, by Theorem~\ref{thrat} and \cite[Proposition~5]{langtate};
on the other hand, it divides~$4$ (see the proof of Theorem~\ref{thrat}),
hence it is equal to~$1$, so that~$X$  contains a line defined over~$k$ and is therefore $k$\nobreakdash-rational.
Adapting to imperfect fields
the work of Kuznetsov and Prokhorov  \cite{KP} on prime Fano threefolds
of genus~$10$ might lead to examples, over fields of characteristic~$3$,
of smooth projective varieties that are $k_{\p}$-rational, have a $k$-point, but are not $k$\nobreakdash-rational.
It remains a completely open problem whether varieties (resp.\ threefolds) satisfying these conditions
exist in all characteristics $p\geq 5$.

(v)
We do not know whether there exists a smooth projective variety over a separably closed field $k$ which is $\ok$-rational but not $k$-rational.
\end{rems}

We include the following proposition, which generalises
\cite[Theorem~1]{coombes}, to justify Remark~\ref{remsLuroth}~(iii).

\begin{prop}[Segre, Manin, Iskovskikh]
\label{surfacesimparfaites}
Let $X$ be a smooth projective $\ok$\nobreakdash-rational surface over $k$. The following assertions are equivalent:
\begin{enumerate}[(i)]
\item $X$ is $k$-rational,
\item $X$ is $k_{\p}$-rational and $X(k)\neq\varnothing$.
\end{enumerate}
If $X$ is minimal, they are also equivalent to
\begin{enumerate}[(i)]
\setcounter{enumi}{2}
\item $K_X^2\geq 5$ and $X(k)\neq\varnothing$.
\end{enumerate}
\end{prop}

\begin{proof}
To prove the proposition, we may assume that~$X$ is minimal,
from which it follows that~$X_{k_{\p}}$ is minimal
(see \cite[Corollary~9.3.7]{poonenbook}).
That (i) implies~(ii) is obvious. 
That (ii) and the minimality of~$X_{k_{\p}}$ imply (iii) results from the birational classification of geometrically
rational surfaces over perfect
fields, due to Segre, Manin and Iskovskikh (see \cite[p.~642]{Iskobirat}).
It remains to prove that (iii) implies~(i).
If $X$ is a del Pezzo surface, this is \cite[Theorem 2.1]{VAdP}. Suppose now that $X$ is not a del Pezzo surface. Then $X$ belongs to the family II described in \cite[Theorem~1]{iskovskikhimperfect}. Since $5\leq K_X^2\leq 8$ by \cite[Theorem~3 (2)--(3)]{iskovskikhimperfect}, one has $K_X^2=8$ by \cite[Theorem~5]{iskovskikhimperfect}. It then follows from \cite[Theorem~3~(2)]{iskovskikhimperfect} that $X$ is either a product of curves of genus $0$, or a projective bundle over a curve of genus $0$. As $X(k)\neq \varnothing$, these curves of genus $0$ are isomorphic to~$\bP^1_k$, and $X$ is $k$-rational.
\end{proof}

\bibliographystyle{myamsalpha}
\bibliography{ch2xk}

\providecommand{\bysame}{\leavevmode\hbox to3em{\hrulefill}\thinspace}
\providecommand{\MR}{\relax\ifhmode\unskip\space\fi MR }
\providecommand{\MRhref}[2]{%
  \href{http://www.ams.org/mathscinet-getitem?mr=#1}{#2}
}
\providecommand{\href}[2]{#2}
\begin{thebibliography}{AKMW02}

\bibitem[ABB14]{abb}
A.~Auel, M.~Bernardara and M.~Bolognesi, \emph{Fibrations in complete
  intersections of quadrics, {C}lifford algebras, derived categories, and
  rationality problems}, J. Math. Pures Appl. (9) \textbf{102} (2014), no.~1,
  249--291.

\bibitem[Abh98]{Abhyankar}
S.~S. Abhyankar, \emph{Resolution of singularities of embedded algebraic
  surfaces}, second ed., Springer Monographs in Mathematics, Springer-Verlag,
  Berlin, 1998.

\bibitem[ACMV17]{ACMV1}
J.~D. Achter, S.~Casalaina-Martin and C.~Vial, \emph{On descending cohomology
  geometrically}, Compos. Math. \textbf{153} (2017), no.~7, 1446--1478.

\bibitem[ACMV19]{ACMVfunctorial}
\bysame, \emph{A functorial approach to regular homomorphisms}, preprint 2019,
  arXiv:1911.09911.

\bibitem[AKMW02]{AKMW}
D.~Abramovich, K.~Karu, K.~Matsuki and J.~W{\l{}}odarczyk, \emph{Torification
  and factorization of birational maps}, J. Amer. Math. Soc. \textbf{15}
  (2002), no.~3, 531--572.

\bibitem[Bea77]{Beauville}
A.~Beauville, \emph{Vari\'et\'es de {P}rym et jacobiennes interm\'ediaires},
  Ann. Sci. \'Ecole Norm. Sup. (4) \textbf{10} (1977), no.~3, 309--391.

\bibitem[Bho90]{Bhosle}
U.~Bhosle, \emph{Pencils of quadrics and hyperelliptic curves in characteristic
  two}, J. reine angew. Math. \textbf{407} (1990), 75--98.

\bibitem[Blo73]{blochformula}
S.~Bloch, \emph{Algebraic {$K$}-theory and algebraic geometry}, Algebraic
  {$K$}-theory, {I}: {H}igher {$K$}-theories ({P}roc. {C}onf., {B}attelle
  {M}emorial {I}nst., {S}eattle {W}ash., 1972), Lecture Notes in Math., vol.
  341, 1973, pp.~259--265.

\bibitem[Blo79]{Bloch}
\bysame, \emph{Torsion algebraic cycles and a theorem of {R}oitman}, Compositio
  Math. \textbf{39} (1979), no.~1, 107--127.

\bibitem[BLR90]{BLR}
S.~Bosch, W.~L\"{u}tkebohmert and M.~Raynaud, \emph{N\'{e}ron models}, Ergeb.
  Math. Grenzgeb. (3), vol.~21, Springer-Verlag, Berlin, 1990.

\bibitem[BM77]{BoMu}
E.~Bombieri and D.~Mumford, \emph{Enriques' classification of surfaces in char.
  {$p$}. {II}}, Complex analysis and algebraic geometry, 1977, pp.~23--42.

\bibitem[BW20]{CGBW}
O.~Benoist and O.~Wittenberg, \emph{The {C}lemens-{G}riffiths method over
  non-closed fields}, Algebr. Geom. \textbf{7} (2020), no.~6, 696--721.

\bibitem[BW25]{erratum}
\bysame, \emph{Erratum to \emph{Intermediate {Jacobians} and rationality over
  arbitrary fields}}, Ann. Sci. {\'E}c. Norm. Sup{\'e}r. (4) \textbf{58}
  (2025), no.~3, 829--830.

\bibitem[CG72]{CG}
C.~Clemens and P.~Griffiths, \emph{The intermediate {J}acobian of the cubic
  threefold}, Ann. of Math. (2) \textbf{95} (1972), 281--356.

\bibitem[CGP15]{CGP}
B.~Conrad, O.~Gabber and G.~Prasad, \emph{Pseudo-reductive groups}, second ed.,
  New Mathematical Monographs, vol.~26, Cambridge University Press, Cambridge,
  2015.

\bibitem[Che60]{Chevalley}
C.~Chevalley, \emph{Une d\'{e}monstration d'un th\'{e}or\`eme sur les groupes
  alg\'{e}briques}, J. Math. Pures Appl. (9) \textbf{39} (1960), 307--317.

\bibitem[CJS09]{CJS}
V.~Cossart, U.~Jannsen and S.~Saito, \emph{Canonical embedded and non-embedded
  resolution of singularities for excellent two-dimensional schemes}, preprint
  2009, arXiv:0905.2191.

\bibitem[Con02]{Conrad}
B.~Conrad, \emph{A modern proof of {C}hevalley's theorem on algebraic groups},
  J. Ramanujan Math. Soc. \textbf{17} (2002), no.~1, 1--18.

\bibitem[Coo88]{coombes}
K.~R. Coombes, \emph{Every rational surface is separably split}, Comment. Math.
  Helv. \textbf{63} (1988), no.~2, 305--311.

\bibitem[CP08]{CPI}
V.~Cossart and O.~Piltant, \emph{Resolution of singularities of threefolds in
  positive characteristic. {I}. {R}eduction to local uniformization on
  {A}rtin-{S}chreier and purely inseparable coverings}, J. Algebra \textbf{320}
  (2008), no.~3, 1051--1082.

\bibitem[CP09]{CPII}
\bysame, \emph{Resolution of singularities of threefolds in positive
  characteristic. {II}}, J. Algebra \textbf{321} (2009), no.~7, 1836--1976.

\bibitem[CR15]{Charu}
A.~Chatzistamatiou and K.~R\"{u}lling, \emph{Vanishing of the higher direct
  images of the structure sheaf}, Compos. Math. \textbf{151} (2015), no.~11,
  2131--2144.

\bibitem[CTS87]{descente2}
J.-L. Colliot-Th\'{e}l\`ene and J.-J. Sansuc, \emph{La descente sur les
  vari\'{e}t\'{e}s rationnelles. {II}}, Duke Math. J. \textbf{54} (1987),
  no.~2, 375--492.

\bibitem[CTSSD87]{ctssd1}
J.-L. Colliot-Th\'{e}l\`ene, J.-J. Sansuc and P.~Swinnerton-Dyer,
  \emph{Intersections of two quadrics and {C}h\^{a}telet surfaces. {I}}, J.
  reine angew. Math. \textbf{373} (1987), 37--107.

\bibitem[Deb96]{Debarre}
O.~Debarre, \emph{Polarisations sur les vari\'{e}t\'{e}s ab\'{e}liennes
  produits}, C. R. Acad. Sci. Paris S\'{e}r. I Math. \textbf{323} (1996),
  no.~6, 631--635.

\bibitem[DG70]{DemaGa}
M.~Demazure and P.~Gabriel, \emph{Groupes alg\'{e}briques. {T}ome {I}:
  {G}\'{e}om\'{e}trie alg\'{e}brique, g\'{e}n\'{e}ralit\'{e}s, groupes
  commutatifs}, Masson \& Cie, \'{E}diteur, Paris; North-Holland Publishing
  Co., Amsterdam, 1970.

\bibitem[dJ96]{dJ}
A.~J. de~Jong, \emph{Smoothness, semi-stability and alterations}, Publ. Math.
  IHES (1996), no.~83, 51--93.

\bibitem[DM98]{DeMa}
O.~Debarre and L.~Manivel, \emph{Sur la vari\'{e}t\'{e} des espaces
  lin\'{e}aires contenus dans une intersection compl\`ete}, Math. Ann.
  \textbf{312} (1998), no.~3, 549--574.

\bibitem[Don80]{Donagi}
R.~Donagi, \emph{Group law on the intersection of two quadrics}, Ann. Scuola
  Norm. Sup. Pisa Cl. Sci. (4) \textbf{7} (1980), no.~2, 217--239.

\bibitem[EGA31]{EGA31}
A.~Grothendieck, \emph{\'{E}l\'{e}ments de g\'{e}om\'{e}trie alg\'{e}brique:
  {III}. \'{E}tude cohomologique des faisceaux coh\'{e}rents, {I}}, Publ. Math.
  IHES (1961), no.~11.

\bibitem[EGA32]{EGA32}
\bysame, \emph{\'{E}l\'{e}ments de g\'{e}om\'{e}trie alg\'{e}brique: {III}.
  \'{E}tude cohomologique des faisceaux coh\'{e}rents, {II}}, Publ. Math. IHES
  (1963), no.~17.

\bibitem[EGA42]{EGA42}
\bysame, \emph{\'{E}l\'ements de g\'eom\'etrie alg\'ebrique: {IV}. \'{E}tude
  locale des sch\'emas et des morphismes de sch\'emas, {II}}, Publ. Math. IHES
  \textbf{24} (1965).

\bibitem[EGA43]{EGA43}
\bysame, \emph{\'{E}l\'ements de g\'eom\'etrie alg\'ebrique: {IV}. \'{E}tude
  locale des sch\'emas et des morphismes de sch\'emas, {III}}, Publ. Math. IHES
  \textbf{28} (1966).

\bibitem[Eis95]{Eisenbud}
D.~Eisenbud, \emph{Commutative algebra}, Graduate Texts in Mathematics, vol.
  150, Springer-Verlag, New York, 1995.

\bibitem[Eis05]{EisenbudSyz}
\bysame, \emph{The geometry of syzygies}, Graduate Texts in Mathematics, vol.
  229, Springer-Verlag, New York, 2005.

\bibitem[EKM08]{Kniga}
R.~Elman, N.~Karpenko and A.~Merkurjev, \emph{The algebraic and geometric
  theory of quadratic forms}, American Mathematical Society Colloquium
  Publications, vol.~56, American Mathematical Society, Providence, RI, 2008.

\bibitem[FGA]{FGA}
A.~Grothendieck, \emph{Fondements de la g\'{e}om\'{e}trie alg\'{e}brique
  [extraits du {S}\'{e}minaire {B}ourbaki, 1957--1962]}, Secr\'{e}tariat
  math\'{e}matique, Paris, 1962.

\bibitem[Ful98]{Fulton}
W.~Fulton, \emph{Intersection theory}, second ed., Ergeb. Math. Grenzgeb. (3),
  vol.~2, Springer-Verlag, Berlin, 1998.

\bibitem[Har77]{Hartshorne}
R.~Hartshorne, \emph{Algebraic geometry}, Graduate Texts in Mathematics,
  vol.~52, Springer-Verlag, New York, 1977.

\bibitem[HM82]{HazeM}
M.~Hazewinkel and C.~F. Martin, \emph{A short elementary proof of
  {G}rothendieck's theorem on algebraic vector bundles over the projective
  line}, J. Pure Appl. Algebra \textbf{25} (1982), no.~2, 207--211.

\bibitem[Hom97]{Homma}
M.~Homma, \emph{A combinatorial characterization of the {F}ermat cubic surface
  in characteristic {$2$}}, Geom. Dedicata \textbf{64} (1997), no.~3, 311--318.

\bibitem[HT19a]{HT1}
B.~Hassett and Y.~Tschinkel, \emph{{$ $}{R}ationality of complete intersections
  of two quadrics over nonclosed fields}, preprint 2019, arXiv:1903.08979, to
  appear in L'Enseign. Math.

\bibitem[HT19b]{HT2}
\bysame, \emph{Cycle class maps and birational invariants}, preprint 2019,
  arXiv:1908.00406, to appear in Comm. Pure Appl. Math.

\bibitem[Isk79]{iskovskikhimperfect}
V.~A. Iskovskih, \emph{Minimal models of rational surfaces over arbitrary
  fields}, Izv. Akad. Nauk SSSR Ser. Mat. \textbf{43} (1979), no.~1, 19--43.

\bibitem[Isk96]{Iskobirat}
V.~A. Iskovskikh, \emph{Factorization of birational mappings of rational
  surfaces from the point of view of {M}ori theory}, Uspekhi Mat. Nauk
  \textbf{51} (1996), no.~4(310), 3--72.

\bibitem[IT14]{illusietemkin}
L.~Illusie and M.~Temkin, \emph{Expos\'{e} {X}: {G}abber's modification theorem
  (log smooth case)}, Ast\'{e}risque, no. 363-364, 2014, Travaux de Gabber sur
  l'uniformisation locale et la cohomologie \'{e}tale des sch\'{e}mas
  quasi-excellents, pp.~167--212.

\bibitem[Jan10]{jannsenweights}
U.~Jannsen, \emph{Weights in arithmetic geometry}, Japan J. Math. \textbf{5}
  (2010), no.~1, 73--102.

\bibitem[Jou70]{RRsansden}
J.-P. Jouanolou, \emph{Riemann-{R}och sans d\'{e}nominateurs}, Invent. math.
  \textbf{11} (1970), 15--26.

\bibitem[Kah18]{KahnMurre}
B.~Kahn, \emph{On the universal regular homomorphism in codimension $2$},
  preprint 2018, to appear in Ann. Inst. Fourier,
  \href{https://doi.org/10.5802/aif.3408}{{\tt
  https://doi.org/10.5802/aif.3408}}.

\bibitem[Kle69]{KleGrass}
S.~L. Kleiman, \emph{Geometry on {G}rassmannians and applications to splitting
  bundles and smoothing cycles}, Publ. Math. IHES (1969), no.~36, 281--297.

\bibitem[Kle86]{Kleiman}
\bysame, \emph{Tangency and duality}, Proceedings of the 1984 {V}ancouver
  conference in algebraic geometry, CMS Conf. Proc., vol.~6, Amer. Math. Soc.,
  Providence, RI, 1986, pp.~163--225.

\bibitem[Kle05]{FGAPic}
\bysame, \emph{The {P}icard scheme}, Fundamental algebraic geometry, Math.
  Surveys Monogr., vol. 123, Amer. Math. Soc., Providence, RI, 2005,
  pp.~235--321.

\bibitem[KM76]{MumKnu}
F.~F. Knudsen and D.~Mumford, \emph{The projectivity of the moduli space of
  stable curves. {I}. {P}reliminaries on ``det'' and ``{D}iv''}, Math. Scand.
  \textbf{39} (1976), no.~1, 19--55.

\bibitem[Kne15]{Knecht}
A.~Knecht, \emph{Degree of unirationality for del {P}ezzo surfaces over finite
  fields}, J. Th\'{e}or. Nombres Bordeaux \textbf{27} (2015), no.~1, 171--182.

\bibitem[Kol96]{Kollarbook}
J.~Koll\'ar, \emph{Rational curves on algebraic varieties}, Ergeb. Math.
  Grenzgeb. (3), vol.~32, Springer-Verlag, Berlin, 1996.

\bibitem[Kol02]{Kollarcubic}
J.~Koll\'{a}r, \emph{Unirationality of cubic hypersurfaces}, J. Inst. Math.
  Jussieu \textbf{1} (2002), no.~3, 467--476.

\bibitem[KP19]{KP}
A.~Kuznetsov and Y.~Prokhorov, \emph{{R}ationality of {F}ano threefolds over
  non-closed fields}, preprint 2019, arXiv:1911.08949, to appear in Amer. J.
  Math.

\bibitem[Lau76]{laumon}
G.~Laumon, \emph{Homologie \'{e}tale}, S\'{e}minaire de g\'{e}om\'{e}trie
  analytique (\'{E}cole {N}orm. {S}up., {P}aris, 1974-75), 1976, pp.~163--188.
  Ast\'{e}risque, No. 36--37.

\bibitem[Lip09]{Lipman}
J.~Lipman, \emph{Notes on derived functors and {G}rothendieck duality},
  Foundations of {G}rothendieck duality for diagrams of schemes, Lecture Notes
  in Math., vol. 1960, Springer, Berlin, 2009, pp.~1--259.

\bibitem[LN07]{LipNee}
J.~Lipman and A.~Neeman, \emph{Quasi-perfect scheme-maps and boundedness of the
  twisted inverse image functor}, Illinois J. Math. \textbf{51} (2007), no.~1,
  209--236.

\bibitem[LT58]{langtate}
S.~Lang and J.~Tate, \emph{Principal homogeneous spaces over abelian
  varieties}, Amer. J. Math. \textbf{80} (1958), 659--684.

\bibitem[Man66]{Maninperfect}
Y.~Manin, \emph{Rational surfaces over perfect fields}, Publ. Math. IHES
  \textbf{30} (1966), 55\nobreakdash--113.

\bibitem[Man86]{Manin}
\bysame, \emph{Cubic forms}, second ed., North-Holland Mathematical Library,
  vol.~4, North-Holland Publishing Co., Amsterdam, 1986, Algebra, geometry,
  arithmetic, Translated from the Russian by M. Hazewinkel.

\bibitem[Mil80]{Milne}
J.~S. Milne, \emph{\'{E}tale cohomology}, Princeton Mathematical Series,
  vol.~33, Princeton University Press, Princeton, N.J., 1980.

\bibitem[Mil86]{MilneJac}
\bysame, \emph{Jacobian varieties}, Arithmetic geometry ({S}torrs, {C}onn.,
  1984), Springer, New York, 1986, pp.~167--212.

\bibitem[Mur64]{MurreContravariant}
J.~P. Murre, \emph{On contravariant functors from the category of pre-schemes
  over a field into the category of abelian groups (with an application to the
  {P}icard functor)}, Publ. Math. IHES (1964), no.~23, 5--43.

\bibitem[Mur73]{Murrecubic}
\bysame, \emph{Reduction of the proof of the non-rationality of a non-singular
  cubic threefold to a result of {M}umford}, Compositio Math. \textbf{27}
  (1973), 63--82.

\bibitem[Mur85]{Murre}
\bysame, \emph{Applications of algebraic {$K$}-theory to the theory of
  algebraic cycles}, Algebraic geometry, {S}itges ({B}arcelona), 1983, Lecture
  Notes in Math., vol. 1124, Springer, Berlin, 1985, pp.~216--261.

\bibitem[Poo17]{poonenbook}
B.~Poonen, \emph{Rational points on varieties}, Graduate Studies in
  Mathematics, vol. 186, American Mathematical Society, Providence, RI, 2017.

\bibitem[Rei72]{Reid}
M.~Reid, \emph{The complete intersection of two or more quadrics}, Ph.D.
  Thesis, Cambridge University, 1972.

\bibitem[Rio14]{RiouGysin}
J.~Riou, \emph{Expos\'{e} {XVI}: {C}lasses de {C}hern, morphismes de {G}ysin,
  puret\'{e} absolue}, Ast\'{e}risque, no. 363-364, 2014, Travaux de Gabber sur
  l'uniformisation locale et la cohomologie \'{e}tale des sch\'{e}mas
  quasi-excellents, pp.~301--349.

\bibitem[Ser01]{Serre}
J.-P. Serre, \emph{$\mathrm{Appendix~to}$ {G}eometric methods for improving the
  upper bounds on the number of rational points on algebraic curves over finite
  fields, $\mathrm{by~K.~Lauter}$}, J. Algebraic Geom. \textbf{10} (2001),
  no.~1, 19--36.

\bibitem[Ser06]{sernesi}
E.~Sernesi, \emph{Deformations of algebraic schemes}, Grundlehren der
  Mathematischen Wissenschaften, vol. 334, Springer-Verlag, Berlin, 2006.

\bibitem[SGA1]{SGA1}
A.~Grothendieck, \emph{Rev\^etements \'etales et groupe fondamental ({SGA} 1)},
  Documents Math\'ematiques, vol.~3, Soci\'et\'e Math\'ematique de France,
  Paris, 2003.

\bibitem[SGA2]{SGA2}
\bysame, \emph{Cohomologie locale des faisceaux coh\'{e}rents et
  th\'{e}or\`emes de {L}efschetz locaux et globaux ({SGA} 2)}, Documents
  Math\'{e}matiques, vol.~4, Soci\'{e}t\'{e} Math\'{e}matique de France, Paris,
  2005.

\bibitem[SGA31]{SGA31r}
\bysame, \emph{Sch\'{e}mas en groupes ({SGA} 3) {I}: {P}ropri\'{e}t\'{e}s
  g\'{e}n\'{e}rales des sch\'{e}mas en groupes}, Documents Math\'{e}matiques,
  vol.~7, Soci\'{e}t\'{e} Math\'{e}matique de France, Paris, 2011.

\bibitem[SGA32]{SGA32}
\bysame, \emph{Sch\'{e}mas en groupes ({SGA} 3) {II}: {G}roupes de type
  multiplicatif, et structure des sch\'{e}mas en groupes g\'{e}n\'{e}raux},
  Lecture Notes in Math., vol. 152, Springer-Verlag, Berlin-New York, 1970.

\bibitem[SGA42]{SGA42}
\bysame, \emph{Th\'{e}orie des topos et cohomologie \'{e}tale des sch\'{e}mas
  ({SGA} 4 {II})}, Lecture Notes in Math., vol. 270, Springer-Verlag,
  Berlin-New York, 1972.

\bibitem[SGA6]{SGA6}
\bysame, \emph{Th\'{e}orie des intersections et th\'{e}or\`eme de
  {R}iemann-{R}och ({SGA} 6)}, Lecture Notes in Math., vol. 225,
  Springer-Verlag, Berlin-New York, 1971.

\bibitem[SGA72]{SGA72}
\bysame, \emph{Groupes de monodromie en g\'{e}om\'{e}trie alg\'{e}brique ({SGA}
  7 {II})}, Lecture Notes in Math., vol. 340, Springer-Verlag, Berlin-New York,
  1973.

\bibitem[Tan18]{Tanaka}
H.~Tanaka, \emph{Behavior of canonical divisors under purely inseparable base
  changes}, J. reine angew. Math. \textbf{744} (2018), 237--264.

\bibitem[Tho93]{Thomason}
R.~W. Thomason, \emph{Les {$K$}-groupes d'un sch\'{e}ma \'{e}clat\'{e} et une
  formule d'intersection exc\'{e}dentaire}, Invent. math. \textbf{112} (1993),
  no.~1, 195--215.

\bibitem[TT90]{TT}
R.~W. Thomason and T.~Trobaugh, \emph{Higher algebraic {$K$}-theory of schemes
  and of derived categories}, The {G}rothendieck {F}estschrift, {V}ol. {III},
  Progr. Math., vol.~88, Birkh\"{a}user Boston, Boston, MA, 1990, pp.~247--435.

\bibitem[VA13]{VAdP}
A.~V\'{a}rilly-Alvarado, \emph{Arithmetic of del {P}ezzo surfaces}, Birational
  geometry, rational curves, and arithmetic, Simons Symp., Springer, Cham,
  2013, pp.~293--319.

\bibitem[Wan18]{Wang}
X.~Wang, \emph{Maximal linear spaces contained in the based loci of pencils of
  quadrics}, Algebr. Geom. \textbf{5} (2018), no.~3, 359--397.

\end{thebibliography}
\end{document}